\documentclass{amsart}
\usepackage{amsfonts,amssymb,amsmath,amsthm,amscd}
\usepackage{url}
\usepackage{enumerate}

\newtheorem{theorem}{Theorem}[section]
\newtheorem{lemma}[theorem]{Lemma}
\newtheorem{proposition}[theorem]{Proposition}
\newtheorem{corollary}[theorem]{Corollary}

\newtheorem{hypothesis}[theorem]{Hypothesis}
\theoremstyle{definition}

\newenvironment{definition}[1][Definition]{\begin{trivlist}
\item[\hskip \labelsep {\bfseries #1}]}{\end{trivlist}}
\newenvironment{remark}[1][Remark]{\begin{trivlist}
\item[\hskip \labelsep {\bfseries #1}]}{\end{trivlist}}

\DeclareFontEncoding{OT2}{}{}
\newcommand{\textcyr}[1]{{\fontencoding{OT2}\fontfamily{wncyr}\fontseries{m}\fontshape{n}\selectfont #1}}
\newcommand{\ilim}[1]{\displaystyle{\lim_{\genfrac{}{}{0pt}{}{\longleftarrow}{\scriptstyle #1}} }\;}

\newcommand{\Sha}{{\mbox{\textcyr{Sh}}}}
\newcommand{\ord}{{\operatorname{ord}}}
\newcommand{\Gal}{{\operatorname{Gal}}}

\newcommand{\unr}{{\operatorname{unr}}}

\newcommand{\coker}{{\operatorname{coker}}}
\newcommand{\Ind}{{\operatorname{Ind}}}
\newcommand{\Res}{{\operatorname{Res}}}
\newcommand{\Ram}{{\operatorname{Ram}}}
\newcommand{\Pic}{{\operatorname{Pic}}}
\newcommand{\Spec}{{\operatorname{Spec~}}}

\newcommand{\rec}{{\operatorname{rec}}}
\newcommand{\ab}{{\operatorname{ab}}}
\newcommand{\nrd}{{\operatorname{nrd}}}
\newcommand{\tors}{{\operatorname{tors}}}
\newcommand{\disc}{{\operatorname{disc}}}
\newcommand{\red}{{\operatorname{red}}}
\newcommand{\res}{{\operatorname{res}}}
\newcommand{\Stab}{{\operatorname{Stab}}}
\newcommand{\fin}{{\operatorname{unr}}}
\newcommand{\sing}{{\operatorname{sing}}}
\newcommand{\Ta}{{\operatorname{Ta}}}
\newcommand{\im}{{\operatorname{im}}}
\newcommand{\Div}{{\operatorname{Div}}}
\newcommand{\Supp}{{\operatorname{Supp}}}
\newcommand{\End}{{\operatorname{End}}}
\newcommand{\Frob}{{\operatorname{Frob}}}
\newcommand{\Hom}{{\operatorname{Hom}}}
\newcommand{\Sel}{{\operatorname{Sel}}}
\newcommand{\Fitt}{{\operatorname{Fitt}}}
\newcommand{\Aut}{{\operatorname{Aut}}}

\newcommand{\Spf}{{\operatorname{Spf}}}

\newcommand{\GL}{{\operatorname{GL_2}}}
\newcommand{\SL}{{\operatorname{SL_2}}}
\newcommand{\PGL}{{\operatorname{PGL_2}}}
\newcommand{\M}{{\operatorname{M_2}}}

\newcommand{\CM}{{\operatorname{CM}}}

\newcommand{\vnew}{{\operatorname{v-new}}}
\newcommand{\triv}{{\operatorname{triv}}}
\newcommand{\N}{\mathfrak{N}}
\newcommand{\ff}{{\bf{f}}}

\author{Jeanine Van Order}
\address{Section de Math\'ematiques\\ Ecole Polytechnique F\'ed\'erale de Lausanne\\ Lausanne 1015, Switzerland}
\email{jeanine.vanorder@epfl.ch}
\thanks{The author acknowledges support from the Swiss National Science Foundation (FNS) grant 200021-125291.}

\begin{document}

\title[On the dihedral main conjectures for Hilbert modular eigenforms]{On the dihedral main conjectures of Iwasawa theory for Hilbert modular eigenforms}

\begin{abstract}
We construct a bipartite Euler system in the sense of Howard for Hilbert modular eigenforms of parallel 
weight two over totally real fields, generalizing works of Bertolini-Darmon, Longo, Nekovar, Pollack-Weston 
and others. The construction has direct applications to Iwasawa main conjectures. For instance, it implies 
in many cases one divisibility of the associated dihedral or anticyclotomic main conjecture, at the same 
time reducing the other divisibility to a certain nonvanishing criterion for the associated $p$-adic $L$-functions. 
It also has applications to cyclotomic main conjectures for Hilbert modular forms over CM fields via the technique 
of Skinner and Urban. \end{abstract} 

\maketitle
\tableofcontents

\section{Introduction} 

Let $F$ be a totally real field of degree $d$, and fix a prime $\mathfrak{p}
\subset \mathcal{O}_F$ with underlying rational prime $p$. Let $\ff
\in \mathcal{S}_2(\N)$ be a cuspidal Hilbert modular eigenform of parallel
weight $2$, level $\N \subset \mathcal{O}_F$, and trival character. Assume 
that $\ff$ is $\mathfrak{p}$-ordinary, in the sense that its $T_{\mathfrak{p}}$-eigenvalue
is a $p$-adic unit with respect to any fixed embedding $\overline{{\bf{Q}}} \rightarrow
\overline{{\bf{Q}}}_p$. Assume as well that $\ord_{\mathfrak{p}}(\N)=1$, with $\ff$ being
either new of level $\N$, or else arising from a newform of level $\N/\mathfrak{p}$.
Let us always view $\ff$ is a $p$-adic modular form via a fixed embedding 
$\overline{{\bf{Q}}} \rightarrow \overline{{\bf{Q}}}_p$, writing $\mathcal{O}_0$
to denote the ${\bf{Z}}_p$-subalgebra of $\overline{\bf{Q}}_p$ generated 
by the Fourier coefficients of $\ff$, $\mathcal{O}$ the integral closure of $\mathcal{O}_0$
in its field of fractions $L$, and $\mathfrak{P}$ the maximal ideal of $\mathcal{O}$. We 
assume for simplicity that $\mathfrak{P}$ is contained in $\mathcal{O}_0$.
Fix a totally imaginary quadratic extension $K$ of $F$, with 
relative discriminant prime to $\N$. The choice of $K$ then determines 
the following factorization of $\N$ in $\mathcal{O}_F$:
\begin{align}\label{fact}\N &= \mathfrak{p} \mathfrak{N}^+ \mathfrak{N}^{-},\end{align}
where $\N^+$ is divisible only by primes that split in $K$, and $\N^{-}$ is divisible
only by primes that remain inert in $K$. Assume that $\N^{-}$ is the squarefree 
product of a number of primes congruent to $d$ mod $2$. In this setting, the root number 
of the Rankin-Selberg $L$-function $L(\ff, K, s)$ at its central value $s =1$ is equal to $1$. 
Moreover, the central value (as well as those of the associated twists by ring class characters)
can be described by the toric integral formula of Waldspurger \cite{Wa}, as generalized for instance by 
Yuan-Zhang-Zhang \cite{YZ^2}. Ultimately, this formula can be used to study the arithmetic behaviour 
of $\ff$ in the dihedral or anticyclotomic ${\bf{Z}}_p^{\delta}$-extension $K_{\mathfrak{p}^{\infty}}$ of $K$, where
$\delta$ denotes the index $[F_{\mathfrak{p}}:{\bf{Q}}_p]$. That is, let $G_{\mathfrak{p}^{\infty}}$ denote the 
Galois group $\Gal(K_{\mathfrak{p}^{\infty}}/K)$, with $\Lambda = \mathcal{O}[[G_{\mathfrak{p}^{\infty}}]]$ the 
associated $\mathcal{O}$-Iwasawa algebra. Using these toric integral formulae, as well as the
class field theoretic description of $G_{\mathfrak{p}^{\infty}}$, there is a natural construction
of the associated $p$-adic $L$-function $\mathcal{L}_{\mathfrak{p}}(\ff, K_{\mathfrak{p}^{\infty}})
\in \Lambda$, as shown in the prequel paper \cite{VO} (following the constructions of Bertolini-Darmon \cite{BD},
\cite{BD2}). In particular, in addition to satisfying the usual interpolation property, this $p$-adic $L$-function is nontrivial 
thanks to the nonvanishing theorem of Cornut and Vatsal \cite[Theorem 1.4]{CV}. The main purpose of the present work is to 
use this construction to prove the one divisibility of the associated dihedral or anticyclotomic main conjecture, as well as to 
outline some applications beyond this. To be more precise, let $\Sel (\ff, K_{\mathfrak{p}^{\infty}})$ denote the 
$\mathfrak{P}^{\infty}$-Selmer group of $\ff$ in $K_{\mathfrak{p}^{\infty}}/K$, with $X(\ff, K_{\mathfrak{p}^{\infty}})$ its Pontryagin dual. 
The Iwasawa main conjecture in this setting predicts that $X(\ff, K_{\mathfrak{p}^{\infty}})$ is a torsion $\Lambda$ module, and 
moreover that there is an equality of principal ideals $(\mathcal{L}_{\mathfrak{p}}(\ff, K_{\mathfrak{p}^{\infty}})) = 
(\operatorname{char}_{\Lambda}(X(\ff, K_{\mathfrak{p}^{\infty}})))$ in $\Lambda$. Here, 
$\operatorname{char}_{\Lambda}(X(\ff, K_{\mathfrak{p}^{\infty}}))$ is the $\Lambda$-characteristic power series of 
$X(\ff, K_{\mathfrak{p}^{\infty}})$, which exists (by the structure theorem of \cite{BB}) as $X(\ff, K_{\mathfrak{p}^{\infty}})$ 
is $\Lambda$-torsion. We show the following results towards this conjecture. Let us first impose the following hypotheses, writing 
$\rho_{\ff}: G_F \longrightarrow \GL(\mathcal{O})$ to denote the $\mathfrak{P}$-adic Galois representation associated to $\ff$ by the 
construction of Carayol \cite{Ca2}, Taylor \cite{Tay} and Wiles \cite{Wi2} (see Theorem \ref{CTW} below). Here, $G_F$ denotes the 
Galois group $\Gal(\overline{{\bf{Q}}}/F)$.

\begin{hypothesis}\label{galrep}~
\begin{itemize}
\item[(i)] The prime $p$ is odd.
\item[(ii)] The prime $\mathfrak{p} \subset \mathcal{O}_F$ is the unique 
prime above $p$ in $K_{\mathfrak{p}^{\infty}}.$
\item[(iii)] The eigenform ${\bf{f}} \in \mathcal{S}_2(\mathfrak{N})$ is 
$\mathfrak{p}$-ordinary.
\item[(iv)] The Galois representation $\rho_{\bf{f}}$ is 
residually irreducible.
\item[(v)] The image of the residual Galois representation 
$\overline{\rho}_{\bf{f}}$ contains $\SL({\bf{F}}_p)$.
\item[(vi)] The degree $d$ is either odd, or else even with the condition that
$\mathfrak{N}^{-} \neq \mathcal{O}_F$. \end{itemize}\end{hypothesis} 
 
\begin{remark} Thanks to Dimitrov \cite[Proposition 0.1]{D}, we have the
following generalizations of relevant results of Serre \cite{Se}
and Ribet \cite{Ri2} here: (i) for all but finitely many rational primes $p$, the 
Galois representation $\rho_{{\bf{f}}}$ 
is residually irreducible (\cite[Proposition 3.1]{D}), and (ii)
for all but finitely many rational primes $p$, there 
exists some power $q= p^a$ of $p$ such that the 
image of the residual Galois representation 
$\overline{\rho}_{{\bf{f}}}$ contains $\SL({\bf{F}}_q)$
(\cite[Proposition 3.8]{Di}). Thus, Hypotheses \ref{galrep} (iv)
and (v) are not prohibitively strong. \end{remark}

\begin{theorem}[Proposition \ref{4.3}, Corollary \ref{dmc}]\label{RESULT} 
Let $\ff\in \mathcal{S}_2(\N)$ be a cupsidal 
Hilbert eigenform as above, with $\N \subset \mathcal{O}_F$
having the factorization $(\ref{fact})$, and with the conditions of 
Hypothesis \ref{galrep}. Assume also that the following standard
hypotheses hold: \\

\begin{itemize}
\item[(A)] The totally real field $F$ is linearly disjoint from the cyclotomic field ${\bf{Q}}(\zeta_p)$,
\item[(B)]  The Galois representation $\rho_{\ff}$ satisfies a certain multiplicity one condition: Hypothesis \ref{freeness}.
\item[(C)]  A variant of Ihara's lemma for Shimura curves holds: Hypothesis \ref{ihara}. \\
\end{itemize} Then, $X(\ff, K_{\mathfrak{p}^{\infty}})$ is $\Lambda$-torsion, and there is an 
inclusion of ideals \begin{align}\label{mainconjecture}
\left(\mathcal{L}_{\mathfrak{p}}(\ff, K_{\mathfrak{p}^{\infty}}) \right) 
&\subseteq \left( \operatorname{char}_{\Lambda}(X(\ff, K_{\mathfrak{p}^{\infty}}))\right)
\text{ in $\Lambda$.}\end{align}\end{theorem}

\begin{remark}[Remark on hypotheses.] We refer the reader to the statements of Hypotheses \ref{freeness}
and \ref{ihara} below for more details. We state them here in this form not only for simplicity of 
exposition, but also because they are at present works in progress by others 
(see for instance \cite{Ch} and \cite{Ch2}). Condition (A) is used to prove a level raising 
at two primes result, see Proposition \ref{gammaeigenform} below. Condition (B) is a standard
hypothesis (cf. \cite[Proposition 6.3]{PW}) that is crucial to our arguments. It is proved 
with our hypotheses on $\rho_{\ff}$ given above granted that $p$ is unramified in $F$
by Cheng in \cite{Ch}. Condition (C) is also treated by Cheng in \cite{Ch2}, assuming
that the level $\N$ is sufficiently large, and that $p >d$. It is likely that these latter two 
technical hypotheses can be loosened. \end{remark}

\begin{remark}[Remark on $\mu$-invariants.]
We can also deduce from $(\ref{mainconjecture})$ one divisibility in the $\mu$-part of the 
main conjecture, following the characterization of the $\mu$-invariant associated
to $\mathcal{L}_{\mathfrak{p}}(\ff, K_{\mathfrak{p}^{\infty}})$ in the author's previous
work \cite{VO} (see \cite[Theorem 4.10]{VO}). Roughly, following the approach of 
Vatsal \cite{Va2}, we find that $\mu(\mathcal{L}_{\mathfrak{p}}({\bf{f}}, K_{\mathfrak{p}^{\infty}})) 
= 2 \nu$, where $\nu = \nu_{\bf{f}}$ is the largest integer such that ${\bf{f}}$ is congruent 
to a constant mod $\mathfrak{P}^{\nu}$. Following the line of argument of Pollack-Weston 
\cite[$\S 2.3$]{PW}, the hypothesis that $\rho_{\bf{f}}$ be residually irreducible should indicate
that $\mu(\mathcal{L}_{\mathfrak{p}}({\bf{f}}, K_{\mathfrak{p}^{\infty}})) =0$, and hence via
$(\ref{mainconjecture})$ that $\mu(\mathcal{L}_{\mathfrak{p}}({\bf{f}}, K_{\mathfrak{p}^{\infty}})) =
\mu( \operatorname{char}_{\Lambda}(X(\ff, K_{\mathfrak{p}^{\infty}})) ) =0$. We hope to take up
a more detailed study of this interesting and subtle issue in a later work, perhaps in the context
of Euler characteristic computations (cf. \cite{VO2}). \end{remark}

The strategy of proof is to generalize the refined Euler system method of Pollack-Weston 
\cite{PW} (following Bertolini-Darmon \cite{BD}) to the setting of totally real fields. 
In doing so, we construct a bipartite Euler system in the sense of Howard 
\cite[Definition 2.3.2]{Ho}. In particular, we obtain from \cite[Theorem 3.2.3]{Ho} the following
criterion for equality in $(\ref{mainconjecture})$, as explained in \cite[$\S$5]{VO}. Let us now
assume for simplicity that $\mathfrak{N}$ is prime to the relative discriminant of $K$ over $F$.
Fix a positive integer $k$. Define a set of admissible primes $\mathfrak{L}_k$ of $F$, each
being inert in $K$, by the condition that for any ideal $\mathfrak{n} \subset \mathcal{O}_F$
in the set $\mathfrak{S}_k$ of squarefree products of primes in $\mathfrak{L}_k$, there exists
a nontrivial eigenform  $\ff^{(\mathfrak{n})}$ of level $\mathfrak{n}\N$ such that the following
congruence on Hecke eigenvalues holds: \begin{align*} \ff^{(\mathfrak{n})} &\equiv \ff \mod 
\mathfrak{P}^k.\end{align*} Let $\mathfrak{S}_k^+ \subset \mathfrak{S}_k$ denote the subset of 
ideals $\mathfrak{n} \in \mathfrak{S}_k$ for which $\omega_{K/F}(\mathfrak{n}\mathfrak{N})= -1$, 
where $\omega_{K/F}$ denotes the quadratic Hecke character associated to $K/F$. 
Equivalently, $\mathfrak{S}_k^+ \subset \mathfrak{S}_k$ denotes the subset of 
ideals $\mathfrak{n} \in \mathfrak{S}_k$ for which the root number of the $L(\ff, K, s)$ is equal
to $+1$. Note that by our hypotheses of $\N$, this set $\mathfrak{S}_k^+$ includes the trivial ideal $\mathfrak{n} =
\mathcal{O}_F$. Let $\mathfrak{S}_k^- \subset \mathfrak{S}_k$ denote the subset of ideals $\mathfrak{n} \in \mathfrak{S}_k$ 
for which $\omega_{K/F}(\mathfrak{n}\mathfrak{N})= +1$, equivalently for which the root number of $L(\ff, K, s)$ is equal to $-1$.
Given an ideal $\mathfrak{n} \in \mathfrak{S}_k^+$, there is an associated $p$-adic $L$-function 
$\mathcal{L}_{\mathfrak{p}}({\bf{f}}^{(\mathfrak{n})}, K_{\mathfrak{p}^{\infty}})$ in $\Lambda$. As explained 
below, $\mathcal{L}_{\mathfrak{p}}({\bf{f}}^{(\mathfrak{n})}, K_{\mathfrak{p}^{\infty}}) = 
\mathcal{L}_{{\bf{f}}^{(\mathfrak{n})}} \mathcal{L}_{{\bf{f}}^{(\mathfrak{n})}}^{*}$, where 
$\mathcal{L}_{{\bf{f}}^{(\mathfrak{n})}} \in \Lambda$ is a completed group ring element 
constructed in a natural way from ${\bf{f}}^{(\mathfrak{n})}$, and 
$\mathcal{L}_{{\bf{f}}^{(\mathfrak{n})}}^{*}$ is the image of 
$\mathcal{L}_{{\bf{f}}^{(\mathfrak{n})}}$ under the involution $\Lambda \rightarrow \Lambda$
sending $\sigma$ to $\sigma^{-1}$ in $G_{\mathfrak{p}^{\infty}}$. Let us write 
$\lambda_{\mathfrak{n}}$ to denote this completed group ring element 
$\mathcal{L}_{{\bf{f}}^{(\mathfrak{n})}}$, which is only well defined up to multiplication
by elements of $G_{\mathfrak{p}^{\infty}}$. Given an ideal $\mathfrak{n} \in \mathfrak{S}_k^-$, 
there is an associated collection of CM points of $\mathfrak{p}$-power conductor on the 
quaternionic Shimura curve $\mathfrak{M}(\mathfrak{N}^{+}, v\mathfrak{n}\mathfrak{N}^{-})$,
where $v$ is a $k$-admissible prime with respect to ${\bf{f}}$, as we explain in $\S\S 7-11$ 
below. As we also explain below, these points can be used to construct classes in the cohomology 
group $H^1(K_{\mathfrak{p}^{\infty}}, T_{ {\bf{f}}, k })$, which we denote here by $\kappa_{\mathfrak{n}}$
(i.e. so that $\kappa_{v\mathfrak{n}} = \zeta(\mathfrak{n})$ in our notations below). We refer the reader to 
the discussion below for more explanation, as well as to $ \S \ref{GalRep}$ for a definition of the mod $\mathfrak{P}^k$
Galois representation $T_{\ff, k}$ associated to ${\bf{f}}$. Anyhow, we construct for each integer 
$k \geq 1$ a pair of families \begin{align}\label{BES} \lbrace \lambda_{\mathfrak{n}}  \in \Lambda/ \mathfrak{P}^k \Lambda  :
\mathfrak{n} \in \mathfrak{S}_k^+ \rbrace  ~~~\text{ and }~~~ \lbrace \kappa_{\mathfrak{n}} \in \widehat{H}^1
(K_{\mathfrak{p}^{\infty}}, T_{\ff, k}) : \mathfrak{n} \in \mathfrak{S}_k^- \rbrace \end{align} which, as $k$ varies, are compatible with 
respect to the inclusion $\mathfrak{S}_{k+1} \subset \mathfrak{S}_k$, as well as with respect to the natural maps $T_{\ff, k+1} \rightarrow 
T_{\ff, k}$ and $\Lambda/\mathfrak{P}^{k+1} \rightarrow \Lambda/ \mathfrak{P}^k$. We show here that these classes satisfy the following 
{\it{first and second explicit reciprocity laws}}: \\

\begin{remark}[The first explicit reciprocity law](Theorem \ref{ERL1}). For any $v\mathfrak{n} \in \mathfrak{S}_k^-$ with $v$ 
a prime, there is an isomorphism of $\Lambda$-modules \begin{align*} \widehat{H}^1_{\sing}(K_{\mathfrak{p}^{\infty}, v}, 
T_{\ff,k}) &\cong \Lambda / \mathfrak{P}^k \Lambda \end{align*} sending $\operatorname{loc}_v(\kappa_{v\mathfrak{n}})$ 
to $\lambda_{\mathfrak{n}}$, where $\operatorname{loc}_v$ denotes the localization map at $v$. \end{remark} ~~\\

\begin{remark}[The second explicit reciprocity law](Theorem \ref{ERL2}). For any $v\mathfrak{n} \in \mathfrak{S}_k^+$
with $v$ a prime, there is an isomorphism of $\Lambda$-modules \begin{align*} \widehat{H}^1_{\unr}(K_{\mathfrak{p}^{\infty}, v}, 
T_{\ff,k}) &\cong \Lambda / \mathfrak{P}^k \Lambda \end{align*} sending $\operatorname{loc}_v(\kappa_{\mathfrak{n}})$ to 
$\lambda_{v\mathfrak{n}}$, where $\operatorname{loc}_v$ denotes the localization map at $v$ \end{remark} ~~\\ Now, since the 
empty product lies in $\mathfrak{S}_k$ for each integer $k \geq 1$, we can construct a distinguished element 
\begin{align*} \begin{cases} \lambda^{\infty} \in \Lambda &\text{if $\omega_{K/F}(\N^-)= -1$} \\
\kappa^{\infty} \in \mathfrak{S}(\ff/K_{\mathfrak{p}^{\infty}}) &\text{if $\omega_{K/F}(\N^-)= +1$}\end{cases}
\end{align*} by taking the inverse limit of $\lambda_1$ or $\kappa_1$ as $k$ varies. Here, $\mathfrak{S}(\ff
/K_{\mathfrak{p}^{\infty}})$ denotes the compactified Selmer group of $\ff$ over $K_{\mathfrak{p}^{\infty}}$.
Note that while the element $\kappa^{\infty}$ has been studied independently by Howard in \cite{Ho1}, it can also be 
recovered directly from the construction given below. Note as well that by the nonvanishing theorems of Cornut and 
Vastal \cite{CV}, neither of these distinguished elements vanishes. Hence, we deduce that 
the pair of families $(\ref{BES})$ defines a nontrivial bipartite Euler system in the sense of Howard \cite[Definition 2.3.2]{Ho}. In 
particular, via Howard's theory of bipartite Euler systems, we obtain the following result. Here, given any eigenform 
$\ff \in \mathcal{S}_2(\N)$ with $\N \subset \mathcal{O}_F$ having the factorization $(\ref{fact})$, we assume Hypothesis 
\ref{galrep} along with the hypotheses (A), (B) and (C) of Theorem \ref{RESULT}.

\begin{theorem}\label{HOWARD} Let $X(\ff, K_{\mathfrak{p}^{\infty}})_{\tors}$ denote the $\Lambda$-torsion
submodule of $X(\ff, K_{\mathfrak{p}^{\infty}})$. 

\begin{itemize}
\item[(i)] We have the following rank formula:
\begin{align*}\operatorname{rank}_{\Lambda}\mathfrak{S}(\ff/ K_{\mathfrak{p}^{\infty}}) = 
\operatorname{rank}_{\Lambda} X(\ff,  K_{\mathfrak{p}^{\infty}}) =
\begin{cases} 0 &\text{if $\omega_{K/F}(\N^-)=-1$} \\
1 &\text{if $\omega_{K/F}(\N^-)=+1$.}\end{cases}\end{align*}

\item[(ii)] For each height one prime $\mathfrak{Q}$ of $\Lambda$,
\begin{align*} &\ord_{\mathfrak{Q}}\left(  \operatorname{char}_{\Lambda} \left( X(\ff, K_{\mathfrak{p}^{\infty}})_{\tors}  \right) \right) \leq \\
& 2 \times \begin{cases} \ord_{\mathfrak{Q}}(\lambda^{\infty})  &\text{if $\omega_{K/F}(\N^-)=-1$} \\
\ord_{\mathfrak{Q}} \left( \operatorname{char}_{\Lambda} \left( \mathfrak{S}(\ff/ K_{\mathfrak{p}^{\infty}})/\Lambda \kappa^{\infty} \right)\right)     
&\text{if $\omega_{K/F}(\N^-)=+1$.}\end{cases}\end{align*}

\item[(iii)] Equality holds in (ii) if the following condition is satisfied: there exists an integer $k_0$ such that for all integers $k \geq k_0$, 
the set \begin{align*} \lbrace  \lambda_{\mathfrak{n}} \in \Lambda / \mathfrak{P}^k \Lambda: \mathfrak{n} \in \mathfrak{S}_k^+  \rbrace 
\end{align*} contains at least one element with nontrivial image in $\Lambda/ (\mathfrak{Q}, \mathfrak{P}^{k_0}).$ In particular, 
equality in (ii) holds if one of the elements $\lambda_{\mathfrak{n}}$ is a unit in $\Lambda$.

\end{itemize}\end{theorem}

\begin{proof} The result follows from the proof of Howard \cite[Theorem 3.2.3]{Ho}, which carries over to 
this setting with minor changes. See also the discussion in \cite{Ho1}. We sketch the deduction for lack of better reference. 
First, note that the general theory of Euler systems over Artinian ring developed in Howard \cite[$\S$2]{Ho} applies to this setting. 
In particular, \cite[Proposition 3.3.1]{Ho} (cf. \cite[Lemma 5.3.13]{MR}) and \cite[Proposition 3.3.3]{Ho}
carry over to this setting. The result of \cite[Lemma 3.3.2]{Ho} is also standard here, see for instance \cite[Theorem 7.1]{HV} , 
using the basic fact that $A(K_{\mathfrak{p}^{\infty}})_{p^{\infty}} \subset A(K_{\mathfrak{p}^{\infty}})_{\tors}$ is finite
for any abelian variety $A$ defined over $K_{\mathfrak{p}^{\infty}}$, in particular for the abelian variety $A_{\ff}$ 
associated to $\ff$ in Proposition \ref{hilbertAV} below. The proof of Howard \cite[Theorem 3.2.3 (c)]{Ho} can then 
be given by the argument of \cite[$\S$ 3.4]{Ho}, with minor modifications, following \cite[$\S$3.3]{Ho1}. 
That is, fix a height one prime ideal $\mathfrak{Q}$ of $\Lambda$. Fix a sequence of specializations $\phi_i: \Lambda \longrightarrow S$, 
in the sense of \cite[Definition 3.2.5]{Ho1}. Suppose that this sequence converges to $\mathfrak{Q}$, following \cite[Definition 3.3.3]{Ho1}. 
Note that such a sequence always exists by \cite[Proposition 3.3.3.]{Ho1}. The argument of 
\cite[$\S 3.4$ p. 21]{Ho} can then be modified by taking tensors $\otimes_{\Lambda} S$ as done in 
\cite[$\S  3.4$]{Ho1} to obtain the analogous result of \cite[Theorem 3.2.3]{Ho} in this setting.\end{proof} Combined, 
Theorems \ref{RESULT} and \ref{HOWARD} imply following criterion for equality in $(\ref{mainconjecture})$.

\begin{corollary}\label{criterion} 
Suppose that for each height one prime ideal $\mathfrak{Q}$ of $\Lambda$,
there exists a positive integer $k_0$ such that for each integer $k \geq k_0$,
the set $\mathfrak{S}_k^+$ contains an ideal $\mathfrak{n}$ for which the image of 
the associated completed group ring element $\lambda_{\mathfrak{n}}$ in the quotient
$\Lambda/(\mathfrak{Q}, \mathfrak{P}^{k_0})$ is not trivial. Then, there is an equality of ideals in 
$(\ref{mainconjecture})$, i.e. the full dihedral (or anticyclotomic) main conjecture of 
Iwasawa theory holds: \begin{align}\label{full} \left(\mathcal{L}_{\mathfrak{p}}(\ff, 
K_{\mathfrak{p}^{\infty}}) \right) &= \left( \operatorname{char}_{\Lambda}(X(\ff, 
K_{\mathfrak{p}^{\infty}}))\right) \text{ in $\Lambda$.}\end{align} In particular, if 
one of the completed group ring elements $\lambda_{\mathfrak{n}}$ is a unit in 
$\Lambda$, then the full main conjecture equality $(\ref{full})$ holds.\end{corollary} Some further remarks 
are in order at this point. The result of Theorem \ref{RESULT} has many antecendents in the literature, among them 
the original work of Bertolini-Darmon \cite{BD}, as well as subsequent generalizations to totally real fields
by Longo (\cite{Lo}, \cite{L30} and \cite{Lo2}), Fouquet \cite{Fo} and Nekovar \cite{Nek}. The main novelty
here is that the remove the restrictive $p$-isolatedness hypotheses found in these works of Bertolini-Darmon
and Longo, following the approach of Pollack-Weston \cite{PW}. This innovation is not merely technical, as 
it allows is to invoke the theory of bipartite Euler systems due to Howard \cite{Ho} to both reduce the other 
divisibility of the main conjecture to a nonvanishing criterion for $p$-adic $L$-functions, as well as to treat both
definite and indefinite cases on the root number simultaneously. Perhaps more intriguingly, Theorem \ref{RESULT}  
above can also be combined with techniques of Skinner-Urban \cite{SU} to give a new proof of the associated cyclotomic 
main conjecture, which previously had only been accessible by the Euler system method of Kato \cite{KK}. Moreover, it seems 
that the techniques of Skinner-Urban \cite{SU} extend to the more general setting of totally real fields (by work in progress of 
\cite{Wan}), in which case the result of Theorem \ref{RESULT} would allow one to deduce the associated cyclotomic main 
conjecture for totally real fields, which at present is not accessible even by the method of Kato's Euler system. 

\begin{remark}[Application to modular abelian varieties.]

We obtain the following consequence for modular 
abelian varieties. Let $A$ be an abelian variety over $F$ of arithmetic 
conductor $\mathfrak{N} \subset \mathcal{O}_F$. Given any
integer $n \geq 1$ and any Galois extension $L$ over $F$ with 
$\mathfrak{P}\mid \mathfrak{p}$ a prime above $\mathfrak{p}$ in $L$, we can
associate to $A$ a residual Selmer group $\Sel_{\mathfrak{P}^n}(A/L),$ 
defined by its inclusion in the exact sequence \begin{align*}0 \longrightarrow 
\Sel_{\mathfrak{P}^n}(A/L)\longrightarrow H^1(L, A[\mathfrak{P}^n]) \longrightarrow
\bigoplus_v H^1(L_v, A[\mathfrak{P}^n])/\im(\mathfrak{K}_v).\end{align*} Here,
the sum runs over all primes $v \subset \mathcal{O}_L,$ and
\begin{align*}\mathfrak{K}_v: A(L_v)/\mathfrak{P}^n A(L_v) \longrightarrow
H^1(L_v, A[\mathfrak{P}^n])\end{align*} denotes the local Kummer map at $v$.
Now, an abelian variety $A/F$ is said to be of {\it{$\GL$-type}} if the
endomorphism algebra $\End(A) \otimes_{\bf{Z}} {\bf{Q}}$ contains a
number field $L$ of degree equal to $\dim(A)$. An abelian variety $A$ defined over $F$ 
of $\GL$-type is said to be {\it{modular}} if there exists a Hilbert modular 
eigenform ${\bf{f}} \in \mathcal{S}_2(\mathfrak{N})$ such
that the Galois representation \begin{align*}\rho_{A, \lambda}: G_F \longrightarrow \GL(\mathcal{O}_L) \cong
\Aut\left(\Ta_{\lambda}\left(A\right)\right)\end{align*} associated to the 
$\lambda$-adic Tate module $\Ta_{\lambda}(A)$ of $A$ is
equivalent to the Galois representation \begin{align*}\rho_{{\bf{f}}, \lambda}: G_F 
\longrightarrow \GL(\mathcal{O}_L)\end{align*} associated to ${\bf{f}}$ by 
the construction of Carayol \cite{Ca2}, Taylor \cite{Tay} and Wiles \cite{Wi2}
(Theorem \ref{CTW} below) for any prime $\lambda \subset \mathcal{O}_L$, where 
$\mathcal{O}_L$ contains all of the Fourier coefficients of ${\bf{f}}$.
One can make analogous definitions for the unramified and ordinary local cohomology groups
$H^1_{\fin}(K_v, A[\mathfrak{P}^n]) \subset H^1(K_v, A[\mathfrak{P}^n])$ 
and $H^1_{\ord}(K_v, A[\mathfrak{P}^n]) \subset H^1(K_v, A[\mathfrak{P}^n])$
as given below for Hilbert modular eigenforms. See for instance the discussion 
in \cite[$\S4$]{L30}. Rather than give them here, let us just state the following characterization. 

\begin{proposition}\label{mav} If $A/F$ is a modular abelian variety associated to 
an eigenform ${\bf{f}} \in \mathcal{S}_2(\mathfrak{N})$ as above, then we have
the following description of $\im(\mathfrak{K}_v)$:

\begin{itemize}
\item[(i)] $\im(\mathfrak{K}_v) = H_{\fin}^1(L_v, A_{{\bf{f}},n})$ if
$v \nmid \mathfrak{N} \subset \mathcal{O}_F$.
\item[(ii)] $\im(\mathfrak{K}_v) = H_{\ord}^1(L_v, A_{{\bf{f}},n})$ 
if $v = \mathfrak{p} \subset \mathcal{O}_F$.
\end{itemize} Here, $A_{\ff, n}$ is the mod $\mathfrak{P}^n$ Galois representation arising from 
the abelian variety $A_{\ff}$ associated to $\ff$, as defined in $ \S \ref{GalRep}$ below. \end{proposition} 

\begin{proof}
The result is well known, see for instance \cite{CG} or \cite[$\S$4.1]{L30}.
\end{proof} Let $$\Sel_{\mathfrak{P}^{\infty}}
(A/K_{\mathfrak{p}^{\infty}}) = \varinjlim_n
\Sel_{\mathfrak{P}^n}(A/K_{\mathfrak{p}^{\infty}}),$$
where the limit is taken with respect to the natural maps $A[\mathfrak{P}^n] 
\rightarrow A[\mathfrak{P}^{n+1}]$. Let $X(A/K_{\mathfrak{p}^{\infty}}) =
\Hom \left(\Sel_{\mathfrak{P}^{\infty}}(A/K_{\mathfrak{p}^{\infty}}),
{\bf{Q}}_p/{\bf{Z}}_p  \right)$. By Proposition \ref{mav}, we can identify 
$\Sel_{\mathfrak{P}^{\infty}}(A/K_{\mathfrak{p}^{\infty}}) = \Sel
({\bf{f}}, K_{\mathfrak{p}^{\infty}})$ to obtain the following result.

\begin{corollary}\label{dmcav} Let $A/F$ be a modular abelian variety. If the
eigenform ${\bf{f}}$ associated to $A$ satisfies all of the conditions of Theorem \ref{RESULT} 
above, then the dual Selmer group $X(A/K_{\mathfrak{p}^{\infty}})$ is $\Lambda$-torsion, 
and there is an inclusion of ideals  
\begin{align*} \left( \mathcal{L}_{\mathfrak{p}}({\bf{f}}, K_{\mathfrak{p}^{\infty}})  \right) \subseteq 
\left( \operatorname{char}_{\Lambda} X(A/
K_{\mathfrak{p}^{\infty}}) \right)~\text{in $\Lambda$}.
\end{align*}\end{corollary} Note as well that the analogous 
formulations of Theorem \ref{HOWARD} and Corollary \ref{criterion} carry over to the setting 
of modular abelian varieties. Now, consider the short exact descent sequence 
\begin{align}\label{tsav} 0 \longrightarrow
A(K_{\mathfrak{p}^{\infty}})\otimes {\bf{Q}}_p/ {\bf{Z}}_p \longrightarrow
\Sel_{\mathfrak{P}^{\infty}}(A/K_{\mathfrak{p}^{\infty}})
\longrightarrow \Sha(A/K_{\mathfrak{p}^{\infty}})[\mathfrak{P}^{\infty}]
\longrightarrow 0. \end{align} Here, $\Sha(A/K_{\mathfrak{p}^{\infty}})[\mathfrak{P}^{\infty}]$ denotes
the $\mathfrak{P}$-primary part of the Tate-Shafarevich group $\Sha(A/K_{\mathfrak{p}^{\infty}})$ of $A$ over $K_{\mathfrak{p}^{\infty}}$.

\begin{corollary}\label{finitecomponent} Let $A/F$ be a modular abelian variety. If the
eigenform ${\bf{f}}$ associated to $A$ satisfies all of the
conditions of Theorem \ref{RESULT} above, then for $\rho$ any finite order
character of $G_{\mathfrak{p}^{\infty}}$ for which the specialization
$\rho^{-1}\left( \mathcal{L}_{\mathfrak{p}}({\bf{f}},K_{\mathfrak{p}^{\infty}})
\right)$ does not vanish, the components $A(K_{\mathfrak{p}^{\infty}})^{\rho}$,
$\Sel_{\mathfrak{P}^{\infty}}({\bf{f}}, K_{\mathfrak{p}^{\infty}})^{\rho}$, and
$\Sha(A/K_{\mathfrak{p}^{\infty}})[\mathfrak{P}^{\infty}]^{\rho}$ are finite.
\end{corollary} \begin{proof} This is a direct consequence of Corollary \ref{dmcav} applied
to $(\ref{tsav})$.  \end{proof} Note that by the nonvanishing theorem of Cornut-Vatsal 
\cite[Theorem 1.4]{CV}, the nonvanishing hypothesis of Corollary \ref{finitecomponent} 
is satisfied for for all but finitely many finite order characters 
$\rho$ of $G_{\mathfrak{p}^{\infty}}$, as can be deduced as can be deduced from the 
algebraicity theorem of Shimura \cite{Sh2}. \end{remark}

\begin{remark}[Notations.] We write ${\bf{A}}_F$ to denote the adeles of $F$, with
${\bf{A}} = \bf{A}_{\bf{Q}}$, and ${\bf{A}}_f$ the finite adeles of ${\bf{Q}}$. We shall 
sometimes write $\widehat{F}^{\times}$ to denote the finite adeles of $F$. Given 
a finite prime $v$ of $F$, we fix a uniformizer $\varpi_v$ of $F_v$. We let $\kappa_v$
denote the residue fields of $F_v$ at $v$, with $q = q_v$ its cardinality, which is
not to be confused with the cohomology class constructed in $(\ref{zetaclass})$
below. Throughout, we write $\mathcal{O}_0$ to denote the 
${\bf{Z}}_p$-algebra generated by the images of the Fourier coefficients of $\ff$,
and $\mathcal{O}$ the integral closure of $\mathcal{O}_0$ in its fraction field $L$.
We let $\mathfrak{P}$ denote the maximal ideal of $\mathcal{O}$, and for each 
integer $n \geq 1$ put $\mathfrak{P}_n = \mathfrak{P}^n \cap \mathcal{O}_0$.
\end{remark}

\section{Automorphic forms}

\begin{remark}[Hilbert modular forms.] 

Given an ideal $\mathfrak{N} \subset \mathcal{O}_F$, let $\mathcal{S}_2(\N)$ 
denote the space of cuspidal Hilbert modular forms of parallel weight $2$, 
level $\N$, and trivial character. The space $\mathcal{S}_2(\N)$ comes equipped
 with the action of standard (classically or adelically defined) operators $T_v$ 
for each prime $v \nmid \N \subset \mathcal{O}_F$ and $U_v$ for each prime $v \mid 
\N \subset \mathcal{O}_F$. Let ${\bf{T}}(\N)$ denote the ${\bf{Z}}$-algebra generated 
by these operators. Given a Hilbert modular form $ \ff \in \mathcal{S}_2(\N)$, let 
$a_{\mathfrak{m}}(\ff)$ denote the normalized Fourier coefficient of $\ff$ at an ideal 
$\mathfrak{m}$ of $F$. We refer to \cite{Ga} \cite{Ge}, or \cite{Go} for precise
definitions and further background.

\begin{definition} A Hilbert modular form $\ff \in \mathcal{S}_2(\N)$ is said to 
be a {\it{normalized eigenform}} if it is a simultaneous eigenvector for all of 
the Hecke operators $T_v$ and $U_v$, with $T_v \ff  = a_v(\ff) \cdot \ff$ if $v 
\nmid \N$, $U_v \ff  = a_v(\ff) \cdot \ff$ if $v \mid \N$, and $a_{\mathcal{O}_F}(\ff) =1$.
\end{definition}

\begin{definition} A normalized eigenform $\ff \in \mathcal{S}_2(\N)$ is said to be a 
{\it{newform}} if there does not exist a form ${\bf{g}} \in \mathcal{S}_2(\mathfrak{M})$ 
for $\mathfrak{M} \mid \N \subset \mathcal{O}_F$ an ideal not equal to $\N$ such that 
$a_{\mathfrak{n}}(\ff) = a_{\mathfrak{n}}({\bf{g}})$ for all ideals $\mathfrak{n} \subset 
\mathcal{O}_F$ prime to $\N$. Given a prime $\mathfrak{q} \mid \N$, we say that 
$\ff \in \mathcal{S}_2(\N)$ is {\it{new at $\mathfrak{q}$}} if it does not arise from 
another form ${\bf{g}} \in \mathcal{S}_2(\N/\mathfrak{q})$ in this way.\end{definition} 
In general, given a normalized eigenform $\ff \in \mathcal{S}_2(\N)$, we fix a factorization 
$\N = \mathfrak{N}^{+}\mathfrak{N}^{-}$ such that $\ff$ is new at all primes dividing 
$\mathfrak{N}^-$. We then write $\mathcal{S}_2(\N^+, \N^{-})$ to denote the space of 
Hilbert modular cusp forms of parallel weight $2$ and level $\N$ that are new at all 
primes dividing $\N^-$, with ${\bf{T}}(\N^+, \N^{-})$ the corresponding algebra of Hecke 
operators. Finally, let us also make the following

\begin{definition} An eigenform ${\bf{f}} \in 
\mathcal{S}_2(\mathfrak{N})$ is {\it{$\mathfrak{p}$-ordinary}} 
if its $T_{\mathfrak{p}}$-eigenvalue $a_{\mathfrak{p}}({\bf{f}})$ is a $p$-adic 
unit with respect to any fixed embedding $\overline{\bf{Q}} \rightarrow 
\overline{{\bf{Q}}}_p$. In this case, if $\mathfrak{p} \nmid \mathfrak{N}$, then there exists a $p$-adic 
unit root $\alpha_{\mathfrak{p}}({\bf{f}})$ to the polynomial \begin{align}\label{char} 
x^2 - a_{\mathfrak{p}}({\bf{f}})x + q,\end{align} where $q$ denotes the cardinality 
of the residue field at $\mathfrak{p}$. \end{definition} 
In what follows, let us always view $\ff \in \mathcal{S}_2(\N)$ as a $p$-adic modular 
form via a fixed embedding $\overline{{\bf{Q}}} \rightarrow \overline{\bf{Q}}_p$. Let 
$\mathcal{O}_0$ denote the ${\bf{Z}}_p$-subalgebra of $\overline{\bf{Q}}_p$ generated 
by the Fourier coefficients of $\ff$, let $\mathcal{O}$ denote the integral closure of 
$\mathcal{O}_0$ in its field of fractions $L$, and let $\mathfrak{P}$ denote the maximal 
ideal of $\mathcal{O}$. \end{remark}

\begin{remark}[The Jacquet-Langlands correspondence.]  

Let $B$ denote a quaternion algebra defined over $F$, with 
$\Ram(B)$ the set of places of $F$ where $B$ is ramified. The theorem 
of Jacquet and Langlands \cite{JL} establishes a bijection from the space 
of automorphic representations of $(B \otimes {\bf{A}}_F)^{\times}$ of 
dimension greater than $1$ to the space of cuspidal automorphic representations 
of $\GL({\bf{A}}_F)$ that are discrete series (i.e. square integrable) at each place 
$v \in \Ram(B)$. This injection being a bijection on its image, we obtain the following 
more concrete result. Fix a compact open subgroup $H \subset \widehat{B}^{\times}$. Let 
$\mathbb{S}_2(H, B)$ denote the space of automorphic forms of weight $2$ and level $H$ on 
$B$ (to be defined precisely later), with the additional assumption that the forms are cuspidal 
if $B$ is an indefinite quaternion algebra. The space $\mathbb{S}_2(H, B)$ comes equipped with 
actions of standard Hecke operators at each prime $v \subset \mathcal{O}_F$ (to be defined 
precisely later). Let $\mathbb{T}(H, B)$ denote the ${\bf{Z}}$-algebra generated by these 
operators. Suppose that the discriminant $\disc(B)$ of $B$ is equal to some ideal 
$\N^{-} \subset \mathcal{O}_F$, and that  the level of $H$ is equal to some ideal 
$\N^{+} \subset \mathcal{O}_F$. Let us then write $\mathbb{S}_2(\N^+, \N^{-}) = \mathbb{S}_2(H,B)$ 
with $\mathbb{T}(\N^+, \N^{-}) = \mathbb{T}(H, B)$. The correspondence of Jacquet and Langlands 
then establishes an isomorphism of Hecke modules $\mathbb{S}_2(\N^+, \N^{-}) 
\cong \mathcal{S}_2(\N^+, \N^{-}).$ \end{remark}

\begin{remark}[Automorphic forms on totally definite quaternion algebras.]

Let $\mathcal{O}$ be any ring. Let $D$ be any totally definite quaternion algebra defined over 
$F$. Fix a compact open subgroup $U \subset \widehat{D}^{\times}$. Let $\mathbb{S}_2(U; 
\mathcal{O}) = \mathbb{S}_2(U, D; \mathcal{O})$ denote the space of $\mathcal{O}$-valued 
automorphic forms of weight $2$ and level $U$ on $D$, i.e. the space of functions 
\begin{align*} \Phi: D^{\times}\backslash \widehat{D}^{\times}/U 
&\longrightarrow \mathcal{O}\end{align*} such that $\Phi(dgu) = \Phi(g)$ for all 
$d \in D^{\times}$, $g \in \widehat{D}^{\times}$, and $u \in U$. Let 
$\mathbb{S}_2(U;\mathcal{O})_{\triv} \subset \mathbb{S}(U; \mathcal{O})$ denote the subspace of 
functions that factor through the adelization of the reducted norm homomorphism, 
$\nrd: \widehat{D}^{\times} \longrightarrow F^{\times}$.

\begin{definition}
Let $\mathcal{S}_2(U; \mathcal{O}) 
= \mathbb{S}_2(U;\mathcal{O})/\mathbb{S}_2(U;\mathcal{O})_{\triv}.$ 
Functions in this space are called {\it{$\mathcal{O}$-valued modular 
forms of weight $2$ and level $U$ on $D$}}. \end{definition} The space 
$\mathbb{S}(U; \mathcal{O})$ comes equipped with actions of Hecke operators, 
defined via double coset operators as follows. Given two compact open subgroups 
$U, U' \subset \widehat{D}^{\times}$ and an element $g \in \widehat{D}^{\times}$, the 
group $gUg^{-1}$ is commensurable with $U'$.  Fixing a decomposition of $U'$ into a 
disjoint union of cosets $\coprod_i \alpha_i (U' \cap g U g^{-1})$ gives an identification 
$U'gU = \coprod_i g_i U$, where $g_i = \alpha_i g$. The associated 
double coset operator $[U'gU]$ is then given by the linear map \begin{align*} [U'gU]: 
\mathbb{S}_2(U; \mathcal{O}) &\longrightarrow \mathbb{S}_2(U'; \mathcal{O}), ~ 
\left( [U'gU]  \Phi  \right)(x) = \sum_i \Phi(x g_i). \end{align*}

\begin{definition} Fix a compact open subgroup $U \subset \widehat{D}^{\times}$. 
Fix a finite set of places $S \supset \Ram(D)$ of $F$ such that $U$ admits a 
decomposition $U_S \times U^S$. The Hecke algebra $\mathbb{T}^S(U)= \mathbb{T}^S(U,B)$ 
is the (commutative) subring of ${\bf{Z}}[U\backslash \widehat{D}^{\times}/U]$ generated 
by double coset operators $[UgU]$ with $g \in (\widehat{D}^S)^{\times}$. It is isomorphic as 
a ring to ${\bf{Z}}[T_w, S_w, S_w^{-1}: w \notin S]$, where $T_w$ and $S_w$ are the standard 
Hecke operators $T_w = [U \eta_w U]$ and $S_w = [U \varpi_w U]$. \end{definition}

Fix a prime $v \notin S$ of $F$. Hence, $D$ is split at $v$, and we may fix an isomorphism 
$D_v \cong \M(F_v)$. Let us assume additionally that this isomorphism sends the component 
$U_v$ to $\GL(\mathcal{O}_{F_v})$. In what follows, we shall make the identification 
$U_v \cong \GL(\mathcal{O}_{F_v})$ implicitly. Let $U(v) \subset U$ be the subgroup defined 
by \begin{align*} U(v) &= \lbrace u \in U: u_v \equiv 
\left( \begin{array}{cccc}  * & * \\ 0 & *  \end{array}\right) \mod \varpi_v \rbrace. 
\end{align*} We have a pair of natural degeneracy maps \begin{align*}
\alpha^* = [U(v){\bf{1}} U] : \mathbb{S}_2(U; \mathcal{O}) &\longrightarrow 
\mathbb{S}_2(U(v); \mathcal{O}), ~(\alpha^*\Phi)(g) = \Phi(g) \\ 
\beta^* = [U(v) \eta_vU] : \mathbb{S}_2(U; \mathcal{O}) &\longrightarrow 
\mathbb{S}_2(U(v); \mathcal{O}), ~(\beta^*\Phi)(g) = \Phi(g\eta_v),\end{align*} 
as well as a pair of associated trace maps \begin{align*}\alpha_* = [U{\bf{1}} U(v)] : 
\mathbb{S}_2(U(v); \mathcal{O}) &\longrightarrow \mathbb{S}_2(U; \mathcal{O})\\ 
\beta_* = [U\eta_vU(v)] : \mathbb{S}_2(U(v); \mathcal{O}) &\longrightarrow 
\mathbb{S}_2(U; \mathcal{O}). \end{align*} These degeneracy and trace maps commute 
with the actions of Hecke operators $T_w$ (for $w \notin \Ram(D) \cup \lbrace v \rbrace$) 
on $\mathbb{S}_2(U; \mathcal{O})$ and $\mathbb{S}_2(U(v);\mathcal{O})$, as well 
as with the action by the centre of $\overline{D}^{\times}$.

\begin{definition} The {\it{$v$-new susbspace of $\mathbb{S}_2(U(v); \mathcal{O})$}} is 
the subspace defined by $$\begin{CD}\mathbb{S}_2(U(v); \mathcal{O})^{\vnew} = 
\ker ( \mathbb{S}_2(U(v); \mathcal{O}) @>{\alpha_*, \beta_*}>> \mathbb{S}_2(U; 
\mathcal{O})^{\oplus 2} ).  \end{CD}$$ This subspace is stable under the actions of Hecke 
operators $T_w$ (for $w \notin \Ram(D) \cup \lbrace v \rbrace$), as well as under the 
action of the centre of $\widehat{D}^{\times}$. \end{definition} \end{remark}

\begin{remark}[Shimura curves.]

Fix a place $\tau_1$ in the set of archimedean places
$\lbrace \tau_1, \ldots, \tau_d \rbrace$ of $F$. Let $B$ 
be any quaternion algebra over $F$ that is split
at $\tau_1$ and ramified at the remaining set of 
real places $\lbrace \tau_2, \ldots, \tau_d \rbrace$. 
Let $X = {\bf{C}}- {\bf{R}}$, which is two
copies of the Poincar\'e upper-half plane. The group 
$B^{\times} \subset B^{\times}_{\tau_1} \cong 
\GL({\bf{R}})$ acts naturally on $X$, via fractional 
linear transformation. Let $H \subset \widehat{B}^{\times}$ be 
any compact open subgroup. The diagonal left action of 
$B^{\times}$ on $ \widehat{B}^{\times} /H 
\times X$ defines a Riemann surface 
\begin{align}\label{Sh}M_H ({\bf{C}}) &= M_H(B,X)({\bf{C}})= 
B^{\times} \backslash \widehat{B}^{\times} \times X /H.\end{align}
Shimura proved that the curve $M_H({\bf{C}})$ has a
canonical model defined over the totally real field $F$. We 
adopt the standard convention of writing $M_H$ to 
denote this model, whose complex points are identified with 
the Riemann surface $M_H({\bf{C}})$. The curve $M_H$ 
is irreducible, but not necessarily geometrically irreducible.
Indeed, by strong approximation and the theorem of the norm,
the reduced norm homomorphism $\nrd: B \longrightarrow F$ is 
seen to induce a bijection of finite sets \begin{align}
\label{components} \pi_0\left(M_H ({\bf{C}}) \right) &\cong 
F^{\times}_{+}\backslash \widehat{F}^{\times}
/\nrd(H),\end{align} as explained for instance 
in \cite[$\S$ 1.2]{Ca}. 

We have the following description of Hecke operators acting on $M_H$. Given 
$g \in \widehat{B}^{\times}$, fix a pair of compact open subgroups $H, H' \subset  
\widehat{B}^{\times}$ such that $g^{-1}Hg \subset H'$. Multiplication on the right by 
$g$ induces natural maps $M_H({\bf{C}}) \rightarrow M_{H'}({\bf{C}})$ which descend 
to finite flat morphisms $M_H \rightarrow M_{H'}$. We may then define from these
maps Hecke operators via double coset operators. To be more precise, given a 
complex point $x = [g, h] \in M_H({\bf{C}})$, the double coset operator $[HgH]$ 
acts as \begin{align*} \left[ HgH\right](x) &= \sum_i [x g_i, h] \in \Div(M_H({\bf{C}})),
\end{align*} where $HgH = \coprod_i g_i H$. The algebra of Hecke operators \
\begin{align*} \mathbb{T}_H &= \End_{ {\bf{Z}}[\widehat{B}^{\times}] } \left( 
{\bf{Z}}[\widehat{B}^{\times}/H] \right) \cong {\bf{Z}}[H \backslash \widehat{B}^{\times}/H] 
\end{align*} has a natural left action on the jacobian $J_H$ of $M_H$. Hence, 
we also write $\mathbb{T}_H$ to denote the subring of $\End(J_H)$ generated
by these Hecke operators. \end{remark}

\begin{remark}[Automorphic forms on indefinite quaternion algebras.]

Let us now fix an indefinite quaternion algebra $B$ as above, ramified at all but 
but one real place $\tau_1$ of $F$. Let us also fix a compact open
subgroup $H \subset \widehat{B}^{\times}$. Recall that we let $M_H({\bf{C}})$
denote the associated complex Shimura curve $(\ref{Sh})$, with $M_H$ its 
canonical model defined over $F$.

\begin{definition} Let $\mathbb{S}_2(H, B)$ denote the space of functions
$\Phi: (B \otimes {\bf{A}}_F)^{\times} \longrightarrow {\bf{C}}$ such that 
\begin{itemize}

\item[(i)] $\Phi$ is left $B^{\times}$-invariant.
\item[(ii)] $\Phi$ is right invariant under ${\bf{R}}^{\times} \times \prod_{i=2}^d
 B_{\tau_i} ^{\times} \subset \left( B \otimes {\bf{R}} \right)^{\times}$.
\item[(iii)] $\Phi$ is right invariant under $H$.
\item[(iv)] For each $g \in \left( B \otimes {\bf{A}}_F \right)^{\times}$ and 
$\theta \in {\bf{R}}$, \begin{align*} \Phi   \left( g  \left[ 
\left( \begin{array}{cccc}  \cos \theta & - \sin \theta \\ \sin \theta & \cos \theta   
\end{array}\right), 1, \ldots, 1 \right]  \right) &= \exp(2 i \theta) \cdot \Phi(g).
\end{align*} \item[(v)] For each $g \in \left( B \otimes {\bf{A}}_F \right)^{\times}$, 
the function defined by \begin{align*} z = x + iy &\longmapsto \Phi(g, z) := 
\frac{1}{y} \cdot \Phi \left(g \left[ \left( \begin{array}{cccc}  -y & x \\ 0 & 1   
\end{array}\right), 1, \ldots, 1 \right]  \right)\end{align*} is holomorphic on the 
lower half plane $\mathfrak{H}^{-}$. \end{itemize} Note that there is a left action of 
$g \in \widehat{B}^{\times}$ on $\mathbb{S}_2(H, B)$ via
the rule $(g \cdot \Phi)(x) = \Phi(xg)$. We refer the reader to 
\cite[$\S 3.6$]{CV} for more details. \end{definition} Let $\Omega_H$ denote 
the sheaf of differentials on the Shimura curve $M_H$, with $\Omega_H({\bf{C}})$ 
its pullback to $M_H({\bf{C}})$. Hence, $\Omega_H({\bf{C}})$ is the sheaf of holomorphic 
$1$-forms on $M_H({\bf{C}})$. Let $\Gamma(\Omega_H({\bf{C}}))$ denote the global 
sections of $\Omega_H({\bf{C}})$.

\begin{proposition}\label{CV3.10} There is a $\widehat{B}^{\times}$-equivariant bijection 
of ${\bf{C}}$ -vector spaces $\Gamma(\Omega_H({\bf{C}})) \cong \mathbb{S}_2(H, B)$. 
\end{proposition}

\begin{proof} The identification is standard, see for instance \cite[Proposition 3.10]{CV}.
\end{proof} Let $J_H$ denote the Jacobian of $M_H$, with $J_H^*(0)$ the complex 
cotangent space of $J_H$ at $0$.

\begin{corollary} There is a $\widehat{B}^{\times}$-equivariant bijection
of ${\bf{C}}$-vector spaces $J_H^*(0) \cong \mathbb{S}_2(H,B)$. In particular,
there is an identification of the associated Hecke algebras: $\mathbb{T}_H = 
\mathbb{T}(H, B)$. \end{corollary}

\begin{proof} This follows from the (canonical) identification of 
$J_H^*(0)$ with $\Gamma(\Omega_H({\bf{C}}))$. \end{proof}
\end{remark}

\section{$p$-adic $L$-functions}

We sketch here the construction of $p$-adic $L$-functions given in \cite{VO}.

\begin{remark}[The integer factorization.]

Recall that we fix a prime $\mathfrak{p}\subset \mathcal{O}_F$, with 
$p$ the underlying rational prime. Fix an integral ideal $\N_0$
with $\ord_{\mathfrak{p}}(\N_0)\leq 1$. Let \begin{align}\label{N}
\N &= \begin{cases}\N_0 &\text{if $\mathfrak{p}\mid \N_0$}\\
\mathfrak{p}\N_0 &\text{if $\mathfrak{p}\nmid \N_0$}.\end{cases}
\end{align} Hence, $\ord_{\mathfrak{p}}(\N) =1$. Fix a totally 
imaginary quadratic extension $K$ of $F$, with relative discriminant
prime to $\N/\mathfrak{p}$. The choice of $K$ then determines a 
unique factorization \begin{align*}\N &= \mathfrak{p} \N^+\N^{-}
\end{align*} of $\N$ in $\mathcal{O}_F$, with $v \mid \N^+$ if and
only if $v$ is split in $K$, and $v \mid \N^{-}$ if and only if $v$
is inert in $K$. Let us assume additionally that $\N^{-}$ is the 
squarefree product of a number of primes congruent to $d \mod 2$.
Hence, there exists a totally definite quaternion algebra $D$ say 
of discriminant $\N^{-}$ defined over $F$. Observe that $D$ is split
at $\mathfrak{p}$ by hypothesis. Hence, we can and do fix an isomorphism
$\iota_{\mathfrak{p}}: D_{\mathfrak{p}} \cong \M(F_{\mathfrak{p}})$. Fix a cuspidal 
Hilbert modular eigenform $\ff \in \mathcal{S}_2(\N)$. Let us assume that
$\ff$ is either a newform, or else that it arises from a newform
of level $\N/\mathfrak{p}$ via the process of $\mathfrak{p}$-stablization.
Fix a compact open subgroup $U \subset \widehat{D}^{\times}$ of level
$\mathfrak{M} \subset \mathcal{O}_F$ prime to $\N^{-}$ (we shall
often just take $\mathfrak{M} = \mathfrak{p}\N^+$), assumed to be 
maximal at $\mathfrak{p}$. \end{remark}

\begin{remark}[Ring class towers.]

Given an ideal $\mathfrak{c} \subset \mathcal{O}_F$, let 
$\mathcal{O}_{\mathfrak{c}} = \mathcal{O}_F + \mathfrak{c}\mathcal{O}_K$
denote the $\mathcal{O}_F$-order of conductor $\mathfrak{c}$ of $K$.
Let $K[\mathfrak{c}]$ denote the abelian extension of $K$ characterized
by class field theory via the isomorphism: \begin{align*}\begin{CD}
\widehat{K}^{\times}/\widehat{\mathcal{O}}_{\mathfrak{c}}^{\times}K^{\times} 
@>{\rec_K}>> \Gal(K[\mathfrak{c}]/K).\end{CD}\end{align*} Here, $\rec_K$ 
denotes the Artin reciprocity map, normalized to send uniformizers to geometric 
Frobenius elements. Write $G[\mathfrak{c}]$ to denote the Galois group 
$\Gal(K[\mathfrak{c}]/K)$. Let $K[\mathfrak{p}^{\infty}]= \bigcup_{n \geq 0} 
K[\mathfrak{p}^n]$ with Galois group $G[\mathfrak{p}^{\infty}] = 
\Gal(K[\mathfrak{p}^{\infty}])$. Hence, $G[\mathfrak{p}^{\infty}]$ has 
the structure of a profinite group, $G[\mathfrak{p}^{\infty}] = \varprojlim_n
G[\mathfrak{p}^n]$. It is well known that the torsion subgroup 
$G[\mathfrak{p}^{\infty}]_{\tors} \subset G[\mathfrak{p}^{\infty}]$ is finite,
and moreover that the quotient $G[\mathfrak{p}^{\infty}]/G[\mathfrak{p}^{\infty}]
_{\tors}$ is topologically isomorphic to ${\bf{Z}}_p^{\delta}$, where $\delta = 
[F_{\mathfrak{p}}: {\bf{Q}}_p]$ (see for instance \cite[Corollary 2.2]{CV}).
Let $G_{\mathfrak{p}^{\infty}} =G[\mathfrak{p}^{\infty}]/G[\mathfrak{p}^{\infty}]
_{\tors}$ denote the ${\bf{Z}}_p^{\delta}$ quotient of $G[\mathfrak{p}^{\infty}]$. 
Let $K_{\mathfrak{p}^{\infty}}$ denote the dihedral or anticyclotomic 
${\bf{Z}}_p^{\delta}$ of $K$, so that $G_{\mathfrak{p}^{\infty}} = 
\Gal(K_{\mathfrak{p}^{\infty}}/K) \cong {\bf{Z}}_p^{\delta}.$ Given a positive
integer $n$, let $K_{\mathfrak{p}^n}$ denote the extension of $K$ for which
$G_{\mathfrak{p}^n} = \Gal(K_{\mathfrak{p}^n}/K) \cong \left({\bf{Z}}/
p^n{\bf{Z}}\right)^{\delta},$ so that $G_{\mathfrak{p}^{\infty}}= \varprojlim_n 
G_{\mathfrak{p}^{n}}$.

\end{remark}

\begin{remark}[Strong approximation.]

Fix a set of representatives $\lbrace x_i \rbrace_{i = 1}^h$ for the modified
class group $\mathfrak{Cl}_F/F_{\mathfrak{p}}^{\times}$, where $\mathfrak{Cl}_F$
denotes the narrow class group \begin{align*} \mathfrak{Cl}_F &= F_+^{\times}\backslash
\widehat{F}^{\times}/\widehat{\mathcal{O}}_F^{\times} \end{align*} of $F$,
with the condition that $(x_i)_{\mathfrak{p}}= 1$ for each $i = 1, \ldots h$.
A standard consequence of the strong approximation theorem (\cite[$\S$ 
III.4, Thm. 4.3]{Vi}) with the theorem of the norm (\cite[$\S$ III.4, Thm. 4.1]{Vi})
shows that there is a canonical bijection \begin{align*}\begin{CD}
\coprod_{i=1}^h D^{\times} \xi_i D_{\mathfrak{p}}^{\times}U @>{\eta_{\mathfrak{p}}}>>   
\widehat{D}^{\times}. \end{CD}\end{align*} Here, each $\xi_i$ is an element of 
$\widehat{D}^{\times}$ such that $(\xi_i)_{\mathfrak{p}}=1$ and $\nrd(\xi_i)=x_i$. 
For each $i = 1, \ldots h$, let us then define a subgroup \begin{align}\label{gamma_i} 
\Gamma_i &= \lbrace d \in D^{\times}: d_v \in (\xi_{i, v})U_v (\xi_{i, v})^{-1} \text{ for all } v \nmid 
\mathfrak{p} \rbrace \subset D^{\times}. \end{align} A standard argument shows 
that these subgroups $\Gamma_i \subset D^{\times}$ embed discretely into 
$D_{\mathfrak{p}}^{\times}$. Hence, via our fixed isomorphism $\iota_{\mathfrak{p}}: 
D_{\mathfrak{p}}^{\times} \cong \GL(F_{\mathfrak{p}})$, we can and do view these subgroups 
as discrete subgroups of $\GL(F_{\mathfrak{p}})$. Now, $\eta_{\mathfrak{p}}$ induces a 
canonical bijection (which we also denote by $\eta_{\mathfrak{p}}$) 
\begin{align*}\begin{CD} \coprod_{i=1}^h \Gamma_i \backslash 
D_{\mathfrak{p}}^{\times}/U_{\mathfrak{p}} @>{\eta_{\mathfrak{p}}}>> D^{\times}\backslash
\widehat{D}^{\times}/U \end{CD} \end{align*} via the map given on each component by 
$[d] \mapsto [\xi_i \cdot d ]$. Hence, we can view each modular form 
$\Phi \in \mathcal{S}_2(H; \mathcal{O}) = \mathcal{S}_2^D(H; \mathcal{O})$ as 
an $h$-tuple of functions $\left( \phi^i \right)_{i=1}^h $ on $\GL(F_{\mathfrak{p}})$
such that \begin{align*} \phi^i(\gamma d u z ) &= \phi^i(d) \end{align*}
for each $i = 1, \ldots, h$, with $\gamma \in \Gamma_i$, $d \in D_{\mathfrak{p}}^{\times}$,
$u \in U_{\mathfrak{p}}$, and $z \in \widehat{F}_{\mathfrak{p}}^{\times}$. A simple argument 
shows that these functions $\phi^i$ factor through homothety classes of full rank lattices
of $F_{\mathfrak{p}}\oplus F_{\mathfrak{p}}$, and hence can be viewed as functions on the
edgeset of the Bruhat-Tits tree of $D_{\mathfrak{p}}^{\times}/F_{\mathfrak{p}}^{\times} \cong
\PGL(F_{\mathfrak{p}}).$ To be more precise, let $\mathcal{T}_{\mathfrak{p}} =
(\mathcal{V}_{\mathfrak{p}}, \mathcal{E}_{\mathfrak{p}})$ denote the
Bruhat-Tits tree of $B_{\mathfrak{p}}^{\times}/F_{\mathfrak{p}}^{\times} \cong
\PGL(F_{\mathfrak{p}})$, which is the tree of maximal
orders of $\M(F_{\mathfrak{p}}) \cong B_{\mathfrak{p}}$ such that

\begin{itemize}

\item[(i)] The vertex set $\mathcal{V}_{\mathfrak{p}}$ is indexed by maximal
orders of $\M(F_{\mathfrak{p}})$.

\item[(ii)] The edgeset $\mathcal{E}_{\mathfrak{p}}$ is indexed by Eichler
orders of level $\mathfrak{p}$ of $\M(F_{\mathfrak{p}}).$

\item[(iii)] The edgeset $\mathcal{E}_{\mathfrak{p}}$ has an orientation, 
i.e. a pair of maps \begin{align*}s, t: \mathcal{E}_{\mathfrak{p}} \longrightarrow
\mathcal{V}_{\mathfrak{p}}, ~~~ \mathfrak{e} \mapsto
(s(\mathfrak{e}), t(\mathfrak{e}))\end{align*} that assigns to each edge
$\mathfrak{e} \in \mathcal{E}_{\mathfrak{p}}$ a {\it{source}}
$s(\mathfrak{e})$ and a {\it{target}} $t(\mathfrak{e})$. Once such a
choice of orientation is fixed, let us write
$\mathcal{E}_{\mathfrak{p}}^{*}$ to denote the so-called ``directed" edgeset
of $\mathcal{T}_{\mathfrak{p}}$.\end{itemize} The group $D_{\mathfrak{p}}^{\times}/
F_{\mathfrak{p}}^{\times}$ acts naturally by conjugation on $\mathcal{T}_{\mathfrak{p}}$.
It is a standard result that this action is transitive, and moreover that there
is an identification $\mathcal{V}_{\mathfrak{p}} \cong \PGL(F_{\mathfrak{p}})/\PGL(
\mathcal{O}_{F_{\mathfrak{p}}})$. In particular, we see that the discrete subgroups 
$\Gamma_i \subset D_{\mathfrak{p}}^{\times} \cong \GL(F_{\mathfrak{p}})$ modulo 
$F_{\mathfrak{p}}^{\times}$ act transitively by conjugation on $\mathcal{T}_{\mathfrak{p}}$. 
Now, each quotient graph $\Gamma_i \backslash \mathcal{T}_{\mathfrak{p}}$ is a finite graph.
Hence, we may consider the disjoint union of finite quotient graphs 
\begin{align*}\coprod_{i = 1}^{h} \Gamma_i \backslash \mathcal{T}_{\mathfrak{p}} =
\left(\coprod_{i = 1}^{h} \Gamma_i \backslash
\mathcal{V}_{\mathfrak{p}}, \coprod_{i = 1}^{h} \Gamma_i
\backslash \mathcal{E}_{\mathfrak{p}}^{*} \right).\end{align*}

\begin{definition} Let $\mathcal{S}_2\left( \coprod_{i=1}^{h} \Gamma_i
\backslash \mathcal{T}_{\mathfrak{p}} ; \mathcal{O} \right)$ denote the space
of vectors  $\left(\phi^i \right)_{i=1}^{h}$ of $\mathcal{O}$-valued, 
$\left(\Gamma_i\right)_{i=1}^{h}$-invariant functions 
on $\mathcal{T}_{\mathfrak{p}}.$ \end{definition}

\begin{remark}
Here, it is understood that $\Phi \in \mathcal{S}_2\left(
\coprod_{i=1}^{h}\Gamma_i \backslash \mathcal{T}_{\mathfrak{p}}; 
\mathcal{O} \right)$ is a function on $\coprod_{i=1}^{h}\Gamma_i 
\backslash \mathcal{V}_{\mathfrak{p}}$ if $\mathfrak{p} \nmid 
\mathfrak{M}$, and a function on $\coprod_{i=1}^{h}\Gamma_i \backslash
\mathcal{E}_{\mathfrak{p}}^{*}$ if $\mathfrak{p} \mid \mathfrak{M}$. We 
refer the reader to the discussion in \cite[$\S$ 3]{VO} for more explanation.
\end{remark} A simple argument (\cite[Proposition 3.6]{VO}) shows that 
the canonical bijection $\eta_{\mathfrak{p}}$ induces a bijection of $\mathcal{O}$-modules 
\begin{align}\label{SAvs}\mathcal{S}_2\left( \coprod_{i=1}^{h} \Gamma_i \backslash 
\mathcal{T}_{\mathfrak{p}} ; \mathcal{O} \right) &\longrightarrow \mathcal{S}_2(U;\mathcal{O}).
\end{align} Let us for simplicity of notation write $\Phi$ to denote both 
a modular form in $\mathcal{S}_2(U; \mathcal{O})$, as well as its correponding
vector of functions $(\phi^i)_{i=1}^h$ in the space $\mathcal{S}_2\left( \coprod_{i=1}^{h} 
\Gamma_i \backslash \mathcal{T}_{\mathfrak{p}} ; \mathcal{O} \right)$. We obtain from
$(\ref{SAvs})$ the following combinatorial description of the standard Hecke 
operators $T_{\mathfrak{p}}$ and $U_{\mathfrak{p}}$ acting on 
$\mathcal{S}_2(H;\mathcal{O})$. Here, $U_{\mathfrak{p}}$ denotes 
the standard Hecke operator defined in \cite[$\S$ 3]{VO}, and {\it{not}} the 
compact open subgroup $U \subset \widehat{D}^{\times}$ defined above.

\begin{remark} [Case I: $\mathfrak{p} \nmid \mathfrak{M}$.]
Let $\Phi(\mathfrak{v})$ denote evaluation of the corresponding 
$h$-tuple of functions $\left(\phi^i \right)_{i=1}^{h} $ at a vertex 
$\mathfrak{v} \in \mathcal{V}_{\mathfrak{p}}$. Then, 

\begin{align} \left( T_{\mathfrak{p}} \Phi \right) ({\bf{\mathfrak{v}}}) =
\sum_{ \mathfrak{w} \rightarrow \mathfrak{v} }
c_{\Phi}(\mathfrak{w}).
\end{align} Here, the sum ranges over all $q+1$ vertices
$\mathfrak{w}$ adjacent to $\mathfrak{v}$. \end{remark}

\begin{remark} [Case II: $\mathfrak{p} \mid \mathfrak{M}$.] 

Let $\Phi(\mathfrak{e})$ denote evaluation of the corresponding 
$h$-tuple of functions $\left(\phi^i \right)_{i=1}^{h} $ at a 
directed edge $\mathfrak{e} \in \mathcal{E}_{\mathfrak{p}}^*$. Then, 

\begin{align} \left( U_{\mathfrak{p}} \Phi \right) ({\bf{\mathfrak{e}}}) =
\sum_{s({\bf{\mathfrak{e}}}') = t({\bf{\mathfrak{e}}})}
c_{\Phi}({\bf{\mathfrak{e}}}'). \end{align} Here, the sum runs over
the $q+1$ edges $\mathfrak{e}' \in \mathcal{E}_{\mathfrak{p}}^{*}$
such that $s(\mathfrak{e}') = t(\mathfrak{e})$, minus the edge
obtained by reversing orientation. \end{remark} We also
obtain the following explicit version of the Jacquet-Langlands
correspondence induced by the bijection $\eta_{\mathfrak{p}}$.
That is, fix an integral ideal $\mathfrak{N} \subset \mathcal{O}_F$
as defined in $(\ref{N})$ above, with underlying integral ideal 
$\mathfrak{N}_0 \subset \mathcal{O}_F$. Fix a Hilbert modular eigenform 
${\bf{f}}_0 \in \mathcal{S}_2(\mathfrak{N}_0)$ that is new at all primes
dividing the ideal $\mathfrak{N}^-$. 

\begin{definition} Let $\mathfrak{N}_0 \subset \mathcal{O}_F$ be an integral
ideal that is not divisible by $\mathfrak{p}$. 
The {\it{$\mathfrak{p}$-stabilization}}
${\bf{f}} \in \mathcal{S}_2(\mathfrak{N})$ of ${\bf{f}}_0 \in 
\mathcal{S}_2(\mathfrak{N}_0)$ is the
eigenform given by \begin{align} \label{p-stab} {\bf{f}} =
{\bf{f}}_0 - \beta_{\mathfrak{p}}({\bf{f}}_0)\cdot \left(
T_{\mathfrak{p}}{\bf{f}}_0\right),
\end{align} where $\beta_{\mathfrak{p}}({\bf{f}}_0)$ denotes the
non-unit root to $(\ref{char})$. This is a $\mathfrak{p}$-ordinary
eigenform in $\mathcal{S}_2(\mathfrak{N})$ with
$U_{\mathfrak{p}}$-eigenvalue $\alpha_{\mathfrak{p}} 
= \alpha_{\mathfrak{p}}({\bf{f}}_0)$. \end{definition}
We now consider an eigenform ${\bf{f}} 
\in \mathcal{S}_2(\mathfrak{N})$ that is given by ${\bf{f}}_0$ if 
$\mathfrak{p}$ divides $\mathfrak{N}_0$, or given by the 
$\mathfrak{p}$-stabilization of ${\bf{f}}_0$ if $\mathfrak{p}$ 
does not divide $\mathfrak{N}_0$. We have the following 
quaternionic description of ${\bf{f}}$ in either case.
To be consistent with the notations above, let us write $U_v$ to
denote the Hecke operator $T_v$ acting on $S_2^B(H; \mathcal{O})$
at a prime $v \mid \mathfrak{N}^+$ (again not to be confused with the
component at $v$ of the level structure $U \subset \widehat{D}^{\times}$).

\begin{proposition}[Jacquet-Langlands]\label{p-JLC} 

Given an eigenform ${\bf{f}} \in \mathcal{S}_2(\mathfrak{N})$ 
as defined above, there exists a function $\Phi \in \mathcal{S}_2
\left( \coprod_{i=1}^{\mathfrak{h}}\Gamma_i \backslash \mathcal{T}_{\mathfrak{p}}; 
{\bf{C}} \right)$ such that

\begin{itemize}
\item $T_v \Phi = a_v({\bf{f}}) \cdot \Phi$ for all $v
\nmid \mathfrak{N}.$
\item $U_v \Phi = \alpha_v({\bf{f}}) \cdot \Phi$ for all $v
\mid \mathfrak{N}^{+}.$
\item $U_{\mathfrak{p}} \Phi = \alpha_{\mathfrak{p}} \cdot \Phi.$
\end{itemize} This function is unique up to multiplication by
non-zero complex numbers. Conversely, given an eigenform $\Phi \in
\mathcal{S}_2\left( \coprod_{i=1}^{\mathfrak{h}} \Gamma_i \backslash
\mathcal{T}_{\mathfrak{p}}; {\bf{C}} \right)$, there exists an
eigenform ${\bf{f}} \in \mathcal{S}_2(\mathfrak{N})$ such that

\begin{itemize}
\item $T_v {\bf{f}} = a_v(\Phi) \cdot {\bf{f}}$ for all $v
\nmid \mathfrak{N}.$
\item $U_v {\bf{f}} = \alpha_v(\Phi) \cdot {\bf{f}}$ for all $v
\mid \mathfrak{N}^{+}.$
\item $U_{\mathfrak{p}} {\bf{f}} = \alpha_{\mathfrak{p}}(\Phi) \cdot {\bf{f}}.$
\end{itemize} Here, $a_v(\Phi)$ denotes the eigenvalue for $T_v$ of
$\Phi$ if $v \nmid \mathfrak{N}$, and $\alpha_v(\Phi)$ the
eigenvalue for $U_v$ of $\Phi$ if $v \mid \mathfrak{N}$.
\end{proposition}

\begin{proof} See \cite[Proposition 3.7]{VO}, which is a direct generalization
of \cite[Proposition 1.3]{BD} to totally real fields.  \end{proof}
\end{remark}

\begin{remark}[Construction of measures.]

We now sketch the construction of $p$-adic measures given in 
\cite{VO}, which generalizes that of \cite[$\S$ 1.2]{BD}. 
Fix an integral ideal $\N \subset \mathcal{O}_F$ having the 
factorization $(\ref{N})$. Recall that we write $D$ to denote the 
totally definite quaternion algebra of discriminant $\N^{-}$ defined 
over $F$. Let $Z$ denote the maximal $\mathcal{O}_F[\frac{1}
{\mathfrak{p}}]$-order of $K$. Let $R \subset D$ denote an Eichler 
$\mathcal{O}_F[\frac{1}{\mathfrak{p}}]$-order of level $\N^+$. Let 
us fix an {\it{optimal embedding}} $\Psi$ of $Z$ into $R$, i.e. 
an injective $F$-algebra homomorphism $\Psi: K \longrightarrow D$ 
such that $\Psi(K) \cap R = \Psi(Z)$. Such an embedding exists if and only
if $K$ is split at all primes dividing the level of $R$ (\cite[$\S$ II.3]{Vi}). 
Hence, such an embedding exists by our choice of integer factorization $(\ref{N})$.

The Galois group $G[\mathfrak{p}^{\infty}]$ has a natural action on the 
directed edgeset $\mathcal{E}_{\mathfrak{p}}^*$ of $\mathcal{T}_{\mathfrak{p}}$.
That is, the reciprocity map $\rec_K$ induces a bijection
\begin{align*}\begin{CD} \widehat{K}^{\times}/\left(K^{\times} \prod_{v \nmid \mathfrak{p}} 
Z_{v}^{\times}\right) @>{r_K}>> G[\mathfrak{p}^{\infty}].\end{CD}\end{align*} 
Passing to the adelization, the optimal embedding $\Psi$ induces an embedding 
\begin{align*}\begin{CD} \widehat{K}^{\times}/ \left(K^{\times}
 \prod_{v \nmid \mathfrak{p}} Z_{v}^{\times}\right) @>{\widehat{\Psi}}>> 
D^{\times}\backslash \widehat{D}^{\times} / \prod_{v\nmid \mathfrak{p}} R_{v}^{\times}. \end{CD}\end{align*} 
Consider the compact open subgroup of $\widehat{D}^{\times}$ defined by 
$\prod_{v \nmid \mathfrak{p}}R_{v}^{\times}$, with associated subgroups $\Gamma_{i}$ as defined 
in $(\ref{gamma_i})$. As explained above, strong approximation 
induces a canonical bijection \begin{align*}\begin{CD} \coprod_{i=1}^{\mathfrak{h}} 
\Gamma_{i}\backslash D_{\mathfrak{p}}^{\times}/ F_{\mathfrak{p}}^{\times}
@>{\eta_{\mathfrak{p}}}>> D^{\times} \backslash \widehat{D}^{\times}
/ \prod_{v \nmid \mathfrak{p}} R_{v}^{\times}.
\end{CD}\end{align*} The composition $\eta_{\mathfrak{p}}^{-1} \circ \widehat{\Psi} \circ
r_{K}^{-1}$ then gives rise to a natural action $\star$ 
of the Galois group $G[\mathfrak{p}^{\infty}]$ on the Bruhat-Tits tree
$\mathcal{T}_{\mathfrak{p}} = (\mathcal{V}_{\mathfrak{p}},
\mathcal{E}_{\mathfrak{p}}^*)$: it is the induced conjugation action on
maximal orders of $D_{\mathfrak{p}} \cong \M(F_{\mathfrak{p}})$. This action
factors through that of $K_{\mathfrak{p}}^{\times}/F_{\mathfrak{p}}^{\times}$
on $\mathcal{T}_{\mathfrak{p}}$ via the local optimal 
embedding $\Psi_{\mathfrak{p}}:K_{\mathfrak{p}} \longrightarrow D_{\mathfrak{p}}$.
Moreover, it fixes a single vertex if $\mathfrak{p}$ is inert in $K$, or 
no vertex if $\mathfrak{p}$ is split in $K$. Given an integer $n \geq 1$, 
let us write $\mathcal{U}_n$ to denote the standard compact open subgroup of 
level $n$ of $K_{\mathfrak{p}}^{\times}/F_{\mathfrak{p}}^{\times}$, \begin{align*} 
\mathcal{U}_n &= \left(1 + \mathfrak{p}^n\mathcal{O}_K \otimes \mathcal{O}_{F_p}
\right)^{\times}/\left(1 + \mathfrak{p}^n\mathcal{O}_{F_p}\right)^{\times}.\end{align*} 
We can then fix a sequence of consecutive edges $\lbrace \mathfrak{e}_j \rbrace_{j \geq 1}$ 
in $\mathcal{E}_{\mathfrak{p}}^*$ such that \begin{align*} \Stab_{K_{\mathfrak{p}}^{\times}/ 
F_{\mathfrak{p}}^{\times}}(\mathfrak{e}_j) &= \mathcal{U}_j.\end{align*} Now, the choice of 
an eigenform $\Phi \in \mathcal{S}_2(\coprod_{i=1}^h \Gamma_i \backslash 
\mathcal{E}_{\mathfrak{p}}^*; \mathcal{O})$ determines a pairing \begin{align*}
\label{pairing}\left[~,~\right]_{\Phi}: G[\mathfrak{p}^{\infty}] \times 
\mathcal{E}_{\mathfrak{p}}^*&\longrightarrow \mathcal{O} \\ (\sigma, \mathfrak{e}) 
&\longmapsto\Phi \left( \eta_{\mathfrak{p}}^{-1}\circ \widehat{\Phi} \circ r_K^{-1}(\sigma)\star 
\mathfrak{e}\right).\end{align*} Let $\mathcal{H}_{\infty}$ 
denote the group $\rec_K^{-1}(G[\mathfrak{p}^{\infty}])$, with profinite structure
$\mathcal{H}_{\infty} = \varprojlim_n \mathcal{H}_n$, where $\mathcal{H}_n = 
\mathcal{H}_{\infty}/\mathcal{U}_n$. Since the $U_{\mathfrak{p}}$-eigenvalue $\alpha_{\mathfrak{p}}$
is invertible in the ring of values $\mathcal{O}$, the pairing $\left[~,~\right]_{\Phi}$
can be shown to give rise to a natural $\mathcal{O}$-valued measure $\vartheta_{\Phi}$  
on $\mathcal{H}_{\infty}$ via the rule \begin{align*}\vartheta_{\Phi}(\sigma 
\mathcal{U}_j) &= \alpha_{\mathfrak{p}}^{-j} \cdot \left[\sigma, \mathfrak{e}_j\right]_{\Phi}
\end{align*} for all compact open subgroups of $\mathcal{H}_{\infty}$ of the form
$\sigma \mathcal{U}_j$, with $\sigma \in \mathcal{H}_{\infty}$. This distribution
gives rise to an element $\mathcal{L}_{\Phi}$ in the completed group ring 
$\mathcal{O}[[G[\mathfrak{p}^{\infty}]]]$ via the rule 
\begin{align*} (\mathcal{L}_{\Phi})_n &= \sum_{h \in \mathcal{H}_n} 
\vartheta_{\Phi}(h \mathcal{U}_n) \cdot h .\end{align*} Let us commit an abuse of 
notation in also writing $\mathcal{L}_{\Phi}$ to denote the image of this element in the 
Iwasawa algebra $\Lambda = \mathcal{O}[[G_{\mathfrak{p}^{\infty}}]]$. This image element 
is not well defined, since a different choice of sequence of consecutive edges 
$\lbrace \mathfrak{e}_j \rbrace_{j \geq 1}$ has the effect of multiplying $\mathcal{L}_{\Phi}$ 
by an element of $G_{\mathfrak{p}^{\infty}}$. Hence, let $\mathcal{L}_{\Phi}^*$ denote the image 
of $\mathcal{L}_{\Phi}$ under the involution $\Lambda \rightarrow \Lambda$ induced
by inversion $\sigma \mapsto \sigma^{-1} \in G_{\mathfrak{p}^{\infty}}$. We then let

\begin{align} \mathcal{L}_{\mathfrak{p}}(\Phi, K_{\mathfrak{p}^{\infty}}) 
&= \mathcal{L}_{\Phi}\mathcal{L}_{\Phi}^{*}. \end{align} Observe that 
$\mathcal{L}_{\mathfrak{p}}(\Phi,K_{\mathfrak{p}^{\infty}})$ is then a well-defined element 
of $\Lambda$. We shall refer to this element $\mathcal{L}_{\mathfrak{p}}(\Phi, 
K_{\mathfrak{p}^{\infty}})\in \Lambda$ as the {\it{$p$-adic $L$-function 
associated to $\Phi$}}. Moreover, if $\Phi$ is associated to a Hilbert modular eigenform 
$\ff \in \mathcal{S}_2(\mathfrak{N})$ via the Jacquet-Langlands correspondence (as 
described above), then we shall also write $\mathcal{L}_{\mathfrak{p}}(\ff,
K_{\mathfrak{p}^{\infty}}) = \mathcal{L}_{\mathfrak{p}}(\Phi, K_{\mathfrak{p}^{\infty}})$, 
with $\mathcal{L}_{\mathfrak{p}}(\ff, K_{\mathfrak{p}^{\infty}}) = \mathcal{L}_{\ff} {\mathcal{L}_{\ff}}^*$.
\end{remark}

\begin{remark}[Interpolation properties.] The $p$-adic $L$-function 
$\mathcal{L}_{\mathfrak{p}}(\Phi, K_{\mathfrak{p}^{\infty}})$ satisfies the
following rough interpolation property; we refer the reader to 
\cite[Theorem 4.7]{VO} for a more precise version. Given a finite order 
character $\rho$ of $G_{\mathfrak{p}^{\infty}}$, let \begin{align*} \rho\left(
\mathcal{L}_{\mathfrak{p}}(\Phi, K_{\mathfrak{p}^{\infty}}) \right) &= 
\int_{G_{\mathfrak{p}^{\infty}}} \rho(\sigma) d \mathcal{L}_{\mathfrak{p}}
(\Phi, K_{\mathfrak{p}^{\infty}})(\sigma)\end{align*} denote the specialization of 
$\mathcal{L}_{\mathfrak{p}}(\Phi, K_{\mathfrak{p}^{\infty}})$ to $\rho$. 
Here, $d\mathcal{L}_{\mathfrak{p}}(\Phi, K_{\mathfrak{p}^{\infty}})$ denotes 
the measure of $\Lambda$ defined by $\mathcal{L}_{\mathfrak{p}}(\Phi, 
K_{\mathfrak{p}^{\infty}})$. Let $\langle \Phi, \Phi \rangle$ denote the 
Petersson inner product of $\Phi$. Fix embeddings $\overline{\bf{Q}}
\rightarrow \overline{\bf{Q}}_p$ and $\overline{\bf{Q}}_p \rightarrow 
{\bf{C}}$. Fix a finite order character $\rho$ of $G_{\mathfrak{p}^{\infty}}$.
Suppose that $\rho$ factors through $G_{\mathfrak{p}^m}$ for some integer
$m \geq 1$. Let us view the values of $\rho$ and 
$d\mathcal{L}_{\mathfrak{p}}(\Phi, K_{\mathfrak{p}^{\infty}})$ as complex values
via the embedding $\overline{\bf{Q}}_p \rightarrow {\bf{C}}$, in which 
case we write $\vert \rho \left(\mathcal{L}_{\mathfrak{p}}(\Phi, 
K_{\mathfrak{p}^{\infty}})\right)\vert$ to denote the complex absolute
value of the specialization $\mathcal{L}_{\mathfrak{p}}(\Phi, K_{\mathfrak{p}^{\infty}})$. 
We then have the following interpolation formula: 
\begin{align}\label{interpolation} \vert \rho \left(\mathcal{L}_{\mathfrak{p}}(\Phi, 
K_{\mathfrak{p}^{\infty}})\right){\vert}^2 &= \alpha_{\mathfrak{p}}^{-4m} \cdot 
\kappa(\Phi, F) \cdot L(\Phi \times \rho, 1). \end{align}
Here, $\kappa(\Phi, F)$ is a nonvanishing product of algebraic 
constants (which can be given precisely in terms of certain special
values of some related $L$-functions), and $L(\Phi \times \rho, 1)$
is the central value of the Rankin-Selberg $L$-function of $\Phi$ 
times the twisted theta series associated to $\rho$. Moreover, both 
sides of $(\ref{interpolation})$ belong to $\overline{{\bf{Q}}}_p$. This result
is deduced from the generalization of Waldspurger's formula 
shown in Yuan-Zhang-Zhang \cite{YZ^2}. In particular, we see from this
that the specialization $ \rho \left(\mathcal{L}_{\mathfrak{p}}(\Phi, 
K_{\mathfrak{p}^{\infty}})\right)$ vanishes if and only if the central
value $L(\Phi \times \rho, 1)$ vanishes. Hence, we deduce from the
nonvanishing theorem of Cornut-Vatsal \cite[Theorem 1.4]{CV} that the 
$p$-adic $L$-function $\mathcal{L}_{\mathfrak{p}}(\Phi, K_{\mathfrak{p}^{\infty}})$
does not vanish identically.\end{remark}

\begin{remark}[The dihedral $\mu$-invariant.] Recall that we let 
$\mathfrak{P}$ denote the maximal ideal of the local ring 
$\mathcal{O}$. Given an element $\lambda \in \Lambda$, we define
the {\it{dihedral $mu$-invariant $\mu(\lambda)$ of $\lambda$}} to
be the largest exponent $c$ such that $\lambda \in \mathfrak{P}^c
\Lambda$. Following the method of Vatsal \cite{Va2} it can be
shown (\cite[Theorem 4.10]{VO}) that $\mu(\mathcal{L}_{\mathfrak{p}}
(\Phi, K_{\mathfrak{p}^{\infty}})) = 2 \nu$, where $\nu = \nu_{\Phi}$
is defined to be  the largest positive integer such that $\Phi$ is 
congruent to a constant modulo $\mathfrak{P}^{\nu}$.
\end{remark}

\section{Galois representations}\label{GalRep}

\begin{remark}[Galois representations associated to Hilbert modular forms.]

Recall that we write $G_F$ denote the absolute Galois group $\Gal(\overline{\bf{Q}}/F)$.

\begin{theorem}[Carayol-Taylor-Wiles]\label{CTW}

Fix an eigenform $\ff \in \mathcal{S}_2(\N)$, with $ \pi_{\ff}$ its associated automorphic 
representation of $\GL(F)$. Let $\mathcal{O}_{\ff}$ be the ring of
integers of any number field such that there exists a morphism \begin{align*}
 \theta_{\ff}: {\bf{T}}(\N) &\longrightarrow \mathcal{O}_{\ff} \end{align*} with
 $\ff \vert T = \theta_{\ff}(T) \ff$ for any Hecke operator $T \in {\bf{T}}(\N)$. Then,
 for each prime $\lambda \subset \mathcal{O}_{\ff}$, there exists a continuous 
 representation \begin{align*} \rho_{\ff, \lambda}: G_F &\longrightarrow \GL(\mathcal{O}_{\ff, \lambda})
\end{align*} such that the following property holds: for any prime $v \subset \mathcal{O}_F$
of residue characteristic not equal to that of $\lambda$, the restriction of the representation 
$\rho_{\ff, \lambda}$ to the decomposition subgroup at $v$ is conjugate to the $\lambda$-adic 
representation of associated by the local Langlands correspondence to the local component of 
 $\pi_{\ff}$ at $v$. Here, $\mathcal{O}_{\ff, \lambda}$ denotes the localization at $\lambda$ of
 $\mathcal{O}_{\ff}$. \end{theorem}

\begin{proof} This results follows from the specializations of works of Carayol \cite{Ca2}, Taylor \cite{Tay}
and Wiles \cite{Wi2} to parallel weight $2$.  \end{proof}  Recall that we view $\ff \in \mathcal{S}_2(\N)$ 
as a $p$-adic modular form via a fixed embedding $\overline{{\bf{Q}}} \rightarrow \overline{\bf{Q}}_p$, writing
$\mathcal{O}_0$ to denote the ${\bf{Z}}_p$-subalgebra of $\overline{\bf{Q}}_p$ generated 
by the Fourier coefficients of $\ff$, $\mathcal{O}$ the integral closure of 
$\mathcal{O}_0$ in its field of fractions $L$, and $\mathfrak{P}$ to denote the maximal 
ideal of $\mathcal{O}$. We then write \begin{align*} \rho_{\ff} = \rho_{\ff, \mathfrak{P}}:G_F 
&\longrightarrow \GL(\mathcal{O})\end{align*} to denote the $\mathfrak{P}$-adic Galois 
representation associated to $\ff$ by Theorem \ref{CTW}. Let $T_{\ff}$ be the lattice $\mathcal{O}^2$, 
together with the action of $G_F$ given by $\rho_{\ff}$. If (as we shall always assume) the residual 
representation $T_{\ff}/\mathfrak{P}$ is irreducible, then $T_{\ff}$ is the unique $G_F$-stable sublattice 
of $L^2$ up to homothety. We then define $A_{\ff} = \left( T_{\ff} \otimes L \right)/T_{\ff} \cong 
\left( L/ \mathcal{O}\right)^2.$ We also define $G_F$-modules \begin{align*} T_{\ff, n} = T_{\ff}/ 
\mathfrak{P}T_{\ff} ~~~\text{and}~~~ A_{\ff, n} = A_{\ff}[\mathfrak{P}^n].\end{align*} 
These modules are of course isomorphic. However, we maintain a notational distinction
as these modules form respective projective and injective systems \begin{align}\label{limits} T_{\ff} = 
\varprojlim_n T_{\ff,n} ~~~&\text{and}~~~A_{\ff} = \varinjlim_n A_{\ff, n}. \end{align} \end{remark}

\begin{remark}[Abelian varieties associated to Hilbert modular forms.]

We now explain how to associate to $\ff \in \mathcal{S}_2(\N)$ an abelian variety over $F$,
following the construction of Carayol \cite{Ca2}. Recall we assume for simplicity that 
$\mathfrak{P}$ is contained in $\mathcal{O}_0$.

\begin{proposition}\label{hilbertAV} Fix an eigenform $\ff \in \mathcal{S}_2(\N)$, with
$\pi_{\ff}$ the associated automorphic representation of $\GL({\bf{A}}_F)$. Assume that either 
$d$ is odd, or else $d$ is even with the condition that there exists a finite place $v \subset 
\mathcal{O}_F$ at which $\pi_{\ff}$ is either special or supercuspidal. Then, we can associate
to $\ff$ an abelian variety $A$ defined over $F$. Moreover, there is a $G_F$-module isomorphism
$\Ta_{\mathfrak{P}}A \otimes_{\mathcal{O}_0} L \cong A_{\ff} \otimes_{\mathcal{O}} L$. \end{proposition}

\begin{proof} The result is presumably well known, see for instance \cite[$\S$ 3]{CV}. 
We sketch the construction for lack of a better reference. 

Fix positive integers $k \geq 2$ and $w$ having the same parity. Let $D_{k,w}$ denote the representation of 
$\GL({\bf{R}})$ that occurs via unitary induction as $\Ind(\mu, \nu)$, where $\mu$ and $\nu$ are the characters 
on ${\bf{R}}^{\times}$ given by

\begin{align*}
\mu(t) &= \vert t \vert ^{\frac{1}{2}(k-1-w)}\operatorname{sgn}(t)^k \\
\nu(t) &=  \vert t \vert ^{\frac{1}{2}(-k+1-w)}.
\end{align*} Fix integers $k_1, \ldots k_d$ all having the same parity. Let
$\pi \cong \bigotimes_v \pi_v$ be a cuspidal automorphic representation of 
$\GL({\bf{A}}_F)$ such that  for each real place $\tau_i$ of $F$, there is an
isomorphism $\pi_{\tau_i} \cong D_{k_i, w}$. It is well know that such 
representations correspond to holomorphic Hilbert modular 
forms of weight ${\bf{k}}=(k_1, \ldots, k_d)$. If $d$ is even, then we assume 
that there exists a finite prime $v \subset \mathcal{O}_F$ where the local component 
$\pi_v$ is either a special or supercuspidal representation of $\GL(F_v)$. 

Let $B$ be a quaternion algebra over $F$ that is ramified at 
$\lbrace \tau_2, \ldots, \tau_d \rbrace$ if $d$ is odd, and ramified at 
$\lbrace \tau_2, \ldots, \tau_d, v \rbrace$ 
if $d$ is even. Let $G = \Res_{F/{\bf{Q}}}(B^{\times})$ denote the associated
algebraic group over ${\bf{Q}}$. Hence, we have an isomorphism
$G({\bf{R}}) \cong \GL({\bf{R}}) \times \left( \mathbb{H}^{\times} \right)^{d-1},$
where $\mathbb{H}$ denotes the Hamiltonian quaternions. Let $\overline{D}_{k,w}$ denote 
the representation of $\mathbb{H}^{\times}$ corresponding to $D_{k,w}$ via the 
Jacquet-Langlands correspondence . We then consider cuspidal automorphic 
representations $\pi' = \bigotimes_v \pi_v'$  of $G({\bf{A}}_F)$ such that 
$\pi_{\tau_1}' \cong D_{k_1, w}$ and $\pi_{\tau_i} \cong  \overline{D}_{k_i, w}$ for 
$i = 2, \ldots, d$. Let us now fix such a representation $\pi'$ associated to 
$\pi = \pi_{\ff}$. Hence, ${\bf{k}} = (2, \ldots, 2)$. Fix a vector $\Phi \in \pi'$. 
Hence, $\Phi$ is seen to be a function in the space $\mathbb{S}_2(H, B)$ for some 
compact open subgroup $H \subset G({\bf{A}}_f)$. By Proposition \ref{CV3.10}, we can 
identify $\Phi$ with a section of the sheaf of holomorphic $1$-forms $\Omega_H({\bf{C}})$ 
on the complex Shimura curve $M_H({\bf{C}})= M_{H}(B, X)({\bf{C}})$. Recall that we let 
$J_H$ denote the Jacobian of the canonical model $M_H$. Let $\mathbb{T}$ denote the 
subalgebra of $\End(J_H)$ generated by Hecke correspondences acting on $J_H$. By 
Proposition \ref{CV3.10}, we deduce the identification $\mathbb{T} = \mathbb{T}(H,B)$.
Consider the homorphism \begin{align*} \theta_{\Phi}: \mathbb{T} &\longrightarrow 
\mathcal{O} \end{align*} that sends each operator in $\mathbb{T}$ the the eigenvalue 
for its action on $\Phi$. Let $I_{\Phi} = \ker(\theta_{\Phi}).$ Consider the quotient 
abelian variety defined by \begin{align*} A_H &= J_H/I_{\Phi} J_H. \end{align*} Hence, 
we have constructed from $\ff \in \mathcal{S}_2(\N)$ an abelian variety $A = A_H$ 
defined over $F$. Now, by the construction of Carayol \cite{Ca2}, we claim that the 
Galois representation $\rho_{\ff}$ is equivalent to the Galois representation 
arising from the $\mathfrak{P}$-adic Tate module of $A$. Hence, we deduce that there 
is an identification of $G_F$-modules $\Ta_{\mathfrak{P}}A \otimes_{\mathcal{O}_0} 
L \cong A_{\ff} \otimes_{\mathcal{O}} L$.  \end{proof}

\begin{corollary}\label{AVpairing} For each integer $n \geq 1$, there is a canonical, 
nondegenerate $G_F$-equivariant pairing \begin{align*} T_{\ff, n} \times A_{\ff, n} 
&\longrightarrow \mu_{p^n}.\end{align*}\end{corollary}

\begin{proof}
The pairing is induced from the Weil pairing, after composition with a suitable choice of 
polarization map. See for instance the discussion in \cite[$\S$ 2.3]{Ho1}.
\end{proof}  Given a prime $w \subset \mathcal{O}_F$, let us
choose a decomposition subgroup of $G_F$ above $w$, which we can and will 
identify with the Galois group $G_{F_{w}} = 
\Gal(\overline{F}_{w}/F_{w})$ for some choice of 
algebraic closure $\overline{F}_{w}$. Let $I_w = I_{F_{w}}$ 
denote the inertia subgroup at $w$. 

\begin{lemma}\label{ord}
For each prime $w \mid \mathfrak{N}^{-} \subset \mathcal{O}_F$, 
the maximal $I_w$-invariant submodule of $A_{\bf{f}}$ is divisible 
of $\mathcal{O}$-corank one.\end{lemma}

\begin{proof}
If ${\bf{f}}$ is associated to a modular abelian variety 
defined over $F$ of arithmetic conductor $\mathfrak{N}$ having good reduction 
outside of $\mathfrak{N}$, ordinary reduction at $\mathfrak{p}$, and 
purely toric reduction at each prime $w \mid \mathfrak{N}^{-}$, 
then this condition is satisfied (cf. \cite[Remark 1, p. 14]{BD} with the relevant
sections of \cite{L30} or \cite{Lo2}). Granted that ${\bf{f}}$ is $\mathfrak{p}$-ordinary, 
we can deduce from Proposition \ref{hilbertAV} that the associated abelian variety
$A$ always has these properties. In particular, the toric reduction 
of $A$ at primes $w \mid \mathfrak{N}^{-}$ can be deduced from
the general theory of N\'eron models given in the standard reference 
\cite[Ch. 9]{BLR}, using the fact that the associated Shimura curve has a 
semistable reduction at $w$. See for instance the discussion in 
\cite[1.6.4]{Nek} along with the description of integral models given below.
\end{proof}

\end{remark}

\section{Selmer groups}

\begin{remark}[Global cohomology.]

Given integers $m,n \geq 1$, we define continuous cohomology
groups of $G_{K_{\mathfrak{p}^m}} = \Gal(\overline{\bf{Q}}/K_{\mathfrak{p}^m})$
with coefficients in the modules $T_{\bf{f}} = \varprojlim_n T_{{\bf{f}}, n}$
and $A_{\bf{f}}= \varinjlim_n A_{{\bf{f}},n},$ 
\begin{align*} H^1(K_{\mathfrak{p}^m}, T_{\bf{f}}) &= \varprojlim_n
H^1(K_{\mathfrak{p}^m}, T_{{\bf{f}},n}) \\
H^1(K_{\mathfrak{p}^m}, A_{{\bf{f}}}) &= \varinjlim_n
H^1(K_{\mathfrak{p}^m}, A_{{\bf{f}},n}).\end{align*}
Note that these identifications can be justified
(see \cite[Proposition 2.2]{Ta2}). 
We also define cohomology groups of
$G_{K_{\mathfrak{p}^{\infty}}} = \Gal(\overline{{\bf{Q}}}/
K_{\mathfrak{p}^{\infty}}),$ \begin{align*}
\widehat{H}^1(K_{\mathfrak{p}^{\infty}}, T_{\bf{f}}) &=
\varprojlim_m H^1(K_{\mathfrak{p}^m}, T_{\bf{f}}), \\
H^1(K_{\mathfrak{p}^{\infty}}, A_{\bf{f}}) &= \varinjlim_m
H^1(K_{\mathfrak{p}^m}, A_{\bf{f}}).\end{align*} Here, the direct limit is
taken with respect to natural restriction maps, and the inverse
limit with respect to natural corestriction maps. Note that 
the compatible actions of the group rings 
$\mathcal{O}[G_{\mathfrak{p}^m}]$ on the cohomology groups 
$H^1(K_{\mathfrak{p}^m}, T_{\bf{f}})$ and 
$H^1(K_{\mathfrak{p}^m}, A_{\bf{f}})$ for each integer $m\geq 1$
induce an action of the Iwasawa algebra $\Lambda = 
\mathcal{O}[[G_{\mathfrak{p}^{\infty}}]]$ on the cohomology
groups $\widehat{H}^1(K_{\mathfrak{p}^{\infty}}, T_{\bf{f}})$
and $H^1(K_{\mathfrak{p}^{\infty}}, A_{\bf{f}})$. \end{remark}

\begin{remark}[Local cohomology.] 

Fix an integer $m \geq 1$. Given a finite prime $v$ of $F$, 
let us for notational simplicity write \begin{align*} 
K_{\mathfrak{p}^m, v} &= K_{\mathfrak{p}^m} \otimes F_v
= \bigoplus_{\mathfrak{v} \mid v} K_{\mathfrak{p}^m, \mathfrak{v}}\end{align*}
to denote the direct sum over completions of $K_{\mathfrak{p}^m}$ at
each prime $\mathfrak{v}$ above $v$ in $K_{\mathfrak{p}^m}$. 
Hence, we can define local cohomology groups 
\begin{align*} \widehat{H}^1(K_{\mathfrak{p}^{\infty},  v} T_{{\bf{f}},n}) 
&= \varprojlim_m \bigoplus_{\mathfrak{v}\mid v} 
H^1(K_{\mathfrak{p}^m, \mathfrak{v}}, T_{{\bf{f}},n}), \\
 \widehat{H}^1(K_{\mathfrak{p}^{\infty},  v} T_{{\bf{f}}}) 
&= \varprojlim_m \bigoplus_{\mathfrak{v}\mid v}
 H^1(K_{\mathfrak{p}^m, \mathfrak{v}}, T_{{\bf{f}}}),\\
 H^1(K_{\mathfrak{p}^{\infty}, v}, A_{{\bf{f}},n}) 
&= \varinjlim_m \bigoplus_{\mathfrak{v}
\mid v} H^1(K_{\mathfrak{p}^m, \mathfrak{v}}, A_{{\bf{f}},n}),\\
 H^1(K_{\mathfrak{p}^{\infty}, v}, A_{{\bf{f}}}) 
&= \varinjlim_m \bigoplus_{\mathfrak{v}
\mid v} H^1(K_{\mathfrak{p}^m, \mathfrak{v}},
A_{{\bf{f}}}).\end{align*} Here (as in the global case), the direct
limits are taken with respect to natural restriction maps, and the
inverse limits with respect to natural corestriction maps.  
Taking appropriate limits from $(\ref{limits})$
again, we then define
 \begin{align*}\widehat{H}^1(K_{\mathfrak{p}^{\infty},v}, T_{\bf{f}})& =
\varprojlim_n H^1(K_{\mathfrak{p}^{\infty}, v}, T_{{\bf{f}},n})\\
H^1(K_{\mathfrak{p}^{\infty},v}, A_{\bf{f}}) &= \varinjlim_n
H^1(K_{\mathfrak{p}^{\infty},v}, A_{{\bf{f}},n}).\end{align*} These
identifications as before can be justified (see 
\cite[Proposition 2.2]{Ta}). As in the global case, 
the Iwasawa algebra $\Lambda$ acts on the local cohomology 
groups $\widehat{H}^1(K_{\mathfrak{p}^{\infty},v}, T_{\bf{f}})$ 
and $H^1(K_{\mathfrak{p}^{\infty},v},A_{\bf{f}})$ in such a way that is 
compatible with respect to the respective corestriction 
and restriction maps. \end{remark}

\begin{remark}[Local Tate pairings.]

Recall that for each integer $n \geq 1$, we have an isomorphism of
$G_F$-modules $T_{{\bf{f}}, n} \cong A_{{\bf{f}},n}.$ Recall as well
that by Corollary \ref{AVpairing} (cf. \cite[$\S 2.3$]{Ho1}), there exists a canonical, $G_F$-equivariant pairing 
\begin{align*}T_{{\bf{f}},n} \times A_{{\bf{f}},n}\longrightarrow  
{\bf{Z}}/p^n {\bf{Z}}(1) = \mu_{p^n}.\end{align*} Composition with the cup product 
of local cohomology then gives a collection of local Tate pairings
\begin{align*}\langle~,~\rangle_{m,v}: \widehat{H}^{1}(K_{\mathfrak{p}^m, v}, 
T_{{\bf{f}},n}) \times H^1(K_{\mathfrak{p}^m, v}, A_{{\bf{f}},n}) &\longrightarrow  
{\bf{Q}}_p / {\bf{Z}}_p.\end{align*} Passage to the limit(s) then induces a perfect pairing
\begin{align}\label{plt} \langle~,~\rangle_{v}: \widehat{H}^{1}(K_{\mathfrak{p}^{\infty}, v}, 
T_{{\bf{f}},n})\times H^1(K_{\mathfrak{p}^{\infty},v}, A_{{\bf{f}},n})
&\longrightarrow  {\bf{Q}}_p / {\bf{Z}}_p.\end{align} We refer the reader
to \cite[$\S$ 1]{Mi}, \cite{BD5} or \cite{Ho1} for relevant background on local Tate
duality. The main fact we shall use is that $(\ref{plt})$ 
induces an isomorphism of $\Lambda$-modules
\begin{align*}\widehat{H}^1(K_{\mathfrak{p}^{\infty},v}, T_{{\bf{f}},n}) 
&\cong H^1(K_{\mathfrak{p}^{\infty}, v}, A_{{\bf{f}},n})^{\vee}. \end{align*} 
Here, $H^1(K_{\mathfrak{p}^{\infty}, v}, A_{{\bf{f}},n})^{\vee}$ denotes
the Pontryagin dual of $H^1(K_{\mathfrak{p}^{\infty}, v},
A_{{\bf{f}},n})$, endowed with the usual $\Lambda$-module structure.
(cf. \cite[$\S$2]{BD}). \end{remark}

\begin{remark}[Singular/unramified structures.]

Given a prime $v \subset \mathcal{O}_F$, let
$ I_{m,v} = \bigoplus_{\mathfrak{v} \mid v} 
I_{m, \mathfrak{v}}$ denote
the direct sum over all primes $\mathfrak{v}$ above $v$ in
$K_{\mathfrak{p}^m}$ of the inertia subgroups 
$I_{m, \mathfrak{v}} = I_{K_{\mathfrak{p}^m}, \mathfrak{v}}$ in $ G_{K_{\mathfrak{p}^m}}$.

\begin{definition} Let $v \nmid \mathfrak{N} \subset \mathcal{O}_F$ be 
a prime that does not divide the residue characteristic of $\mathfrak{p}$. 
Let $M_{{\bf{f}},n}$ denote either $A_{{\bf{f}}, n}$ or $T_{{\bf{f}},n}$.

\begin{itemize} \item[(i)] The {\it{singular structure}}
$H^{1}_{\sing}(K_{\mathfrak{p}^m, v}, M_{{\bf{f}},n}) \subset 
H^{1}(K_{\mathfrak{p}^m, v}, M_{{\bf{f}},n})$ is 
\begin{align*}H^{1}_{\sing}(K_{\mathfrak{p}^m, v}, M_{{\bf{f}},n}) 
&= H^{1}(I_{m,v}, M_{{\bf{f}},n})^{G_{K_v}}.\end{align*}
\item[(ii)] The {\it{residue map}} $\partial_{v}$ is the
natural restriction map \begin{align*} \partial_{v}: 
H^{1}(K_{\mathfrak{p}^m,v}, M_{{\bf{f}},n}) &\longrightarrow 
H^{1}_{\sing}(K_{\mathfrak{p}^m,v}, M_{{\bf{f}},n}).\end{align*}
\item[(iii)] The {\it{unramified structure}}
$H^{1}_{\fin}(K_{\mathfrak{p}^m, v}, M_{{\bf{f}},n}) \subset 
H^{1}(K_{\mathfrak{p}^m,v}, M_{{\bf{f}},n})$ is the
kernel of the residue map $\partial_v$.\end{itemize}\end{definition}
Analogous definitions hold under passage to projective limits in the
case that $M_{{\bf{f}},n} = T_{{\bf{f}},n}$, and under inductive
limits in the case that $M_{{\bf{f}},n} = A_{{\bf{f}},n}$. Let us
also write $\partial_v$ to denote the induced residue maps on
$G_{K_{\mathfrak{p}^{\infty}}}$-cohomology. \end{remark}

\begin{remark}[Ordinary structures.]  

Recall that for each prime divisor $w \mid \mathfrak{N}^{-}$,
Lemma \ref{ord} shows that the maximal $I_{w}$-invariant 
submodule of $A_{\bf{f}}$ is divisible of $\mathcal{O}$-corank one. 
Hence, we have an exact sequence of $I_w$-modules  
\begin{align}\label{ordsequence} 0 \longrightarrow
A_{\bf{f}}^{(w)} \longrightarrow A_{\bf{f}}
\longrightarrow A_{\bf{f}}^{(1)} \longrightarrow 0,
\end{align} where $A_{{\bf{f}}}^{(1)}$ is the maximal
submodule of $A_{{\bf{f}}}$ on which $I_w$ acts trivially,
giving trivial isomorphisms of $I_{w}$-modules
\begin{align}\label{ordisom} A_{\bf{f}}^{(w)} \cong A_{\bf{f}}^{(1)}
\cong L/\mathcal{O}.\end{align}  
Suppose now that we consider the prime $\mathfrak{p} \subset 
\mathcal{O}_F$. Recall we assume that ${\bf{f}}$ is 
$\mathfrak{p}$-ordinary, in which case it is known that there is an 
exact sequence of $I_{\mathfrak{p}}$-modules 
\begin{align}\label{ordsequencep} 0 \longrightarrow
A_{\bf{f}}^{(\mathfrak{p})} \longrightarrow A_{\bf{f}}
\longrightarrow A_{\bf{f}}^{(1)} \longrightarrow 0,
\end{align} where $I_{\mathfrak{p}}$ acts on $A_{\bf{f}}^{(\mathfrak{p})}$ by the 
cyclotomic character $\varepsilon_{p}:G_F \longrightarrow \Aut(\mu_{p^{\infty}})$ 
times a multiplicative factor of $\pm 1$.

\begin{definition} Given a prime $w \mid \mathfrak{p}\mathfrak{N}^{-} 
\subset \mathcal{O}_F$, we define the {\it{ordinary structure}} 
$H^{1}_{\ord}(K_{\mathfrak{p}^{\infty}, \mathfrak{q}},A_{{\bf{f}},n}) 
\subset H^1(K_{\mathfrak{p}^{\infty}, \mathfrak{q}},A_{{\bf{f}},n})$
as follows.

\begin{itemize}
\item[(i)] If $w \mid \mathfrak{N}^{-} \subset \mathcal{O}_F$, then
it is the unramified cohomology
\begin{align*} H^{1}_{\ord}(K_{\mathfrak{p}^{\infty}, \mathfrak{q}},
A_{{\bf{f}},n}) = H^{1}(K_{\mathfrak{p}^{\infty}, \mathfrak{q}},
A_{{\bf{f}},n}^{(w)}).\end{align*}
\item[(ii)] At the prime $\mathfrak{p} \subset \mathcal{O}_F$, 
\begin{align*} H^{1}_{\ord}(K_{\mathfrak{p}^{\infty},\mathfrak{p}}, A_{{\bf{f}},n})
= \res_{\mathfrak{p}}^{-1} H^1\left( I_{K_{\mathfrak{p}^{\infty},
\mathfrak{p}}}, A_{{\bf{f}},n}^{(\mathfrak{p})}\right).\end{align*} Here,
$\res_{\mathfrak{p}}: H^1(K_{\mathfrak{p}^{\infty},\mathfrak{p}},
A_{{\bf{f}},n}) \longrightarrow H^1(I_{K_{\mathfrak{p}^{\infty},
\mathfrak{p}}}, A_{{\bf{f}},n})$ denotes the map induced from the
restriction at the prime above $\mathfrak{p}$ in
$K_{\mathfrak{p}^{\infty}}.$
\end{itemize}
\end{definition} Note that we do not define ordinary parts at primes 
$w \mid \mathfrak{N}^{+} \subset \mathcal{O}_F$, as 
these groups are seen easily to vanish by variant 
of the argument given in Corollary \ref{2.4/2.5} 
below (cf. \cite[5.2.2]{Lo2}). Note as well that we may also 
define ordinary cohomology groups for the $G_F$-modules 
$A_{\bf{f}}$ and $T_{\bf{f}}$ by taking the limits 
$(\ref{limits})$. \end{remark}

\begin{remark}[Admissible primes.] 

Here, we define the notation of an $n$-admissible prime with respect to 
$\ff$, for $n \geq 1$ an integer. As we shall see below in Proposition \ref{3.2}, 
the set of $n$-admissible primes controls the Selmer group $\Sel(\ff, K_{\mathfrak{p}^{\infty}})$.

\begin{definition} A prime $v \subset \mathcal{O}_F$ is siad to be {\it{$n$-admissible with 
respect to ${\bf{f}}$ in $K$}} for some integer $n\geq 1$ if
\begin{itemize}
\item[(i)] $v \nmid \mathfrak{p}\mathfrak{N}$.
\item[(ii)] $v$ is inert in $K$.
\item[(iii)] $\mathfrak{P}$ does not divide ${\bf{N}}(v)^2 -1.$
\item[(iv)] $\mathfrak{P}_n$ divides one of ${\bf{N}}(v) + 1 - a_v({\bf{f}})$ or
${\bf{N}}(v) + 1 + a_v({\bf{f}})$.
\end{itemize}
\end{definition} We shall use the following two facts repeatedly 
throughout.
\begin{itemize}
\item[(1)] If $v \subset \mathcal{O}_F$ is an $n$-admissible prime 
with respect to ${\bf{f}}$, then the associated mod $\mathfrak{P}^n$ Galois
representation $T_{{\bf{f}},n}$ is unramified at $v$. Moreover, the
arithmetic Frobenius at $v$ acts semisimply on
$T_{{\bf{f}},n}$ with eigenvalues $1$ and $ {\bf{N}}(v)^2,$ both of
which are distinct mod $\mathfrak{P}^n$.

\item[(2)] If $v \subset \mathcal{O}_F$ is an $n$-admissible prime with 
respect to ${\bf{f}}$, then by condition (ii) it is inert in $K$.
We commit an abuse of notation in writing $v$ to also denote the 
prime above $v$ in $K$. Hence, $K_v$ denotes the localization at 
the prime above $v$ in $F$, which is isomorphic to the quadratic 
unramified extension of $F_v$. Writing $F_{v^2}$ to denote the 
quadratic unramified extension of $F_v$, we shall then always 
make the implicit identification $K_v \cong F_{v^2}$. \end{itemize} \end{remark}

\begin{remark}[Some identifications.]

Here, we give some identifications for the finite, singular and
ordinary structures of the local Galois cohomology groups defined above.
The results here are analogous to the case of $F={\bf{Q}}$ (cf. \cite[$\S$2]{BD}).

\begin{proposition}\label{localtate}
Let $v \subset \mathcal{O}_F$ be a finite prime. 
\begin{itemize} \item[(i)] If $v \nmid \mathfrak{N}$, 
then the cohomology groups 
$\widehat{H}^{1}_{\fin}(K_{\mathfrak{p}^{\infty},
v}, T_{{\bf{f}},n})$ and $H^1_{\fin}(K_{\mathfrak{p}^{\infty}, v},
A_{{\bf{f}},n})$ annihilate each other under the local Tate pairing
$\langle~,~\rangle_v$.
\item[(ii)] If $v\mid\mathfrak{N}$ with $\ord_v(\mathfrak{N}) = 1$, 
then the cohomology groups $\widehat{H}^{1}_{\ord}
(K_{\mathfrak{p}^{\infty}, v},T_{{\bf{f}},n})$ and $H^1_{\ord}(K_{\mathfrak{p}^{\infty}, v},
A_{{\bf{f}},n})$ annihilate each other under the local Tate pairing
$\langle~,~\rangle_v$.
\end{itemize} In particular, the local cohomology groups
$\widehat{H}^{1}_{\sing}(K_{\mathfrak{p}^{\infty}}, T_{{\bf{f}},n})$
and $H^1_{\fin}(K_{\mathfrak{p}^{\infty}}, A_{{\bf{f}},n})$ are
Pontryagin dual to each other.\end{proposition}

\begin{proof}
The result over finite layers $K_{\mathfrak{p}^m,v}$ is a standard
consequence of local Tate duality, see \cite[Proposition 2.3]{BD} or
\cite[$\S 2.1$]{BD0}. Passage to limits then proves the claim.
 \end{proof} Let us from now on write \begin{align}\label{ltp}\langle ~,~\rangle_v:
\widehat{H}^{1}_{\sing}(K_{\mathfrak{p}^{\infty}, v},
T_{{\bf{f}},n}) \times H^1_{\fin}(K_{\mathfrak{p}^{\infty}, v},
A_{{\bf{f}},n}) &\longrightarrow  {\bf{Q}}_p / {\bf{Z}}_p \end{align} to
denote the perfect pairing induced by the local Tate duality 
of Proposition \ref{localtate}. 

\begin{lemma}\label{cft}
Let $v \nmid p$ be any finite prime of $\mathcal{O}_F$.
If $v$ is inert in $K$, then $v$ splits completely in
$K_{\mathfrak{p}^m}$ for any integer $m \geq 1$. \end{lemma}

\begin{proof} This is a standard consequence of global 
class field theory. See for instance \cite[Proposition 2.3]{Ta},
using that $K_{\mathfrak{p}^m}$ is of generalized dihedral type
over $F$.  \end{proof}

\begin{corollary}\label{2.4/2.5}
Let $v \nmid \mathfrak{N}$ be a finite prime of $\mathcal{O}_F$.
\begin{itemize}
\item[(i)] If $v$ splits in $K$, then we have identifications
\begin{align*} \widehat{H}^{1}_{\sing}(K_{\mathfrak{p}^{\infty}, v}, 
T_{{\bf{f}},n}) &=H^1_{\fin}(K_{\mathfrak{p}^{\infty}, v}, A_{{\bf{f}},n}) 
= 0.\end{align*}
\item[(ii)] If $v$ is inert in $K$ with $v \nmid \mathfrak{N}$,
then we have identifications
\begin{align*} \widehat{H}^{1}_{\sing}(K_{\mathfrak{p}^{\infty}, v},
T_{{\bf{f}},n}) &\cong H^1_{\sing}(K_v, T_{{\bf{f}},n})
\otimes \Lambda,\\ H^{1}_{\fin}(K_{\mathfrak{p}^{\infty}, v}, A_{{\bf{f}},n})
&\cong \operatorname{Hom}\left( H^1_{\sing}(K_v,
T_{{\bf{f}},n}) \otimes \Lambda,{\bf{Q}}_p/{\bf{Z}}_p\right).\end{align*}
\end{itemize}\end{corollary}

\begin{proof} The first assertion is shown in \cite[Lemma 2.4]{BD}.
That is, it suffices by nondegeneracy of the local Tate pairing
(\ref{ltp}) to show that $H^1_{\fin}(K_{\mathfrak{p}^{\infty}, v},
A_{{\bf{f}},n})$ vanishes. If $v\mathcal{O}_K = v_1 v_2$, then 
a similar application of global class field theory as used in
Lemma \ref{cft} shows that that the Frobenius at each $v_i$ 
is the topological generator of a finite index subgroup 
of $G_{\mathfrak{p}^{\infty}} \cong {\bf{Z}}_p^{\delta}$. We can
then view $K_{\mathfrak{p}^{\infty}, v}$ as the direct sum of 
copies of the maximal unramified $p$-extension of $F_v$. 
Since $A_{{\bf{f}},n} = A_{\bf{f}}[\mathfrak{P}^n]$
has exponent $\mathfrak{P}^n$, we deduce that any unramified
class in $H^1(K_{\mathfrak{p}^m,v}, A_{{\bf{f}},n})$ must have trivial
restriction to $H^1(K_{\mathfrak{p}^{m'}, v}, A_{{\bf{f}},n})$ for
$m'$ sufficiently large. Hence, $H^1_{\fin}(K_{\mathfrak{p}^{\infty}, v}, 
A_{{\bf{f}},n})=0$. The second assertion is shown in \cite[Lemma 2.5]{BD}. 
That is, since $v$ splits completely in $K_{\mathfrak{p}^{\infty}}$ by Lemma \ref{cft},
any choice of prime $\mathfrak{v}_m$ above $v$ in $K_{\mathfrak{p}^m}$ determines
an isomorphism \begin{align*} H^1(K_{\mathfrak{p}^m,v}, T_{{\bf{f}},n}) 
&\longrightarrow H^1(K_v,T_{{\bf{f}},n}) \otimes \mathcal{O}[G_{\mathfrak{p}^m}].\end{align*}
A compatible system of choices of primes $\mathfrak{v}_m$ above $v$ in
$K_{\mathfrak{p}^{\infty}}$ then determines an isomorphism
\begin{align*} \widehat{H}^1(K_{\mathfrak{p}^{\infty}, v}, T_{{\bf{f}},n}) 
&\cong H^1(K_v,T_{{\bf{f}},n}) \otimes \Lambda.\end{align*} 
Passage to the singular cohomology then proves the claim, with the latter
isomorphism being a consequence of Proposition \ref{localtate}.
 \end{proof}

\begin{lemma}\label{2.6/2.7}

If $v \subset \mathcal{O}_F$ is an $n$-admissible prime with respect
to ${\bf{f}}$, then the local cohomology groups $H^1_{\sing}(K_v,
T_{{\bf{f}},n})$ and $H^1_{\fin}(K_v, A_{{\bf{f}},n})$ are both
isomorphic to $\mathcal{O}/\mathfrak{P}^n$. Moreover, the local
cohomology groups $\widehat{H}^1_{\sing}(K_{\mathfrak{p}^{\infty},
v}, T_{{\bf{f}},n})$ and $H^1_{\fin}(K_{\mathfrak{p}^{\infty}, v},
A_{{\bf{f}},n})$ are both free of rank one over
$\Lambda/\mathfrak{P}^n.$
\end{lemma}

\begin{proof} The first assertion follows from the same proof as
given in \cite[Lemma 2.6]{BD}, using the identification
$H_{\sing}^1(K_v, T_{{\bf{f}},n}) = H^1(I_{K_v}, T_{{\bf{f}},n}
)^{G_{K_v}}$ along with the fact that $T_{{\bf{f}},n}$ is unramified
at $v$. The second assertion then follows
directly from the second part of Lemma \ref{2.4/2.5} above (cf.
\cite[Lemma 2.7]{BD}). 
\end{proof}\end{remark}

\begin{remark}[Residual Selmer groups.] 

Recall that we defined the residue maps $\partial_v$ on local 
cohomology to be the natural restriction maps
\begin{align*}\partial_v: H^1(K_{\mathfrak{p}^{\infty},v}, A_{ {\bf{f}},n}) 
&\longrightarrow H^1_{\sing}(K_{\mathfrak{p}^{\infty},v},  A_{ {\bf{f}},n}) \\
\partial_v: \widehat{H}^1(K_{\mathfrak{p}^{\infty},v}, T_{ {\bf{f}},n}) 
&\longrightarrow \widehat{H}^1_{\sing}(K_{\mathfrak{p}^{\infty},v},  
T_{ {\bf{f}},n}). \end{align*} Let us commit an abuse 
of notation in also writing $\partial_v$ to denote the 
composition of these maps with the restriction from 
$K_{\mathfrak{p}^{\infty}}$ to $K_{\mathfrak{p}^{\infty},v}$,
which gives residue maps on global cohomology 
\begin{align*}\partial_v: H^1(K_{\mathfrak{p}^{\infty}}, A_{ {\bf{f}},n}) 
&\longrightarrow H^1_{\sing}(K_{\mathfrak{p}^{\infty},v},  A_{ {\bf{f}},n}) \\
\partial_v: \widehat{H}^1(K_{\mathfrak{p}^{\infty}}, T_{ {\bf{f}},n}) 
&\longrightarrow \widehat{H}^1_{\sing}(K_{\mathfrak{p}^{\infty},v},  
T_{ {\bf{f}},n}). \end{align*} Let us establish for future 
reference the following notations:
\begin{itemize}
\item If $\partial_v(c) = 0$ for a class $c \in
\widehat{H}^{1}(K_{\mathfrak{p}^{\infty}}, T_{{\bf{f}},n})$, then
$\vartheta_v(c)$ denotes the image of $c$ in
$\widehat{H}^{1}_{\fin}(K_{\mathfrak{p}^{\infty}, v},
T_{{\bf{f}},n})$.

\item If $\partial_v(c) = 0$ for a class $c
\in H^{1}(K_{\mathfrak{p}^{\infty}}, A_{{\bf{f}},n})$, then
$\vartheta_v(c)$ denotes the image of $c$ in
$H^{1}_{\fin}(K_{\mathfrak{p}^{\infty}, v}, A_{{\bf{f}},n})$.
\end{itemize} Recall that the integer factorization 
$\mathfrak{N} = \mathfrak{p}\mathfrak{N}^{+}\mathfrak{N}^{-}
\subset \mathcal{O}_F$ of $(\ref{N})$ is assumed. 
Let us write $s_v$ denote the image of a class $s$ under 
the restriction from $K_{\mathfrak{p}^{\infty}} $ to $K_{\mathfrak{p}^{\infty}, v}$.

\begin{definition}
The {\it{residual Selmer group}}
$\Sel_{{\bf{f}},n}(K_{\mathfrak{p}^{\infty}})$ associated to
$({\bf{f}}, n, K_{\mathfrak{p}^{\infty}})$ is defined to be the
group of classes $s \in H^1(K_{\mathfrak{p}^{\infty}},
A_{{\bf{f}},n})$ such that
\begin{itemize}
\item[(i)] The residue $\partial_v(s)$ vanishes at all primes
$v \nmid \mathfrak{N}$.
\item[(ii)] The restriction $s_v$ is ordinary at all primes
$v \mid \mathfrak{p}\mathfrak{N}^{-}$.
\item[(iii)] The restriction $s_v$ is trivial at all primes
$v \mid \mathfrak{N}^{+}$.
\end{itemize}\end{definition} Observe that the residual Selmer group 
$\Sel_{{\bf{f}},n}(K_{\mathfrak{p}^{\infty}})$ depends only on the 
mod $\mathfrak{P}^n$ Galois representation $T_{ {\bf{f}}, n }$ associated to 
${\bf{f}}$, and not on $T_{\bf{f}}$ itself! \end{remark}

\begin{remark}[Compactified Selmer groups.] We now define compactified
Selmer groups.

\begin{definition} Let $\mathfrak{S} \subset \mathcal{O}_F$ be
any integral ideal prime to $\mathfrak{N}$. The {\it{compactified
Selmer group}} $\widehat{H}^{1}_{\mathfrak{S}}(K_{\mathfrak{p}^{\infty}},
T_{{\bf{f}},n})$ associated to $({\bf{f}}, n,
K_{\mathfrak{p}^{\infty}})$ is defined to be the group of classes
$\mathfrak{s} \in \widehat{H}^1(K_{\mathfrak{p}^{\infty}},
T_{{\bf{f}},n})$ such that
\begin{itemize}
\item[(i)] The residue $\partial_v(\mathfrak{s})$ vanishes at all primes
$v \nmid \mathfrak{S}\mathfrak{N}$.
\item[(ii)] The restriction $\mathfrak{s}_{v}$ is ordinary at all primes
$v \mid \mathfrak{p}\mathfrak{N}^{-}$.
\item[(iii)] The restriction $\mathfrak{s}_{v}$ is arbitrary at all primes $v \mid
\mathfrak{S}\mathfrak{N}^{+}$ .\end{itemize}
\end{definition}\end{remark} 

\begin{remark}[Admissible sets.] Let us also for future reference define
the notion of an $n$-admissible set with respect to ${\bf{f}}$.

\begin{definition}
A finite set of primes $\mathfrak{S}$ of $\mathcal{O}_F$ is said to
be {\it{$n$-admissible with respect to ${\bf{f}}$}} if
\begin{itemize}
\item[(i)] Each prime $v\in \mathfrak{S}$ is $n$-admissible with respect to ${\bf{f}}$.
\item[(ii)] The natural map $\Sel_{{\bf{f}},n}(K) \longrightarrow
\bigoplus_{v\mid \mathfrak{S}}H^{1}_{\fin}(K_v, A_{{\bf{f}},n})$ is
injective.
\end{itemize}\end{definition}

\begin{theorem}\label{3.3}
If $\mathfrak{S}$ in an $n$-admissible set of primes of
$\mathcal{O}_F$ with respect to ${\bf{f}}$, then
$\widehat{H}^{1}_{\mathfrak{S}}(K_{\mathfrak{p}^{\infty}},
T_{{\bf{f}},n})$ is free of rank $\vert \mathfrak{S} \vert$ over
$\Lambda/\mathfrak{P}^n.$
\end{theorem}

\begin{proof} See \cite[Theorem 3.3]{BD}. The proof in the more
general setting follows in the same way from \cite[Theorem
3.2]{BD5}, which is given for arbitrary abelian extensions over $K$.

\end{proof} Now, recall that the Galois group $G_{\mathfrak{p}^{\infty}}$ is topologically
isomorphic to ${\bf{Z}}_p^{\delta}$ with $\delta = [F_{\mathfrak{p}}: {\bf{Q}}_p],$
and hence pro-$p$. Hence, the Iwasawa algebra 
$\Lambda$ is a local ring of dimension $\delta +1$. Let $\mathfrak{m}_{\Lambda}$ 
denote the maximal ideal of $\Lambda$. We have the following result.

\begin{theorem} \label{3.4} ~
\begin{itemize}
\item[(i)] The natural map $H^1(K, A_{{\bf{f}},1}) \longrightarrow
H^1(K_{\mathfrak{p}^{\infty}},A_{{\bf{f}},1})[\mathfrak{m}_{\Lambda}]$ induced by
restriction is an isomorphism.
\item[(ii)] If $\mathfrak{S}$ is an $n$-admissible set of primes with respect
to ${\bf{f}}$, then the natural map $H^1(K, T_{{\bf{f}},1})
\longrightarrow H^1(K_{\mathfrak{p}^{\infty}},
T_{{\bf{f}},1})/\mathfrak{m}_{\Lambda}$ induced by
corestriction is an injection.
\end{itemize}
\end{theorem} 

\begin{proof} See \cite[Theorem 3.4]{BD}. The proof given there
carries over here with the same argument by using Proposition
\ref{3.3} above. 
\end{proof}\end{remark}

\begin{remark}[Relations between Selmer groups.] 

Let us start with some motivation. Our method of approach to dihedral main
conjectures generalizes the Euler system argument of Bertolini-Darmon
\cite{BD}. More precisely, it generalizes the refinement of this argument 
given by Pollack-Weston in \cite{PW}. As such, it requires the construction 
of classes in $\widehat{H}_{ v}^1(K_{\mathfrak{p}^{\infty}}, T_{{\bf{f}}, n})$ 
indexed by $n$-admissible primes $v \subset \mathcal{O}_F$ with respect to
${\bf{f}}$. The residues of these classes can be related
to their corresponding group ring elements 
$\mathcal{L}_{\Phi} \in \Lambda/\mathfrak{P}^n$ via the first and second explicit
reciprocity laws introduced above (Theorems \ref{erl1} and \ref{erl2} below). 
Here, $\Phi$ denotes the mod $\mathfrak{P}^n$ quaternionic eigenform 
corresponding to ${\bf{f}}$ mod $\mathfrak{P}^n$ under the 
Jacquet-Langlands correspondence. Recall that we write 
$L_{\mathfrak{p}}({\bf{f}}, K_{\mathfrak{p}^{\infty}})$  to denote 
the associated $\mathfrak{p}$-adic $L$-function 
$\mathcal{L}_{\mathfrak{p}}(\Phi, K)= \mathcal{L}_{\Phi}
\mathcal{L}_{\Phi}^{*} \in \Lambda/\mathfrak{P}^n$.
The explicit reciprocity laws, which a priori only give relations in the compactified 
Selmer group $\widehat{H}_{ v}^1(K_{\mathfrak{p}^{\infty}}, T_{{\bf{f}}, n})$, in 
fact give relations in the dual residual Selmer group
\begin{align*}\Sel_{{\bf{f}},n}(K_{\mathfrak{p}^{\infty}})^{\vee} = \
\Hom(\Sel_{{\bf{f}},n}(K_{\mathfrak{p}^{\infty}}),
{\bf{Q}}_p/{\bf{Z}}_p) \end{align*} thanks to the following result.

\begin{proposition}\label{cftpairing}
If $\mathfrak{s} \in \widehat{H}^{1}_v(K_{\mathfrak{p}^{\infty}},
T_{{\bf{f}},n})$, then for all $s \in
\Sel_{{\bf{f}},n}(K_{\mathfrak{p}^{\infty}}),$
\begin{align*}\langle
\partial_v\left(\mathfrak{s}\right), \vartheta_v(s) \rangle_v=
0.\end{align*}
\end{proposition}

\begin{proof} By direct generalization of \cite[Proposition
2.10]{BD}. That is, fix classes $\mathfrak{s}$ and $s$ as above. The
global reciprocity law of class field theory implies that
\begin{align} \label{rlgcft}\sum_{v} \langle \partial_v(\mathfrak{s}),
\vartheta_v(s) \rangle_v= 0.\end{align} Here, the sum runs over all
finite primes $v \subset \mathcal{O}_F$. Let $\mathfrak{S} \subset
\mathcal{O}_F$ be any integral ideal prime to $\mathfrak{N}$. If
$\mathfrak{s} \in \widehat{H}^{1}_{\mathfrak{S}}(K_{\mathfrak{p}^{\infty}},
T_{{\bf{f}},n})$ and $s \in \Sel_{{\bf{f}},n}(K_{\mathfrak{p}^{\infty}}),$ 
then $\langle \partial_v(\mathfrak{s}), \vartheta_v(s) \rangle_v=0$ for all
$v \nmid \mathfrak{S}$ by local conditions defining these groups,
and $\partial_v\left( s \right) = 0 $ for all $v \mid \mathfrak{S}$.
It follows from (\ref{rlgcft}) that \begin{align*}\sum_{v \mid \mathfrak{S}} 
\langle \partial_v(\mathfrak{s}), \vartheta_v(s)\rangle_v &= 0.\end{align*} 
Taking $\mathfrak{S} = v$ then proves the claim. 
\end{proof} Finally, we make the following

\begin{definition} Let \begin{align}\label{selmer}
\Sel({\bf{f}}, K_{\mathfrak{p}^{\infty}}) = \varinjlim_n
\Sel_{{\bf{f}},n}(K_{\mathfrak{p}^{\infty}}),\end{align} where the
limits are taken with respect to those in $(\ref{limits}).$ We claim
as before that these identifications can be justified, for instance
by \cite[Theorem 2.2]{Ta2} (cf. also \cite[Proposition 3.6]{PW}). \end{definition}\end{remark}

\section{Control theorems}

We shall use the following results to prove the main conjecture
divisibility $(\ref{mainconjecture})$.

\begin{remark}[The Fitting ideals criterion.]

Given $R$ a ring, and $X$ a finitely-presented $R$-module,
let $\Fitt_R(X)$ denote the Fitting ideal of $X$ over $R$. We refer
the reader to \cite[Appendix]{MW} for instance for definitions and background
on Fitting ideals. 

\begin{proposition} \label{3.1} Suppose that $X$ is a finitely-generated 
$\Lambda$-module, and that $\mathcal{L}$ is an element of $\Lambda$. 
If $\varphi(\mathcal{L}) \subset \Fitt_{\mathcal{O}'}\left( 
X \otimes_{\varphi} \mathcal{O}'\right)$ for all homomorphisms 
$\varphi: \Lambda \longrightarrow \mathcal{O}'$ with $\mathcal{O}'$ 
any discrete valuation ring, then $\mathcal{L} \subset 
\operatorname{char}_{\Lambda}\left(X\right).$\end{proposition}

\begin{proof} See \cite[Proposition 3.1]{BD}, which proves the
claim for the case of $F={\bf{Q}}$ (i.e. with $\delta =1$), and 
\cite[Proposition 7.4]{L30} for the general case. \end{proof}\end{remark}

\begin{remark}[Control of Selmer.] 

Recall that for each finite prime $v \nmid \mathfrak{N} \subset \mathcal{O}_F$ not 
dividing the residue characteristic of $\mathfrak{p}$, we have a natural residue map 
$\partial_v: H^1(K_v, A_{{\bf{f}},1}) \longrightarrow H^1_{\sing}(K_v, A_{{\bf{f}}, 1})$.
Recall as well that we commit a minor abuse of notation in also writing $\partial_v$ to denote 
the composition of maps \begin{align*} H^1(K, A_{{\bf{f}}, 1}) \longrightarrow 
H^1(K_v, A_{{\bf{f}},1}) \longrightarrow H^1_{\sing}(K_v, A_{{\bf{f}}, 1}). \end{align*}

\begin{theorem}\label{3.2} Given a nonzero class
$s \in H^1(K, A_{{\bf{f}}, 1})$, there exist infinitely many
$n$-admissible primes $v \subset \mathcal{O}_F$ with respect to
${\bf{f}}$ in $K$ such that $\partial_v(s) = 0$ and $\vartheta_v(s)\neq 0$.
\end{theorem}

\begin{proof} By direct generalization of \cite[Theorem 3.2]{BD}. That is, 
fix a class $s \in H^1(K, A_{{\bf{f}},n})$. Let $F(A_{{\bf{f}}, n})$ denote the 
extension of $F$ fixed by the kernel of the $G_F$-representation 
$A_{{\bf{f}}, n}$. Let $L$ denote the compositum extension $K F(A_{{\bf{f}}, n})$.
Since we assume that the relative discriminant $\mathfrak{D}_{K/F}$ 
is prime to the level $\mathfrak{N}$, we claim that 
the extensions $F(A_{{\bf{f}}, n})$ and $K$ are 
linearly disjoint over $F$. Granted this property, we obtain the following 
description of the Galois group 
of $L$ over $F$: \begin{align*} \Gal(L/F) &= \Gal(K/F) 
\times \Gal(F(A_{{\bf{f}},n})/F) \\
&\subseteq \lbrace {\bf{1}}, \tau \rbrace \times
\Aut_{\mathcal{O}/\mathfrak{P}^n}(A_{{\bf{f}},n}).
\end{align*} Here, $\tau \in \Gal(K/F)$ denotes the complex conjugation
automorphism. Hence, any element of $\Gal(L/F)$ can be
written as a pair $(\tau^j, T)$, with $j \in \lbrace 0,1 \rbrace
$ and $T \in \Aut_{\mathcal{O}_F/\mathfrak{P}^n}(A_{{\bf{f}},n}).$
Let $\overline{s}$ denote the image of $s$ under restriction to the cohomology group
\begin{align*} H^1(L, A_{{\bf{f}}, 1}) = \Hom(\Gal(\overline{{\bf{Q}}}/L), 
A_{{\bf{f}}, 1}).\end{align*} Let $L_{s}$ denote the extension of $L$ cut out by this class. 
Assume without loss of generality that $s$ belongs to a fixed 
eigenspace for the action of the complex conjugation automorphism $\tau$. 
Let us then write $\varpi$ to denote the eigenvalue of $\tau$ acting on $s$, so 
that we have the relation $\tau \cdot s = \varpi s,$ where $ \varpi \in \lbrace \pm 1 \rbrace $.
It follows from this assumption that $L_s$, a priori only Galois over $K$, 
is in fact Galois over $F$. Moreover, since $A_{{\bf{f}},1}$ is an irreducible 
$G_F$-module by Hypothesis \ref{galrep} (iii),  we can and will make the 
following identification: 
\begin{align*}\Gal(L_s/F) = A_{{\bf{f}},1} \rtimes \Gal(L/F).\end{align*}
Here, $\Gal(L/F)$ acts on the normal abelian subgroup
$A_{{\bf{f}},1}$ by the rule \begin{align}\label{rule} (\tau^j, T)(a) 
&= \varpi^j \overline{T}a, \end{align} where $a$ denotes an element 
of $ A_{{\bf{f}}, 1}$, and $\overline{T}$ denotes the image of $T$ in 
$\Aut_{\mathcal{O}/\mathfrak{P}}(A_{{\bf{f}}, 1})$. Since the image of 
$\overline{\rho}_{\bf{f}}$ contains $\SL({\bf{F}}_p)$ by Hypothesis 
\ref{galrep}, we can and will identify $\Aut_{\mathcal{O}/\mathfrak{P}}(A_{{\bf{f}},1})$ with 
$\SL({\bf{F}}_p)$. We deduce from this description that $\Gal(L_s/F)$ 
contains at least one element $(a, \tau, T)$ such that the following
conditions hold:

\begin{itemize}
\item[1.] The automorphism $T$ has distinct eigenvalues 
$\varpi$ and $\lambda$, where the eigenvalue 
$\lambda$ lies in $\left( \mathcal{O} / \mathfrak{P}\right)^{\times}$, 
has order prime to $p$, and satisfies the property that 
${\bf{N}}(\lambda)$ is not congruent to $ \pm 1$ mod $ p$. 
\item[2.] The vector $a \in A_{{\bf{f}}, 1}$ belongs to the $\varpi$-eigenspace 
for the action of $\overline{T}$.\end{itemize} 
Let us now take $v \nmid \mathfrak{N}$ to be {\it{any}} prime of $F$
that is unramified in the extension $L_s$, with the additional 
condition that \begin{align}\label{frob} \Frob_v(L_s/F) &= (a, \tau, T). \end{align}
Observe that infinitely many such primes exist by the Cebotarev density theorem. 
We deduce from $(\ref{frob})$ that $ \Frob_v(L/F) = (\tau, T)$, and in particular 
that $v$ is $n$-admissible with respect to ${\bf{f}}$. We now argue 
that $\vartheta_v(s) \neq 0$. To see this, fix a prime $\mathfrak{v}$ above
$v$ in $L$. Let $e$ denote the degree of the corresponding residue field. Note 
that $e$ is necessarily even, as $L_{\mathfrak{v}}$ contains the quadratic unramified 
extension of $F_v$. Using $(\ref{rule})$ along with condition $2.$ 
for $(a, \tau, T)$, we find that 
\begin{align*} \Frob_{\mathfrak{v}}(L_s/L) &= (a, \tau, T)^e 
= a + \varpi \overline{T}a  + \overline{T}^2a + \ldots \varpi 
\overline{T}^{e-1}a  = ea. \end{align*} Here, the addition symbol 
denotes group multiplication. Recall that we let $\overline{s}$ 
denote the image of $s$ in $H^1(L, A_{{\bf{f}}, 1}) = 
\Hom(\Gal(\overline{\bf{Q}}/L), A_{{\bf{f}},1})$ 
under restriction. Since $e$ is prime to $p$ by Hypothesis \ref{galrep} (i), we find that 
\begin{align*} \overline{s}\left(\Frob_v(L_s/L) \right) 
&= e \cdot \overline{s}(a) \neq 0. 
\end{align*} Hence, the restriction at $\mathfrak{v}$ of $\overline{s}$ 
does not vanish. Hence, $\vartheta_v(s)$ does not
vanish, as required.  \end{proof}\end{remark}

\section{The Euler system argument}

Let first describe the Euler system that we shall construct 
in the subsequent sections. This construction and subsequent argument will 
generalize those of Bertolini-Darmon \cite{BD}, or more specifically the refinements 
of these due to Pollack-Weston \cite{PW}. Fix an integer $n \geq 1$. Recall that we 
write $\mathcal{S}_2(\mathfrak{N}^{+}, \mathfrak{N}^{-})$ to denote the subspace of 
$\mathcal{S}_2(\mathfrak{N}^{+} \mathfrak{N}^{-})$ consisting of cuspforms that are 
new at all primes $v \subset \mathcal{O}_F$ dividing $\mathfrak{N}^{-}.$ Recall as well 
that we write ${\bf{T}}(\mathfrak{N}^{+}, \mathfrak{N}^{-})$ to denote 
the algebra of Hecke operators acting faithfully on 
$\mathcal{S}_2(\mathfrak{N}^{+}, \mathfrak{N}^{-})$, with 
${\bf{T}}_0(\mathfrak{N}^{+}, \mathfrak{N}^{-})$ its $p$-adic completion.
Let us now fix an eigenform ${\bf{f}} \in \mathcal{S}_2(\mathfrak{N}^{+}, 
\mathfrak{N}^{-})$. We shall use the theories of level raising congruences 
and CM points on Shimura curves over totally real fields to construct 
for each $n$-admissible prime $v \subset \mathcal{O}_F$ with respect to 
${\bf{f}}$ a class \begin{align}\label{zeta}\zeta(v) \in
\widehat{H}^{1}(K_{\mathfrak{p}^{\infty}}, T_{{\bf{f}},n}).\end{align} 
Observe that since $v$ is $n$-admissible, we have the decompositions 
\begin{align*} \widehat{H}^1(K_{\mathfrak{p}^{\infty}, v}, T_{{\bf{f}},n}) &=
\widehat{H}_{\fin}^1(K_{\mathfrak{p}^{\infty}, v}, T_{{\bf{f}},n})
\oplus \widehat{H}_{\ord}^1(K_{\mathfrak{p}^{\infty}, v},
T_{{\bf{f}},n}) \\ &= \widehat{H}_{\fin}^1(K_{\mathfrak{p}^{\infty},
v}, T_{{\bf{f}},n}) \oplus \widehat{H}_{\sing}^1
(K_{\mathfrak{p}^{\infty}, v}, T_{{\bf{f}},n}).\end{align*} 
Hence, we may view the homomorphism
\begin{align*} \vartheta_v: 
\widehat{H}^1(K_{\mathfrak{p}^{\infty}, v}, T_{{\bf{f}},n}) \longrightarrow
\widehat{H}^1(K_{\mathfrak{p}^{\infty}, v}, T_{{\bf{f}},n}) /
\widehat{H}_{\ord}^1(K_{\mathfrak{p}^{\infty}, v}, T_{{\bf{f}},n}) \end{align*} 
as a projection onto the first component of the first decomposition,
and the homomorphism 
\begin{align*}\partial_v: \widehat{H}^1(K_{\mathfrak{p}^{\infty}, v}, T_{{\bf{f}},n}) 
\longrightarrow \widehat{H}_{\sing}^1(K_{\mathfrak{p}^{\infty}, v}, 
T_{{\bf{f}},n})\end{align*} as a projection onto the second component of the second
decomposition. We shall deduce in subsequent sections the following
explicit reciprocity laws for the classes $(\ref{zeta})$.

\begin{theorem}\emph{(The first explicit reciprocity law).}\label{erl1}
Let ${\bf{f}} \in \mathcal{S}_2(\mathfrak{N}^{+}, \mathfrak{N}^{-})$ be a 
$\mathfrak{p}$-ordinary eigenform, as defined above. Assume that 
the conditions of Theorem \ref{raiseonefree} and Corollary \ref{galoisid} below are satsified. If 
$v \subset \mathcal{O}_F$ is an $n$-admissible prime with respect to 
${\bf{f}}$, then $\vartheta_v\left( \zeta(v)\right)=0$. Moreover, the equality
\begin{align}\label{ERL1} \partial_v\left( \zeta(v) \right) = \mathcal{L}_{\ff}
\end{align} holds in $\widehat{H}^1_{\sing}(K_{\mathfrak{p}^{\infty},v}, T_{{\bf{f}},n}) \cong
\Lambda/\mathfrak{P}^n,$ up to multiplication by
elements of $\mathcal{O}^{\times}$\ or $G_{\mathfrak{p}^{\infty}}$.
\end{theorem}

\begin{proof} See Theorem \ref{ERL1} below.\end{proof}
To state the second reciprocity law for these classes
$(\ref{zeta})$, we require the following weak 
level-raising result at two primes. That is, 
let $v_1$ and $v_2$ be two distinct $n$-admissible primes 
with respect to ${\bf{f}}$ such that 
\begin{align}\label{v_i} {\bf{N}}(v_i) + 1 - \varepsilon_i \cdot 
a_{v_i}({\bf{f}}) \equiv 0 \mod \mathfrak{P}_n \end{align} for each
of $i = 1,2$, where $\varepsilon_i \in \lbrace \pm 1 \rbrace$.

\begin{proposition}
\label{elr} Let ${\bf{f}} \in \mathcal{S}_2(\mathfrak{N}^{+}, \mathfrak{N}^{-})$ 
be a $\mathfrak{p}$-ordinary eigenform, as defined above. Assume that the 
conditions of Theorem \ref{raiseonefree} and Corollary \ref{galoisid} below are satisfied, 
and moreover that $F$ is linearly disjoint from the cyclotomic field ${\bf{Q}}(\zeta_p)$. There exists a 
mod $\mathfrak{P}_n$ eigenform ${\bf{g}}$ with respect to the Hecke algebra 
${\bf{T}}_0(\mathfrak{N}^{+}, v_1 v_2 
\mathfrak{N}^{-})$ such that the following congruences hold:

\begin{itemize}
\item[(i)] $T_w ({\bf{g}}) \equiv a_w({\bf{f}}) \cdot{\bf{g}} 
\mod \mathfrak{P}_n$ for all primes $w \nmid v_1 v_2 
\mathfrak{N}^{+}\mathfrak{N}^{-}$ of $ \mathcal{O}_F$.
\item[(ii)] $U_w ({\bf{g}}) \equiv a_w({\bf{f}}) \cdot{\bf{g}} 
\mod \mathfrak{P}_n$ for all primes 
$w \mid \mathfrak{N}^{+}\mathfrak{N}^{-}$ of $ \mathcal{O}_F$.
\item[(iii)] $U_{v_i}({\bf{g}}) \equiv \varepsilon_i \cdot{\bf{g}} \mod \mathfrak{P}_n$ 
for $i=1,2$.\end{itemize}\end{proposition} 

\begin{proof} See Proposition \ref{gammaeigenform} below. 
\end{proof} We then use this result to deduce the following

\begin{theorem}\emph{(The second explicit reciprocity law).} \label{erl2}
Keep the notations and hypotheses of Proposition \ref{elr}. The
equality \begin{align}\label{ERL2} \vartheta_{v_1}\left( \zeta(v_2) \right) =
\mathcal{L}_{{\bf{g}}} \end{align} holds in
$\widehat{H}^1(K_{\mathfrak{p}^{\infty},v_2}, T_{{\bf{f}},n}) \cong
\Lambda/\mathfrak{P}^n,$ up to multiplication by
elements of $\mathcal{O}^{\times}$ or $G_{\mathfrak{p}^{\infty}}$.
\end{theorem}

\begin{proof} See Theorem \ref{ERL2} below. 
\end{proof} Observe that since the choice of $n$-admissible primes $v_1$ and $v_2$ is
symmetric in Theorem \ref{erl2}, we obtain the following immediate

\begin{corollary}\label{erl2symm} The equality \begin{align}
\label{ERL2symm} \vartheta_{v_1}\left( \zeta(v_2) \right) =
\vartheta_{v_2} \left( \zeta(v_1) \right)\end{align} holds in
$\Lambda / \mathfrak{P}^n$, up to multiplication by
elements of $\mathcal{O}^{\times}$ or $G_{\mathfrak{p}^{\infty}}$.
\end{corollary}

\begin{remark}[The inductive argument.]

We now prove the main conjecture divisibility, assuming the existence of an 
Euler system of classes $(\ref{zeta})$ that satisfy the first and second explicit 
reciprocity laws (Theorems \ref{erl1} and \ref{erl2}). The arguments in this section
are essentially the same as those of \cite{PW} (based on those of 
\cite[$\S$4]{BD} but removing the unneccesary $p$-isolatedness hypothesis), which 
extend without much trouble to this setting.

Recall that ${\bf{T}}_0(\mathfrak{N}^{+}, \mathfrak{N}^{-})$ denotes the $p$-adic
completion of the Hecke algebra ${\bf{T}}(\N^+, \N^{-})$  acting on the space of 
cusp forms $\mathcal{S}_2(\mathfrak{N}^{+}, \mathfrak{N}^{-})$. 
Fix an integer $n \geq 1$. Let us now always ${\bf{f}} 
\in \mathcal{S}_2(\mathfrak{N}^{+}, \mathfrak{N}^{-})$
view as a homomorphism 
\begin{align*}\theta_{\bf{f}}:{\bf{T}}_0(\mathfrak{N}^{+}, \mathfrak{N}^{-})
&\longrightarrow \mathcal{O}_0/\mathfrak{P}_n\end{align*} 
in the natural way, by sending Hecke operators to their 
associated eigenvalues. We shall often commit an abuse of 
notation in writing ${\bf{f}}$ to denote this homomorphism $\theta_{\bf{f}}$. 

\begin{definition} Fix an $\mathcal{O}$-algebra homomorphism $\varphi:
\Lambda_{\mathcal{O}} \longrightarrow \mathcal{O}'$. Here, $\mathcal{O}'$ 
is any discrete valuation ring, with maximal ideal denoted by $\mathfrak{P}'$. 
Let $s_{\bf{f}}$ denote the $\mathcal{O}'$-length of 
$\Sel({\bf{f}}, K_{\mathfrak{p}^{\infty}})^{\vee} \otimes_{\Lambda}
\mathcal{O}'$. Let $2t_{\bf{f}}$ denote the $\mathcal{O}'$-valuation of 
$\varphi\left( \mathcal{L}_{\mathfrak{p}}({\bf{f}},K_{\mathfrak{p}^{\infty}})\right)$ 
in $\mathcal{O}'/\varphi(\mathfrak{P}')^n$, setting 
$2t_{\bf{f}} = \infty$ if 
$\varphi\left( \mathcal{L}_{\mathfrak{p}} ({\bf{f}}, K_{\mathfrak{p}^{\infty}}) \right)=0$. \end{definition} 

\begin{proposition}\label{4.3} Fix integers $n \geq 1$ and $t_0 \geq 0$. Let
$\widetilde{{\bf{f}}}$ be an $\mathcal{O}_0/ \mathfrak{P}_{n + t_0}$-valued eigenform 
for the completed Hecke algebra ${\bf{T}}_0(\mathfrak{N}^{+}, \mathfrak{N}^{-})$, 
with ${\bf{f}}$ its projection onto $\mathcal{O}_0/\mathfrak{P}_n$. Assume that

\begin{itemize}
\item[(i)] The homomorphism
 $\theta_{\bf{f}}: {\bf{T}}_0(\mathfrak{N}^{+}, \mathfrak{N}^{-}) 
\longrightarrow \mathcal{O}_0/\mathfrak{P}_n$ is surjective.
\item[(ii)] The first and second explicit reciprocity laws 
(Theorems \ref{erl1} and \ref{erl2}) hold.
\item[(iii)] We have the inequality $2t_{\bf{f}} < 2t_0$. 
\end{itemize} Then, we have the inequality $s_{\bf{f}} \leq 2t_{\bf{f}}$.
\end{proposition} Before getting to the proof, let us give the following

\begin{corollary}\label{dmc} Keep the notations and hypotheses of Proposition \ref{4.3}.
Then, the dual Selmer group $\Sel({\bf{f}},K_{\mathfrak{p}^{\infty}})^{\vee}$ is a torsion
$\Lambda$-module, hence has a characteristic power series 
$\operatorname{char}_{\Lambda_{\mathcal{O}}} \Sel({\bf{f}},
K_{\mathfrak{p}^{\infty}})^{\vee}. $ Moreover, there is an inclusion of ideals
\begin{align}\label{dmcd} 
\left( \mathcal{L}_{\mathfrak{p}}({\bf{f}}, K_{\mathfrak{p}^{\infty}})  \right) \subseteq 
\left( \operatorname{char}_{\Lambda}\Sel({\bf{f}},
K_{\mathfrak{p}^{\infty}})^{\vee} \right) ~\text{in $\Lambda$}.\end{align}
\end{corollary}

\begin{proof} Let $X = \Sel({\bf{f}}, K_{\mathfrak{p}^{\infty}})^{\vee}$. 
Observe that to show the divisibility $(\ref{dmcd})$, it suffices by
Proposition \ref{3.1} to show the containment $\varphi \left(
\mathcal{L}_{{\bf{f}}} \right) \in \Fitt_{\mathcal{O}'}(X)$, where
$\varphi: \Lambda \longrightarrow \mathcal{O}'$ is any
homomorphism, and $\mathcal{O}'$ any discrete valuation ring. Fix a
such a ring $\mathcal{O}'$ and homomorphism $\varphi:\Lambda
\longrightarrow \mathcal{O}'$. Observe that if $\varphi\left(
\mathcal{L}_{\mathfrak{p}}({\bf{f}}, K) \right)=0$, then
$\mathcal{L}_{\mathfrak{p}}({\bf{f}},K) \in \Fitt_{\mathcal{O}'}(X)$
trivially. If $\varphi\left( \mathcal{L}_{\mathfrak{p}}({\bf{f}}, K) \right)
\neq 0$, then let us take $t_0$ to be larger than the 
$\mathcal{O}'$-valuation of $\varphi\left( \mathcal{L}_{\mathfrak{p}}({\bf{f}},K)\right).$
Using Proposition \ref{4.3} for all $n \geq 0$, it follows that
$\varphi\left( \mathcal{L}_{\mathfrak{p}}({\bf{f}},K)\right) \in
\Fitt_{\mathcal{O}'}\left( X \right)$. Now, observe that once $(\ref{dmcd})$ is shown, the 
nonvanishing of the $\mathfrak{p}$-adic $L$-function $\mathcal{L}_{\mathfrak{p}}({\bf{f}} ,
K_{\mathfrak{p}^{\infty}})$ (deduced from \cite[Theorem 1.4]{CV}) implies that 
$\Sel({\bf{f}},K_{\mathfrak{p}^{\infty}})$ is $\Lambda$-cotorsion. 
The result follows.  \end{proof}

\begin{remark}[Proof of Proposition \ref{4.3}.]  Let us keep
all of the notations defined above. We start by defining the
following classes. Fix an $(n + t_{\bf{f}})$-admissible prime $v
\subset \mathcal{O}_F$ with respect to ${\bf{f}}$. Define from this
an $(n + t_{\bf{f}})$-admissible set $\mathfrak{S} = \lbrace v
\rbrace$ with respect to ${\bf{f}}$, and a cohomology class
$$\zeta(v)\in
\widehat{H}_{\mathfrak{S}}^1(K_{\mathfrak{p}^{\infty}}, T_{{\bf{f}},
n + t_{\bf{f}}})$$ as in $(\ref{zeta})$. Let $\zeta_{\varphi}'(v)$
denote the image of $\zeta(v)$ in
\begin{align}\label{tensormodule} H_{\mathfrak{S}}^{1}(K_{\mathfrak{p}^{\infty}},
T_{{\bf{f}}, t_{\bf{f}} +n}) \otimes_{\varphi}
\mathcal{O}'.\end{align} Note that $(\ref{tensormodule})$ is free of
rank $1$ over $\mathcal{O}'/\varphi(\mathfrak{P})^{n+ t_{\bf{f}} } $
by Theorem \ref{3.3}. Let \begin{align*}t = \ord_{\mathfrak{P}'}\left(
\zeta_{\varphi}'(v) \right).\end{align*} Since the residue map $\partial_v$ is a homomorphism,
Theorem \ref{erl1} implies that
\begin{align}\label{ineq1}
 t < \ord_{\mathfrak{P}'}\left( \partial_v(\zeta_{\varphi}'(v))\right) = t_{\bf{f}}. \end{align}
Let us now write $\xi_{\varphi}'(v)$ to denote an element of the module $(\ref{tensormodule})$ such that
\begin{align*} (\mathfrak{P}')^{t_{\bf{f}}} \cdot \xi_{\varphi}'(v) = \zeta_{\varphi}'(v). \end{align*}
Let $\xi_{\varphi}''(v)$ denote the image of this element $\xi_{\varphi}'(v)$ in $H_{\mathfrak{S}}^{1}(K_{\mathfrak{p}^{\infty}}, T_{{\bf{f}}, n} )
\otimes_{\varphi} \mathcal{O}'.$

\begin{lemma}\label{4.5/4.6}
The element $\xi''_{\varphi}(v)$ satisfies the following properties.
\begin{itemize}
\item[(1)] $\ord_{\mathfrak{P}'}( \xi_{\varphi}''(v)) = 0.$
\item[(2)]$\partial_w \left( \xi_{\varphi}''(v) \right) =0$
for all primes $w \mid v \mathfrak{N}^{+}$ in $\mathcal{O}_F$.
\item[(3)] $\vartheta_v\left( \xi_{\varphi}''(v) \right)=0$.
\item[(4)] $\ord_{\mathfrak{P}'}\left( \partial_v \left( \xi_{\varphi}''(v)
\right) \right) = t_{\bf{f}} - t$.
\item[(5)]$\partial_v\left( \xi_{\varphi}''(v)\right)$ lies in the kernel of
the natural surjection
$$\pi_v: \widehat{H}^1_{\sing}(K_{\mathfrak{p}^{\infty},v}, T_{{\bf{f}},n})
\longrightarrow \Sel_{{\bf{f}},n}(K_{\mathfrak{p}^{\infty}})^{\vee}
\otimes_{\varphi} \mathcal{O}'.$$
\end{itemize}
\end{lemma}

\begin{proof} See \cite[Lemmas 4.5 and 4.6]{BD}. Property $(1)$ follows from the
definition of $\xi_{\varphi}''(v),$ along with the fact that
$\ord_{\mathfrak{P}'} \left( \zeta_{\varphi}'(v)\right) = t$.
Property $(2)$ follows from the fact that $\zeta_{\varphi}'(v) \in
\widehat{H}^1_{\mathfrak{S}}(K_{\mathfrak{p}^{\infty}}, T_{{\bf{f}},
n+t_{\bf{f}} })$, using explicit definitions. Properties $(3)$ and
$(4)$ follow from Theorem \ref{erl1}. Property $(5)$ follows from
the same argument of \cite[Lemma 4.6]{BD}, which uses the global
reciprocity law of class field theory. .
\end{proof} Using Lemma \ref{4.5/4.6}, we can show the following

\begin{proposition}\label{4.7}
If $t_{\bf{f}}=0$, then
$\Sel_{{\bf{f}},n}(K_{\mathfrak{p}^{\infty}})^{\vee}$ is trivial.
\end{proposition}

\begin{proof} See \cite[Proposition 4.7]{BD}. If $t_{\bf{f}} = 0$, then
$\varphi \left( \mathcal{L}_{\bf{f}}\right)$ is a unit. Theorem
\ref{erl1} then implies that the residue $\partial_v\left( \zeta(v)
\right)$ generates $\widehat{H}^1_{\sing}(K_{\mathfrak{p}^{\infty},
v}, T_{{\bf{f}},n})$ for $v \subset \mathcal{O}_F$ any
$n$-admissible prime with respect to ${\bf{f}}$. Observe that this
renders the projective map $\pi_v$ in Lemma \ref{4.5/4.6}$(5)$
trivial. Let us now suppose that
$\Sel_{{\bf{f}},n}(K_{\mathfrak{p}^{\infty}})^{\vee}$ were not
trivial. Nakayama's lemma would then imply that $$\left(
\Sel_{{\bf{f}},n}(K_{\mathfrak{p}^{\infty}})
[\mathfrak{m}_{\Lambda}]\right)^{\vee} =
\Sel_{{\bf{f}},n}(K_{\mathfrak{p}^{\infty}})/\mathfrak{m}_{\Lambda}
\neq 0$$ Here, as above, $\mathfrak{m}_{\Lambda}$ denotes the
maximal ideal of $\Lambda$. We could then choose a
class $s \neq 0$ in $\Sel_{{\bf{f}},n}(K_{\mathfrak{p}^{\infty}})/
\mathfrak{m}_{\Lambda}$. By Theorem \ref{3.4}, we
could then identify this class $s$ with an element of $H^1(K,
A_{{\bf{f}},1})$. By Theorem \ref{3.2}, we could then choose another
$n$-admissible prime $q \subset \mathcal{O}_F$ with respect to
${\bf{f}}$ such that $\vartheta_q(s) \neq 0$. But observe then that
the projection $\pi_v$ cannot be trivial, by the nondegeneracy of
the local Tate pairing $\langle ~,~\rangle_v$. This supplies the
desired contradiction.  \end{proof} We are now ready to begin
a proof of Proposition \ref{4.3} by induction on $t_{\bf{f}}$. By
Proposition \ref{4.7}, we may assume without loss of generality that
$t_{\bf{f}} >0$. Let us write $\Pi_0$ to denote the set of $(n
+ t_0)$-admissible primes for which the valuation
$\ord_{\mathfrak{P}'}\left(\zeta_{\varphi}'(v) \right)$ is minimal.

\begin{lemma}\label{4.8}
Suppose that $t = \ord_{\mathfrak{P}'}\left(
\zeta_{\varphi}(v)\right)$ with $v \in \Pi_0$. Then, $t <
t_{\bf{f}}$.
\end{lemma}

\begin{proof} See \cite[Lemma 4.8]{BD}. Suppose otherwise that the claim did not hold.
Then, by $(\ref{ineq1})$, it would follow that
\begin{align*}\ord_{\mathfrak{P}'}\left( \zeta_{\varphi}'(v)\right)  =
\ord_{\mathfrak{P}'}\left( \varphi(\mathcal{L}_{\bf{f}})\right) \end{align*}
for all $(n+t_{\bf{f}})$-admissible primes $v \subset \mathcal{O}_F$
with respect to ${\bf{f}}$. By Theorem \ref{3.4}, we could then find
a nontrivial class $s$  in $H^1(K, A_{{\bf{f}},1}) \cap
\Sel_{{\bf{f}},n}(K_{\mathfrak{p}^{\infty}})$. By Theorem \ref{3.2},
we could then choose an $(n + t_{\bf{f}})$-admissible prime $v$ such
that $\vartheta_v(s) \neq 0$. Property $(4)$ of Lemma \ref{4.5/4.6}
implies that the natural image of $\vartheta_v\left(
\zeta_{\varphi}'(v) \right)$ in $H^1(K_v, T_{{\bf{f}},1})
\otimes_{\varphi} \mathcal{O}'$ does not vanish. Property $(5)$ of
Lemma \ref{4.5/4.6} implies that this image is orthogonal under the
local Tate pairing $\langle ~, ~\rangle_v$ to the nonvanishing class
$\vartheta_s(v)$, contradicting the fact that $\langle ~,
~\rangle_v$ is a perfect, nondegenerate pairing between the
$\mathcal{O}'/\mathfrak{P}'$-vector spaces to which these classes
belong.  \end{proof} Let us now fix a prime $v_1 \in \Pi_0$.
Let $s \in H^1(K, T_{\bf{f}}) \otimes \mathcal{O}'/\mathfrak{P}'$
denote the image of $\zeta_{\varphi}'(v)$ in \begin{align*}
\widehat{H}^1_{\mathfrak{S}}( K_{\mathfrak{p}^{\infty}}, T_{\bf{f}})
\otimes \mathcal{O}'/\mathfrak{P}' &\subset
\widehat{H}^1_{\mathfrak{S}}( K_{\mathfrak{p}^{\infty}},
T_{\bf{f}})/\mathfrak{m}_{\Lambda} \otimes
\mathcal{O}'/\mathfrak{P}'
\\ &\subset \widehat{H}^1_{\mathfrak{S}}(
K, T_{\bf{f}}) \otimes \mathcal{O}'/\mathfrak{P}'.\end{align*} By
Theorem \ref{3.2}, there exists an $(n + t_0)$-admissible prime
$v_2$ such that $\vartheta_{v_2}(s) \neq 0$. Here, $\vartheta_{v_2}:
\widehat{H}^1_{\mathfrak{S}}(K_{\mathfrak{p}^{\infty}, v_2},
T_{\bf{f}}) \longrightarrow
\widehat{H}^1_{\fin}(K_{\mathfrak{p}^{\infty}, v_2}, T_{\bf{f}}).$
Now, observe that we have the relations \begin{align}\label{ineq2} t
= \ord_{\mathfrak{P}'}\left( \zeta_{\varphi}'(v_1) \right) \leq
\ord_{\mathfrak{P}'}\left( \zeta_{\varphi}'(v_2) \right) \leq
\ord_{\mathfrak{P}'} \left( \vartheta_{v_1}(\zeta_{\varphi}(v_2)
\right) . \end{align} The first inequality follows from the
definition of $\Pi_0$. The second inequality follows from the fact
that $\vartheta_{v_2}$ is a homomorphism. Corollary \ref{erl2symm}
to the second explicit reciprocity law then gives us the relation
\begin{align}\label{eq3} \ord_{\mathfrak{P}'} \left(
\vartheta_{v_1}\left(\zeta_{\varphi}'(v_2)\right)\right) =
\ord_{\mathfrak{P}'} \left(
\vartheta_{v_2}\left(\zeta_{\varphi}'(v_1)\right)\right).
\end{align} Now, since $\vartheta_{v_2}(s) \neq 0$, we find that 
\begin{align*}\ord_{\mathfrak{P}'}
\left( \vartheta_{v_2} \left( \zeta_{\varphi}'(v_2) \right)\right) =
\ord_{\mathfrak{P}'}\left( \zeta_{\varphi}'(v_1) \right).\end{align*} It
follows that the inequalities of $(\ref{ineq2})$ are equalities. In
particular, \begin{align*}\ord_{\mathfrak{P}'} \left( \zeta_{\varphi}'(v_2)
\right) = t.\end{align*} Hence, we find that $v_2 \in \Pi_0$.

Let ${\bf{g}}$ denote the
$\mathcal{O}_0/\mathfrak{P}_{(n+t_0)}$-valued eigenform associated
to ${\bf{f}}$ and the pair of $n$-admissible primes $(v_1, v_2)$ with respect
to ${\bf{f}}$ by Theorem \ref{elr}. By
Theorem \ref{erl1}, we have that $$\vartheta_{v_2} \left(
\zeta_{\varphi}'(v_1) \right) = \mathcal{L}_{\bf{g}}.$$ It follows
that $t_{\bf{g}} = t < t_{\bf{f}}.$ Now, since ${\bf{g}}$ satisfies
all of the hypotheses of Proposition \ref{4.3}, we may apply the
inductive hypothesis to deduce that $s_{\bf{g}} \leq 2t_{\bf{g}}.$
We may now argue in the same way as \cite[pp. 34-35]{BD} to conclude
the argument. That is, let us write
$\Sel_{[v_1v_2]}(K_{\mathfrak{p}^{\infty}}) \subseteq
\Sel_{{\bf{f}},n}(K_{\mathfrak{p}^{\infty}})$ to denote the subgroup
of classes that are trivial at primes dividing $v_1 v_2$. Let
$\Sel_{v_1 v_2}^{\bf{f}}(K_{\mathfrak{p}^{\infty}})$ denote the
group defined by the exactness of the sequence
\begin{align}\label{sesf} 0 \longrightarrow
\Sel_{v_1v_2}^{{\bf{f}}}(K_{\mathfrak{p}^{\infty}}) \longrightarrow
\Sel_{{\bf{f}},n}(K_{\mathfrak{p}^{\infty}})^{\vee} \longrightarrow
\Sel_{[v_1 v_2]}(K_{\mathfrak{p}^{\infty}})^{\vee} \longrightarrow
0. \end{align} Observe that by applying local Tate duality 
(Proposition \ref{localtate}) to the natural inclusion
\begin{align*}\Sel_{v_1v_2}^{{\bf{f}}}(K_{\mathfrak{p}^{\infty}})^{\vee} \subseteq
H^{1}_{\fin}(K_{\mathfrak{p}^{\infty}, v_1}, A_{{\bf{f}},n}) \oplus
H^{1}_{\fin}(K_{\mathfrak{p}^{\infty}, v_2}, A_{{\bf{f}},n}), \end{align*} we
obtain a natural surjection \begin{align*}\eta_{{\bf{f}}}:
\widehat{H}^{1}_{\sing}(K_{\mathfrak{p}^{\infty}, v_1},
T_{{\bf{f}},n}) \oplus
\widehat{H}^{1}_{\sing}(K_{\mathfrak{p}^{\infty}, v_2},
T_{{\bf{f}},n}) \longrightarrow \Sel_{v_1
v_2}^{{\bf{f}}}(K_{\mathfrak{p}^{\infty}}).\end{align*} Let
$\eta_{{\bf{f}}}^{\varphi}$ denote the map obtained from
$\eta_{{\bf{f}}}$ after tensoring with $\mathcal{O}'$ via $\varphi$.
By Lemma \ref{2.6/2.7}, the domain of $\eta_{{\bf{f}}}$ is
isomorphic to $\left( \mathcal{O}'/\varphi(\mathfrak{P})^n
\right)^2$. Property $(5)$ of Lemma \ref{4.5/4.6} implies that
$\ker\left(\eta_{{\bf{f}}}^{\varphi}\right)$ contains the vectors
$\left(\partial_{v_1}\left(\xi_{\varphi}''(v_1)\right), 0 \right)$
and $ \left(0,
\partial_{v_2}\left(\xi_{\varphi}''(v_2)\right) \right)$ in
\begin{align*}\left( \widehat{H}^{1}_{\sing}(K_{\mathfrak{p}^{\infty}, v_1},
T_{{\bf{f}},n}) \oplus
\widehat{H}^{1}_{\sing}(K_{\mathfrak{p}^{\infty}, v_2},
T_{{\bf{f}},n})\right) \otimes_{\varphi}\mathcal{O}' \cong \left(
\mathcal{O}'/\varphi(\mathfrak{P})^n \right)^2.\end{align*} Observe that by
property $(3)$ of Lemma \ref{4.5/4.6}, we have the equalities \begin{align*}
t_{{\bf{f}}} - t_{\bf{g}} = \ord_{\mathfrak{P}'}\left(
\partial_{v_1}\left(\xi_{\varphi}''(v_1) \right)\right) =
\ord_{\mathfrak{P}'}\left( \partial_{v_2}\left( \xi_{\varphi}''(v_2)
\right)\right).\end{align*} Thus, we obtain the inclusion
\begin{align}\label{fitt} \left( \mathfrak{P}' \right)^{t_{{\bf{f}}}
- t_{\bf{g}} } \in \Fitt_{\mathcal{O}'}\left( \Sel_{v_1
v_2}^{{\bf{f}}}(K_{\mathfrak{p}^{\infty}}) \otimes_{\varphi}
\mathcal{O}' \right).\end{align} Let us now repeat the same argument
for the eigenform ${\bf{g}}$. That is, consider the short exact
sequence \begin{align}\label{sesg} 0 \longrightarrow
\Sel_{v_1v_2}^{\bf{g}}(K_{\mathfrak{p}^{\infty}}) \longrightarrow
\Sel_{{\bf{g}},n}(K_{\mathfrak{p}^{\infty}})^{\vee} \longrightarrow
\Sel_{[v_1 v_2]}(K_{\mathfrak{p}^{\infty}})^{\vee} \longrightarrow
0, \end{align} and the natural surjective map induced by local Tate
duality
$$\eta_{{\bf{g}}}:\widehat{H}^{1}_{\sing}(K_{\mathfrak{p}^{\infty}, v_1}, T_{{\bf{g}},n})
\oplus \widehat{H}^{1}_{\sing}(K_{\mathfrak{p}^{\infty}, v_2},
T_{{\bf{g}},n}) \longrightarrow \Sel_{v_1
v_2}^{{\bf{g}}}(K_{\mathfrak{p}^{\infty}}).$$ Let
$\eta_{{\bf{g}}}^{\varphi}$ denote the map obtained from
$\eta_{{\bf{g}}}$ after tensoring with $\mathcal{O}'$ via $\varphi$.
The global reciprocity law of class field theory implies that
$\ker\left(\eta_{{\bf{g}}}^{\varphi}\right)$ contains the vectors
$\left(\vartheta_{v_1}\left(\xi_{\varphi}''(v_2)\right), 0 \right)$
and $\left(\vartheta_{v_1}\left(\xi_{\varphi}''(v_1)\right),
\vartheta_{v_2}\left(\xi_{\varphi}''(v_1)\right) \right) = \left(0,
\vartheta_{v_2}\left(\xi_{\varphi}''(v_1) \right) \right)$ in
\begin{align*}\left(\widehat{H}^{1}_{\sing}(K_{\mathfrak{p}^{\infty}, v_1},
T_{{\bf{g}},n}) \oplus
\widehat{H}^{1}_{\sing}(K_{\mathfrak{p}^{\infty}, v_2},
T_{{\bf{g}},n})\right) \otimes_{\varphi}\mathcal{O}' \cong \left(
\mathcal{O}'/\varphi(\mathfrak{P})^n \right)^2.\end{align*} By Corollary
\ref{erl2symm},
\begin{align*}\ord_{\mathfrak{P}'}\left(\vartheta_{v_1}\left(
\xi_{\varphi}''(v_2)\right) \right) = \ord_{\mathfrak{P}'}\left(
\vartheta_{v_2}\left( \xi_{\varphi}''(v_1)\right) \right) =
t_{{\bf{g}}} - t = 0.\end{align*} It follows that $\Sel_{v_1
v_2}^{{\bf{g}}}(K_{\mathfrak{p}^{\infty}})\otimes_{\varphi}\mathcal{O}'$
is trivial, in which case the natural surjective map of
$(\ref{sesg})$ defines an isomorphism \begin{align}\label{fittiso}
\Sel_{{\bf{g}},n}(K_{\mathfrak{p}^{\infty}})^{\vee} \longrightarrow
\Sel_{[v_1 v_2]}(K_{\mathfrak{p}^{\infty}})^{\vee}.\end{align} Now,
Lemma \ref{4.8} implies that $t_{{\bf{g}}}< t_{{\bf{f}}}$. Recall
that since ${\bf{g}}$ satisfies the hypotheses of Proposition
\ref{4.3}, we may invoke the inductive hypothesis to conclude that
\begin{align}\label{induct} \varphi\left(\mathcal{L}_{{\bf{g}}}\right)^2 \in
\Fitt_{\mathcal{O}'}\left(
\Sel_{{\bf{g}},n}(K_{\mathfrak{p}^{\infty}})^{\vee}\otimes_{\varphi}\mathcal{O}'\right).\end{align}
Now, \begin{align*}(\mathfrak{P}')^{2t_{{\bf{f}}}} &=
(\mathfrak{P}')^{2(t_{{\bf{f}}} -~ t_{{\bf{g}}})}\cdot
(\mathfrak{P}')^{2t_{{\bf{g}}}}\\ &\in \Fitt_{\mathcal{O}'}\left(
\Sel_{v_1v_2}^{{\bf{f}}}(K_{\mathfrak{p}^{\infty}})
\otimes_{\varphi}\mathcal{O}'\right) \cdot
\Fitt_{\mathcal{O}'}\left( \Sel_{{\bf{g}},
n}(K_{\mathfrak{p}^{\infty}})^{\vee}
\otimes_{\varphi}\mathcal{O}'\right)\end{align*} by $(\ref{fitt})$
and $(\ref{induct})$. The isomorphism $(\ref{fittiso})$ gives an
inclusion \begin{align*} & \Fitt_{\mathcal{O}'}\left(
\Sel_{v_1v_2}^{{\bf{f}}}(K_{\mathfrak{p}^{\infty}})
\otimes_{\varphi}\mathcal{O}'\right) \cdot
\Fitt_{\mathcal{O}'}\left( \Sel_{{\bf{g}},
n}(K_{\mathfrak{p}^{\infty}})^{\vee}
\otimes_{\varphi}\mathcal{O}'\right)\\&\subseteq
\Fitt_{\mathcal{O}'}\left(
\Sel_{v_1v_2}^{{\bf{f}}}(K_{\mathfrak{p}^{\infty}})
\otimes_{\varphi}\mathcal{O}'\right) \cdot
\Fitt_{\mathcal{O}'}\left(
\Sel_{[v_1,v_2]}(K_{\mathfrak{p}^{\infty}})^{\vee}
\otimes_{\varphi}\mathcal{O}'\right).\end{align*} The short exact
sequence $(\ref{sesf})$ and the theory of Fitting ideals then give
\begin{align*} & \Fitt_{\mathcal{O}'}\left(
\Sel_{v_1v_2}^{{\bf{f}}}(K_{\mathfrak{p}^{\infty}})
\otimes_{\varphi}\mathcal{O}'\right) \cdot
\Fitt_{\mathcal{O}'}\left(
\Sel_{[v_1v_2]}(K_{\mathfrak{p}^{\infty}})^{\vee}
\otimes_{\varphi}\mathcal{O}'\right)\\
&\subseteq \Fitt_{\mathcal{O}'}\left(
\Sel_{{\bf{f}},n}(K_{\mathfrak{p}^{\infty}})^{\vee}\right).\end{align*}
In particular, we may conclude that \begin{align*}\varphi\left(
\mathcal{L}_{{\bf{f}}}\right)^2 \in \Fitt_{\mathcal{O}'}\left(
\Sel_{{\bf{f}},n}(K_{\mathfrak{p}^{\infty}})^{\vee}\otimes_{\varphi}\mathcal{O}'\right),\end{align*}
which proves Proposition \ref{4.3}.   \end{remark} \end{remark}

\section{Integral models of Shimura curves}

We collect here some facts about integral models of 
Shimura curves over totally real fields, following the works 
of Carayol \cite{Ca}, Cerednik \cite{Ce}, Drinfeld \cite{Dr} and Varshavsky \cite{Var}, \cite{VarII},
as required for the Euler system construction. The reader should note that in some places, in
particular where we describe the work(s) of Carayol \cite{Ca}, we assume for simplicity that 
the degree $d$ of the totally real field is greater than $1$. This however does not affect the 
validity of the results stated below, for which the $d=1$ cases have already been 
established (see also the article of Buzzard \cite{Bu}).

\begin{remark}[Reduction at split primes.]

Fix an indefinite quaternion algebra $B$ over $F$
as above, ramified at all the real places of $F$ save 
a fixed real place $\tau_1$. Fix a finite prime $v \subset \mathcal{O}_F$
where $B$ is split. Hence, we may fix an isomorphism $B_v \cong \M(F_v)$.

\begin{remark}[Integral models.] 

Fix a compact open subgroup $H \subset 
\widehat{B}^{\times}$. Let us assume that $H$ factorizes as $H_v \times H^v$, 
with $H_v \subset B_v^{\times} $ assumed to be maximal,
i.e. isomorphic to $\GL(\mathcal{O}_{F_v})$. The following theorem was first proved 
by Morita \cite{Mo}, then subsequently generalized by Carayol in \cite{Ca}. Recall that 
let $M_H$ denote the quaternionic Shimura curve associated to the complex manifold $M_H({\bf{C}}) = M_H(B,X)({\bf{C}})$.

\begin{theorem}[Morita-Carayol]\label{mc} Fix a finite prime $v \subset \mathcal{O}_F$ that splits the 
quaternion algebra $B$. Let $H \subset \widehat{B}^{\times}$ be any compact open
subgroup admitting the factorization $H_v \times H^v$, with $H_v \cong \GL(\mathcal{O}_{F_v})$.
Then, the Shimura curve $M_H$ has good reduction at $v$. In particular, there exists a smooth, proper 
model ${\bf{M}}_H$ of $M_H$ over $\mathcal{O}_{(v)}$. This model is unique up to isomorphism.
\end{theorem}

\begin{proof} See \cite{Mo} and \cite{Ca}, where the result is proved for $H^v$ ``sufficiently small".
If $H^v$ is not ``sufficiently small", then it is still possible to obtain an integral model ${\bf{M}}_H$
of $M_H$ over $\mathcal{O}_{(v)}$, as explained in \cite[$\S 12$]{Jar} or \cite[ $\S 3.1.3$]{CV2} (cf.
\cite[p. 508]{KM}). That is, let $H^{'v} \subset H^v$ be any sufficiently small, compact open normal 
subgroup, and put $H' =  H^{'v} \times H_v$. We can then define ${\bf{M}}_{H}$ to be the quotient of 
${\bf{M}}_{H'}$ by the $\mathcal{O}_{(v)}$-linear right action of $H/H'$. It is 
then possible to show that this model ${\bf{M}}_{H}$ is proper and regular if $H_v$ is maximal. 
Moreover, this construction does not depend on the choice of auxiliary $H^{'v}.$  \end{proof}
\end{remark}

\begin{remark}[Supersingular points. ] 

Recall that we fix an isomorphism $B_v \cong \M(F_v)$. 
To be consistent with the notations of Carayol \cite{Ca}, let us write $H_v^0$ to denote the compact open 
subgroup of $B_v^{\times}$ corresponding to $\GL(\mathcal{O}_{F_v})$ under this isomorphism. Given 
an integer $n \geq1$, let $H_v^n \subset H_v^0$ denote the subgroup corresponding to matrices that 
are congruent to $1 \mod v^n$. Given any integer $n \geq 0$, we then write \begin{align*}
M_{n, H^v} &= M_{H_v^n \times H^v} \end{align*} to denote the associated Shimura curve. 
Let ${\bf{M}}_{0, H^v}$ denote the integral model of $M_{0, H^v}$ over $\mathcal{O}_{(v)}$.
Consider the right action of the quotient group $H_v^0/H_v^n \cong \GL(\mathcal{O}_{F_v}/v^n)$ on 
the $\mathcal{O}_{F_v}$-module $\left(v^{-n}/\mathcal{O}_{F_v} \right)^2$ 
given by the rule \begin{align*} g \in \GL(\mathcal{O}_{F_v}/v^n) \text{ sends } h 
\in \left( v^{-n}/\mathcal{O}_{F_v}\right)^2 \text{ to } g^{-1} \cdot 
h.\end{align*} This same group $H_v^0/H_v^n$ acts on the Shimura curve 
$M_{n,H^v}$ via the quotient $M_{0,H^v}$. If $H^v$ is ``sufficiently small" 
in the sense of \cite{Ca}, then this 
action is free. One can then define a scheme of $\mathcal{O}_{F_v}$-modules 
over $M_{0,H^v}$: \begin{align*} E_n = \left( M_{n,H^v} \times \left( 
v^{-n}/\mathcal{O}_{F_v} \right)^2\right)/\GL(\mathcal{O}_{F_v}/v^n). 
\end{align*} These $E_n$ form a compatible system with respect to the 
indices $n$ and $H^v$, and the inductive limit \begin{align*}E_{\infty} := 
\lim_n E_n\end{align*} is the Barsotti-Tate group associated to the 
projective limit \begin{align*}M_{\infty}
:= \ilim {H^v} M_{0,H^v}.\end{align*} A variant of the main method of
\cite{Ca} can be used to find unique extensions of the groups $E_n$ to 
finite, locally free group schemes ${\bf{E}}_n$ over the smooth, proper 
$\mathcal{O}_{(v)}$-schemes ${\bf{M}}_{0,H^v}$. As before, the inductive 
limit \begin{align*}{\bf{E}}_{\infty} := \lim_n {\bf{E}}_n\end{align*} is 
the Barsotti-Tate group associated to the projective limit 
\begin{align*}{\bf{M}}_{\infty} := \ilim H {\bf{M}}_{0,H^v}.\end{align*} 
The group ${\bf{E}}_{\infty}$ has also been studied by Drinfeld \cite{Dr2} 
(cf. also \cite[Appendice]{Ca}) as a ``divisible $\mathcal{O}_{F_v}$-module 
of height $2$". In particular, this description gives the following classification
of points. Let $x$ be a point in the special fibre ${\bf{M}}_{0,H^v} \otimes \kappa_v$. 
Consider the covering \begin{align*}{\bf{M}}_{\infty} \longrightarrow 
{\bf{M}}_{0,H^v},\end{align*} and choose a lift $y$ of $x$. Consider the 
pullback of ${\bf{E}}_{\infty}/{\bf{M}}_{\infty}$ with respect to the map 
$y: \Spec(\overline{\kappa}_v) \longrightarrow {\bf{M}}_{\infty}.$ The 
resulting $\mathcal{O}_{F_v}$-module, written here as 
${\bf{E}}_{\infty}\vert x$, does not depend on choice of $y$. Drinfeld's 
theory shows that there are only two possibilities for this module:

\begin{itemize}
\item[(i)]${\bf{E}}_{\infty}\vert x \cong
\left( F_v/\mathcal{O}_{F_v}\right) \times
\Sigma_1$. Here, $\left( F_v/\mathcal{O}_{F_v} \right)$ is the
constant divisible $\mathcal{O}_{F_v}$-module, and $\Sigma_1$ the
unique formal $\mathcal{O}_{F_v}$ module of height $1$.
\item[(ii)]${\bf{E}}_{\infty} \vert x \cong
\Sigma_2.$ Here, $\Sigma_2$ is the unique formal
$\mathcal{O}_{F_v}$-module of height $2$.\end{itemize} Hence,
we can make the following 

\begin{definition} A geometric point $x$ in the special fibre 
${\bf{M}}_{0,H^v} \otimes \kappa_v$ is {\it{ordinary}} if ${\bf{E}}_{\infty} 
\vert x \cong \left( F_v/\mathcal{O}_{F_v} \right) \times \Sigma_1$, and 
{\it{supersingular}} if ${\bf{E}}_{\infty}\vert x \cong \Sigma_2$. 
\end{definition} Carayol \cite[$\S$ 11]{Ca} shows that the set of 
supersingular points ${\bf{M}}_{0,H^v}^{ss}\otimes \kappa_v$ of $ 
{\bf{M}}_{0,H^v}\otimes \kappa_v$ is finite and nonempty. That is, let 
$D$ denote the quaternion algebra obtained from $B$ by switching 
invariants at $\tau_1$ and $v$. Hence, \begin{align*}\Ram(D) = 
\Ram(B) \cup \lbrace \tau_1, v \rbrace.\end{align*} 

\begin{proposition}[Carayol]\label{carayolss} Let ${\bf{M}}_{0,H^v}^{ss} = {\bf{M}}_{0, H^v} 
\otimes \kappa_v$ denote the set of supersingular points of ${\bf{M}}_{0,H^v} \otimes \kappa_v$. 
There are bijections of finite sets \begin{align*} {\bf{M}}_{0,H^v}^{ss} &\cong 
D^{\times} \backslash \widehat{B}^{v \times} \times F_{v}^{\times}/H^v 
\times \mathcal{O}_{F_v}^{\times} \\ &\cong 
D^{\times} \backslash \widehat{D}^{\times} /H^v \times 
\mathcal{O}_{D_v}^{\times}.\end{align*} \end{proposition}

\begin{proof}
See \cite[$\S$ 11.2]{Ca} for the case of $d>1$. The result for $d=1$ is well known, see for
instance the paper of Ribet \cite{Ri1}.
\end{proof} \end{remark}

\begin{remark}[Geometric connected components.] 

Let $\mathcal{M}_{n, \nrd(H)}$ denote the set of geometric connected components of 
$M_{n,H}$, viewed as a finite $F$-scheme. Let ${\bf{M}}_{n,H^v}$ denote the normalization
of ${\bf{M}}_{0, H^v}$ in $M_{n, H^v}$. (Carayol in \cite{Ca}, using the theory
of Drinfeld bases with an analogue of the Serre-Tate theorem, shows that ${\bf{M}}_{n, H^v}$
is an integral model of $M_{n, H^v}$ over $\mathcal{O}_{(v)}$. Moreover, it is a regular 
scheme, finite and flat over ${\bf{M}}_{0, H^v}$). The reciprocity law for canonical 
models gives an isomorphism \begin{align}\label{rlcm}\mathcal{M}_{n, \nrd(H)} 
\longrightarrow \Spec(F').\end{align} Here, $F'$ is a certain finite abelian 
extension of $F$. Hence, $\mathcal{M}_{n, \nrd(H)}$ extends in a natural 
way to a normal $\mathcal{O}_{(v)}$-scheme \begin{align*} 
{\bf{\mathcal{M}}}_{n, \nrd(H)} \longrightarrow \Spec(\mathcal{O}_{
(v)}'), \end{align*} with $\mathcal{O}_{(v)}'$ the ring of $v$ integers 
of $F'$. The structural morphism of $M_{n,H}$ in $\mathcal{M}_{n, \nrd(H)}$ 
is then shown by Carayol \cite{Ca} to extend to a morphism 
\begin{align}\label{morphism}{\bf{M}}_{n,H} \longrightarrow 
{\bf{\mathcal{M}}}_{n, \nrd(H)}.\end{align} This morphism is smooth outside 
of the finite set of supersingular points. Moreover, if $x$ is a geometric 
point in the special fibre ${\bf{\mathcal{M}}}_{n, \nrd(H)} \otimes 
\kappa,$ then the fibre over $x$ in $(\ref{morphism})$ is given by a union 
of smooth, irreducible curves indexed by 
${\bf{P}}^1(\mathcal{O}_{F_v}/v^n)$ that intersect transversally at each 
supersingular point, and nowhere else. \end{remark} \end{remark}

\begin{remark}[Reduction at ramified primes.]

We now consider the reduction of a Shimura curve
modulo a prime that divides the discriminant of the underlying
quaternion algebra.

\begin{remark}[Admissible curves.] 

Let us for future reference establish the notion of an {\it{admissible curve}}, 
following Jordan-Livn\'e \cite[$\S$ 3]{JoLi}. Let $R$ be the 
ring of integers of any local field, with $\kappa$ the residue 
field, and $\pi$ a uniformizer.

\begin{definition} A curve $C$ defined over $R$ is said
to be {\it{admissible}} if 
\begin{itemize}
\item[1.] $C$ is proper and flat over $R$, with a smooth
generic fibre.
\item[2.] The special fibre of $C$ is reduced. The normalization
of each of its irreducible components is isomorphic to 
${\bf{P}}^1_{\kappa}$. The only singular points on the special
fibre of $C$ are $\kappa$-rational, ordinary double points. 
\item[3.] The completion of the local ring of $C$ at any one of 
its singular points $x$ is isomorphic as an $R$-algebra to 
$R[[X,Y]] / (XY - \pi^{m(x)})$ for some uniquely determined 
integer $m(x) \geq 1$. \end{itemize}\end{definition}

The special fibre of an admissible curve $C/R$ can be described
as a graph, following \cite{Ku}. In general, a graph $\mathcal{G} 
= (\mathcal{V}, \mathcal{E})$ here consists of a vertex set 
$\mathcal{V}$ and an edgeset $\mathcal{E}$. We fix an orientation 
of $\mathcal{G}$, i.e. a pair of maps $s, t: \mathcal{E} 
\longrightarrow \mathcal{V}$ that associates to each edge $e \in 
\mathcal{E}$ a {\it{source}} $s(e)$ and a {\it{target}} $t(e)$. 
Each edge $e$ then has an associated opposite edge $\overline{e}$ 
such that $s(\overline{e}) = t(e)$ with $\overline{\overline{e}} = e$. 
(The possibility that $\overline{e} = e$ is allowed). A graph $\mathcal{G}$ 
is said to have a {\it{length}} if there exists a function \begin{align*} 
l = l_{\mathcal{G}}: \mathcal{E} &\longrightarrow 
{\bf{N}} = \lbrace 1, 2, 3, \ldots \rbrace \end{align*} with 
$l(e) = l(\overline{e})$. Such a graph has the following 
standard representation: a marked point corresponds to 
a vertex; a line joining two marked points corresponds to a pair 
of edges $\lbrace e, \overline{e} \rbrace$, and has ``length'' 
$l(e) = l(\overline{e})$. 

\begin{definition} Let $C$ be an admissible curve over $R$, and 
$C_0$ its special fibre. The {\it{dual graph $\mathcal{G}(C) = 
(\mathcal{V}(C), \mathcal{E}(C))$ of $C$}}
is the following graph:

\begin{itemize}
\item[(i)] The vertex set $\mathcal{V}(C)$ consists of 
the components of $C_0$.
\item[(ii)] The edge set $\mathcal{E}(C)$ consists 
of the branches of $C_0$ through each double point of $C_0$. 
\item[(iii)] If an edge $e \in \mathcal{E}(C)$ passes
through a double point $x\in C_0$, then $\overline{e}$ is the 
other branch of $C_0$ passing through $x$. Moreover, $s(e)$ 
is the component of $C_0$ containing $x$, and $t(e) = s(\overline{e})$.
\item[(iv)] The lenth $l(e)$ of an edge $e \in \mathcal{E}(C)$
passing through a double point $x$ is the uniquely determined integer
$m(x)$ defined above.
\end{itemize}\end{definition} \end{remark}

\begin{remark}[Mumford-Kurihara uniformization.]

Fix a finite prime $v$ of $F$. Let ${\bf{C}}_v$ denote the 
completion of a fixed algebraic closure of $F_v$. Let $\widehat{\Omega} = 
\widehat{\Omega}_{F_v}$ denote the $v$-adic upper half plane over $F_v$, viewed
as a formal scheme over $\mathcal{O}_{F_v}$. Hence, $\widehat{\Omega}$ is flat 
and locally of finite type over $\mathcal{O}_{F_v}$. It is regular and irreducible, 
and supports a natural action of $\PGL(F_v)$. The generic fibre $\Omega$ of 
$\widehat{\Omega}$, which also supports a natural action of $\PGL(F_v)$, 
is a rigid analytic space with ${\bf{C}}_v$-points given by \begin{align*} 
\Omega({\bf{C}}_v) &= {\bf{P}}^1_{{\bf{C}}_v} - {\bf{P}}^1_{F_v} = {\bf{C}}_v 
- F_v .\end{align*} We refer the reader to \cite{BC}, \cite{JoLi} or \cite{Mum} 
for further background on this construction. Let us just collect the following crucial
facts:

\begin{itemize}

\item[(i)] The special fibre of $\widehat{\Omega}$ is reduced 
and geometrically connected. Its components are smooth, projective, 
$\kappa_v$-rational curves that intersect transversally.

\item[(ii)] The dual graph of $\widehat{\Omega}_v$ equipped with its natural 
$\PGL(F_v)$-action is identified canonically with the Bruhat-Tits tree
$\Delta = (\mathcal{V}(\Delta), \mathcal{E}(\Delta))$ of $\SL(F_v)$ (as
constructed for instance in \cite[$\S$ 1]{Mum}).

\item[(iii)] If $\Gamma \subset \PGL(F_v)$ is a discrete, cocompact subgroup,
then the quotient $\Gamma \backslash\widehat{\Omega}$ is a formal scheme over
$\mathcal{O}_{F_v}$, identified canonically with the completion of some scheme
$\Omega_{\Gamma}$ over $\mathcal{O}_{F_v}$ along its closed fibre. 

\end{itemize}

\begin{theorem}[Mumford-Kurihara]\label{mumfordkurihara}
If $\Gamma \subset \PGL(F_v)$ is any discrete, cocompact subgroup,
then the associated scheme $\Omega_{\Gamma}$ is an admissible curve
over $\mathcal{O}_{F_v}$ whose dual graph is canonically isomorphic 
to $\Gamma \backslash \Delta$ minus loops. \end{theorem}

\begin{proof} See \cite{Mum} for the torsionfree case, with
\cite{Ku} for the general case.  \end{proof}\end{remark}

\begin{remark}[Cerednik-Varshavsky uniformization.]

Fix a finite prime $v \subset \mathcal{O}_F$. Let $F_v^{unr}$ 
denote the maximal unramified extension of $F_v$, with 
$\mathcal{O}_{F_v}^{\unr}$ its ring of integers. Let \begin{align*}
\widehat{\Omega}^{\unr} &= \widehat{\Omega}
\times_{\Spf (\mathcal{O}_{F_v})} \Spf (\mathcal{O}_{F_v}^{\unr}).
\end{align*} Following Drinfeld \cite{Dr}, we define a 
natural action of $\GL(F_v)$ on $\widehat{\Omega}^{\unr}$ as 
follows: for any $\gamma \in \GL(F_v)$ and $(x, u) \in 
\widehat{\Omega}^{\unr}$, \begin{align*} \gamma \cdot (x, u)
&= ([\gamma]x, \Frob_v^{n(\gamma)}u).\end{align*} Here, $[\gamma]$ 
denotes the class of $\gamma$ in $\PGL(F_v)$, and 
$n(\gamma) = - \ord_v(\det(\gamma))$. Suppose now that 
$\Gamma \subset \GL(F_v)$ is a discrete, cocompact subgroup
containing some power of the matrix $$\left( \begin{array}{cc}
\pi_v &  0\\ 0 & \pi_v \end{array}\right).$$ Then, the quotient
$\Gamma \backslash \widehat{\Omega}^{\unr}$ exists, and is given
canonically by the completion of a scheme $\Omega^{\unr}_{\Gamma}$ 
along its closed fibre. This scheme $\Omega^{\unr}_{\Gamma}$ is 
moreover an admissible curve over $\mathcal{O}_{F_v}$.

Let $\mathfrak{N}^+$ and $\mathfrak{N}^{-} $ be relatively coprime ideals of $\mathcal{O}_F$. 
Let $\mathfrak{N}= \mathfrak{N}^+\mathfrak{N}^{-}$. Suppose that $\mathfrak{N}^{-}$ is the 
squarefree product of a number of primes congruent to $d \mod 2$.
Fix a prime divisor $\mathfrak{q}$ of $\mathfrak{N}^{-}$. 
Let $v$ be a finite prime of $F$ that does not divide $\mathfrak{N}$. 
Let $B$, $B'$, and $D$ be the quaternion algebras over $F$ with ramification 
sets given by \begin{align*}\Ram(B) &= \lbrace  \tau_2, \ldots \tau_d \rbrace \cup \lbrace w 
: w \mid \mathfrak{N}^{-}/\mathfrak{q} \rbrace \\
\Ram(B') &= \Ram(B) \cup \lbrace v, \mathfrak{q} \rbrace \\ 
\Ram(D) &= \Ram(B) \cup \lbrace \tau_1, \mathfrak{q} \rbrace. 
\end{align*} Hence, $B$ is indefinite with $\disc(B) = 
\mathfrak{N}^{-}/\mathfrak{q}$, $B'$ is indefinite with 
$\disc(B')=v\mathfrak{N}^{-}$, and $D$ totally definite 
with $\disc(D) = \mathfrak{N}^{-}$.
Note that we have isomorphisms $\widehat{B}^{v\mathfrak{q}} 
\cong \widehat{B}'^{v\mathfrak{q}}\cong \widehat{D}^{v\mathfrak{q}}$. 
Let us fix compatible isomorphisms $\widehat{B}^{v\mathfrak{q}} 
\cong \widehat{B}'^{v\mathfrak{q}}, ~
\widehat{B}^{\mathfrak{q}} \cong \widehat{D}^{\mathfrak{q}}, ~
\widehat{B}'^{v} \cong \widehat{D}^{v}.$ In particular,
let us fix an isomorphism \begin{align}\label{varpi}
\varphi: \widehat{D}^v \cong \widehat{B}'^v.\end{align} 
Fix a compact open subgroup $U \subset \widehat{D}^{\times}$
of level $\mathfrak{N}^+$. Let us assume that \begin{align}\label{Hbarv}
U^v &= U_S \prod_{w \notin S \cup \lbrace v \rbrace}
U_w, \end{align} where $S \supset \Ram(D)$ is any finite 
set of places of $F$. Let us then define \begin{align}\label{H'} 
H' &= \varphi \left(U^v \right) \times \mathcal{O}_{B_v'}^{\times}. 
\end{align} 

\begin{theorem}[Cerednik-Varshavsky]\label{cerednikvarshavsky}
Let $M_{H'}$ be a Shimura curve as defined above. Suppose 
that $H'$ admits the factorization $H' = H'_v \times H^{' v}$, with
$H_v'$ maximal. Then there exists an integral model ${\bf{M}}_{H'}$ of $M_{H'}$
over $\mathcal{O}_{(v)}$ whose completion along its 
closed fibre is canonically isomorphic to
\begin{align}\label{cv} V_{H'} &=\GL(F_v)\backslash \widehat{\Omega}^{\unr} 
\times D^{\times} \backslash \widehat{D}^{\times}/U^v 
.\end{align} This canonical isomorphism is $\widehat{B}'^{\times v}$-equivariant, 
where $\widehat{B}'^{\times v}$ acts on ${\bf{M}}_{H'}$ in the 
natural way, and on $V_{H'}$ via its action on the 
finite set $D^{\times} \backslash \widehat{D}^{\times}/U^{v}$.
\end{theorem}
\begin{proof} See \cite[$\S$ 3.1]{Raj}. Existence of the integral
model ${\bf{M}}_{H'}$ follows from Varshavsky \cite[Theorem 5.3]{VarII},
taking $r=1$, $v_1 = v$, $D = D$, $D^{\operatorname{int}} = B'$, and 
$\mathcal{G}' = \widehat{D}^{\times v}$. Note that the conditions of 
\cite[Theorem 5.3]{VarII} are satisfied by \cite[1.5.2]{Var}. Identification 
of the completion of ${\bf{M}}_{H'}$ along its closed fibre with $V_{H'}$ is 
then a consequence of Cerednik \cite[Theorem 2.2]{Ce}.  \end{proof}
By Theorem \ref{mumfordkurihara}, ${\bf{M}}_{H'}$ is seen easily to be an 
admissible (hence semistable) curve over $\mathcal{O}_{F_v}$, with dual graph 
$\mathcal{G}({\bf{M}}_{H'})$ given canonically by \begin{align}\label{MK} 
\mathcal{G}({\bf{M}}_{H'}) &=
\GL(F_v)^+ \backslash \Delta \times D^{\times}\backslash
\widehat{D}^{\times}/U^v. \end{align} Here, $\GL(F_v)^+ 
\subset \GL(F_v)$ denotes the subset of matrices whose determinants
have even $v$-adic valuation, and $\Delta = (\mathcal{V}(\Delta),
\mathcal{E}(\Delta))$ is the Bruhat-Tits tree of $\SL(F_v)$. 

\begin{corollary}\label{MK}
Let $\mathcal{G}({\bf{M}}_{H'}) = (\mathcal{V}({\bf{M}}_{H'}), 
\mathcal{E}({\bf{M}}_{H'}))$ denote the dual graph of the special 
fibre of ${\bf{M}}_{H'}$. We have the following identifications: 
\begin{align*} \mathcal{V}({\bf{M}}_{H'}) &\cong D^{\times} 
\backslash \widehat{D}^{\times}/ U \times {\bf{Z}}/2{\bf{Z}} 
\\ \mathcal{E}({\bf{M}}_{H'}) &\cong D^{\times} \backslash \widehat{D}^{\times}
/U(v).\end{align*} Here, $U = U_v \times U^v$ with $U_v \cong \GL(\mathcal{O}_{F_v})$,
and $$U(v) = \lbrace u \in U : u_v   \cong \left( \begin{array}{cccc} * & * \\ 0 & * \end{array}\right) 
\mod \varpi_v\rbrace.$$ \end{corollary}

\begin{proof} See \cite[3.2]{Raj}. The result is easy to deduce from the 
standard identifications $\mathcal{V}(\Delta) \cong \PGL(F_v)/
\PGL(\mathcal{O}_{F_v})$ and $\mathcal{E}(\Delta) \cong 
\GL(F_v)^{+}/ V_0(v)F^{\times}$, where $V_0(v) \subset \GL(\mathcal{O}_{F_v})$
denotes the matrices congruent to $0 \mod v$.  \end{proof}\end{remark}

\begin{remark}[Orientation of the dual graph.]

Let us from now on fix the following orientation of the dual graph $\mathcal{G}({\bf{M}}_{H'})
= (\mathcal{V}({\bf{M}}_{H'}), \mathcal{E}({\bf{M}}_{H'}))$ via $(\ref{MK})$. 
That is, let us call a vertex in $\mathcal{V}(\Delta)$ {\it{even}}
or {\it{odd}} according to its distance from the origin vertex corresponding to the local maximal
order $\M(\mathcal{O}_{F_v})$ (see \cite[$\S$ II.2]{Vi}). Since $\GL(F_v)^+$ consists of matrices
having even $v$-adic valuation, its action by conjugation on maximal orders is seen to send
even vertices to even vertices, and odd vertices to odd vertices. In particular, the notions of 
even and odd vertices on the quotient graph $\GL(F_v)^{+}\backslash \Delta$ are
well defined. Hence, the notions are also well defined on the dual graph $\mathcal{G}({\bf{M}}_{H'})$.
We then chose an orientation $s,t: \mathcal{E}({\bf{M}}_{H'}) \longrightarrow \mathcal{V}({\bf{M}}_{H'})$
such that for any edge $e \in \mathcal{E}({\bf{M}}_{H'})$, the source $s(e)$ is even, and the target $t(e)$ is odd.
\end{remark}\end{remark}

\section{Character groups and connected components}

Fix a Shimura curve $M = M_H$ associated to an indefinite quaternion 
algebra $B$, as above. Fix a prime $v \subset \mathcal{O}_F$. Let $F_{v^2}$
denote the quadratic unramified extension of $F_v$. Assume that 
the level $H$ factorizes as $H = H_v \times H^v$, with $H_v \subset B_v^{\times}$ 
maximal. Let ${\bf{M}}={\bf{M}}_H$ denote the integral model of $M$ over 
$\mathcal{O}_{F_v}$ basechanged to $\mathcal{O}_{F_{v^2}}$. Hence, ${\bf{M}}$ is the 
basechange of the integral model of Theorem \ref{mc} if $v$ does not divide the discriminant 
of $B$, or else the basechange of the integral model of Theorem 
\ref{cv} if $v$ does divide the discriminant of $B$. Write

\begin{itemize}
\item[] $J$ for the Jacobian of $M$,
\item[] ${\bf{J}}$ for the N\'eron model of $J\otimes_F F_{v^2}$
over $\mathcal{O}_{F_{v^2}}$.
\item[] ${\bf{J}}_{v}$ for the special fibre ${\bf{J}} \otimes
\kappa_{v^2},$
\item[] ${\bf{J}}_{v}^{0}$ for the component of the identity of
${\bf{J}}_{v},$
\item[] $\Phi_{v}$ for the group of geometric connected components
${\bf{J}}_{v}/{\bf{J}}_{v}^{0}.$
\end{itemize}

\begin{definition} Let $\operatorname{Tor}({\bf{J}}_{v}^{0})$ denote the maximal
subtorus of ${\bf{J}}_{v}^{0}$. The group $\mathcal{X}_v =
\operatorname{Hom}( \operatorname{Tor}({\bf{J}}_{v}^{0}),
{\bf{G}}_m)$ is the {\it{character group associated to ${\bf{J}}_{v}$}}. 
\end{definition} We have two different descriptions of the character group 
$\mathcal{X}_v$ and the group of connected components $\Phi_v$: a combinatorial 
one due to Raynaud \cite{Ray}, and a cohomological one due to Grothendieck \cite{Groth}. 
Following Edixhoven \cite{Ed}, we combine these to obtain a third 
description, which we shall use later to describe the specialization of 
divisors in the group of connected components $\Phi_v$.

\begin{remark}[Dual graph description.] 

Let $\mathcal{G}_v = (\mathcal{V}(\mathcal{G}_v), \mathcal{E}
(\mathcal{G}_v))$ be the dual graph associated to the special 
fibre ${\bf{M}}\otimes \kappa_{v^2}$. (In the case where $v$ does 
not divide the discriminant of $B$, the dual graph of ${\bf{M}}$
is defined in the same way as for admissible curves). 
Let ${\bf{Z}}[\mathcal{V} (\mathcal{G}_v)]$ denote the module of 
formal divisors supported on $\mathcal{V}(\mathcal{G}_v)$ with coefficients in {\bf{Z}}, and 
${\bf{Z}}[\mathcal{V}(\mathcal{G}_v)]^0$ the submodule of divisors 
having degree zero on each connected component of 
$\mathcal{V}(\mathcal{G}_v)$. Let ${\bf{Z}}[\mathcal{E}(\mathcal{G}_v)]$ 
denote the module of formal divisors supported on $\mathcal{E}(\mathcal{G}_v)$ with
coefficients in ${\bf{Z}}$. Fixing an orientation $s,t:
\mathcal{E}(\mathcal{G}_v) \longrightarrow
\mathcal{V}(\mathcal{G}_v)$, we then define boundary and
coboundary maps respectively by
\begin{align*} d_* &= t_* - s_*: {\bf{Z}}[\mathcal{E}(\mathcal{G}_v)] \longrightarrow
{\bf{Z}}[\mathcal{V}(\mathcal{G}_v)],\\ d^* &= t^* - s^*:
{\bf{Z}}[\mathcal{V}(\mathcal{G}_v)] \longrightarrow
{\bf{Z}}[\mathcal{E}(\mathcal{G}_v)].
\end{align*}

\begin{theorem}[Raynaud]\label{raynaud} There is a canonical short exact
sequence \begin{align}\label{raynaud} \begin{CD} 0 @>>>
\mathcal{X}_v @>>> {\bf{Z}}[\mathcal{E}(\mathcal{G}_v)]
@>{d_*} >> {\bf{Z}}[\mathcal{V}(\mathcal{G}_v)]^0 @>>> 0,
\end{CD}\end{align} as well as a canonical isomorphism $\mathcal{X}_v
\cong H_1(\mathcal{G}_v, {\bf{Z}})$. In particular, there is an 
isomorphism $\mathcal{X}_v \cong \ker(d_*)$.
\end{theorem}

\begin{proof}
See \cite[Proposition 8.1.2]{Ray} or \cite[Theorem 9.6/1]{BLR}
with the discussion in \cite{JoLi}. The result is also described 
in \cite[$\S$ 1]{Ed}. \end{proof}

\begin{corollary}\label{GOODid} 
Assume that $v$ does not divide the discriminant of $B$, and that $H$ has the 
factorization $ H_v \times H^v$ with $H_v $ maximal. Then, $\mathcal{X}_v \cong {\bf{Z}}[\mathcal{E}(\mathcal{G}_v)]^0$.
\end{corollary}

\begin{proof} The result follows from Carayol's description of singular (= supersingular) points of ${\bf{M}}_{0,H^v}$,
cf. \cite[Proposition 5.3]{BD}.

 \end{proof}

\begin{corollary}\label{BADid}
Assume that $v$ does not divide the discriminant of $B$, and that $H$ has the 
factorization $ H_v \times H^v$ with $H_v $ maximal. Choose an orientation
$s, t: \mathcal{E}(\mathcal{G}_v) \longrightarrow \mathcal{V}(\mathcal{G}_v)$ such
that for any edge $e \in \mathcal{E}(\mathcal{G}_v)$, the source $s(e)$ is even, and 
the target $t(e)$ is odd. Then, writing $\delta_*$ to denote the restriction of the coboundary
map $d_*$ to ${\bf{Z}}[\mathcal{E}(\mathcal{G}_v)]^0 $, the character group $\mathcal{X}_v$ 
fits into the short exact sequence \begin{align*}\begin{CD} 0 @>>> \mathcal{X}_v @>>> 
{\bf{Z}}[\mathcal{E}(\mathcal{G}_v)]^0 @>{\delta_*}>> {\bf{Z}}[\mathcal{V}(\mathcal{G}_v)]^0 
@>>> 0.\end{CD}\end{align*}\end{corollary}

\begin{proof}
We claim that with this choice of orientation, the elements of $H_1(\mathcal{G}_v, {\bf{Z}})$ belong
to ${\bf{Z}}[\mathcal{E}(\mathcal{G}_v)]^0 $, cf. \cite[Proposition 5.5]{BD}. \end{proof}

\end{remark}

\begin{remark}[Vanishing cycles description.]
The character group $\mathcal{X}_v$ can also be
described in language of vanishing cycles of \cite[$\S$ XIII and
XV]{Groth} to give the following main result.

\begin{theorem}[Grothendieck]\label{SGA7}
There is a canonical short exact sequence
\begin{align}\label{monoses} \begin{CD} 0 @>>> \mathcal{X}_v @>{\lambda}>>
{\widehat{\mathcal{X}}}_v @>>> \Phi_{v} @>>>0.\end{CD}\end{align}
Here, $\widehat{\mathcal{X}}_v$ denotes the ${\bf{Z}}$-dual of
$\mathcal{X}_v$, and $\lambda$ denotes the canonical injection
induced by the monodromy pairing of \cite[$\S$9]{Groth}. In
particular, there is a canonical isomorphism
$\coker({\lambda})\cong \Phi_v$.
\end{theorem}

\begin{proof}
See \cite[Th\'eor\`eme 11.5]{Groth}. The result is also 
described in \cite[$\S$1]{Ed}. 
\end{proof}

\end{remark}

\begin{remark}[Comparison description (Edixhoven).]

Following \cite[(1.6)]{Ed}, we may then compare the 
descriptions of Raynaud and Grothendieck via the 
following commutative diagram, whose rows and columns 
are exact:

\begin{align*}\begin{CD}
@. @. 0 @. 0 @.\\
@. @. @VVV @VVV @.\\
@. @. \mathcal{X}_v @>{\operatorname{id}}>> \mathcal{X}_v\\
@. @. @VVV @VV{\lambda}V @.\\
0 @>>> {\bf{Z}}[\mathcal{V}(\mathcal{G}_v)] @>{d^*}>> {\bf{Z}}[\mathcal{E}(\mathcal{G}_v)]
@>>> \widehat{\mathcal{X}}_v @>>> 0\\
@. @VV{-\operatorname{id}}V @VV{d_*}V @VVV @.\\
0 @>>> {\bf{Z}}[\mathcal{V}(\mathcal{G}_v)] @>{\mu_0}>> {\bf{Z}}[\mathcal{V}
(\mathcal{G}_v)]^0 @>>> \Phi_v @>>> 0\\
@. @. @VVV @VVV @.\\
@. @. 0 @. 0 @.
\end{CD}\end{align*} Here, the composition of $\mu_0$ with the natural
inclusion ${\bf{Z}}[\mathcal{V}(\mathcal{G}_v)]^0 \longrightarrow
{\bf{Z}}[\mathcal{V}(\mathcal{G}_v)]$ is given by the map \begin{align*} 
\mu: {\bf{Z}}[\mathcal{V}(\mathcal{G}_v)] &\longrightarrow 
{\bf{Z}}[\mathcal{V}(\mathcal{G}_v)], ~~~ \mu(C) = \sum_{C'}(C \cdot C')C',
\end{align*} where $C,C' \in \mathcal{V}(\mathcal{G}_v)$ are irreducible
components of the special fibre ${\bf{M}} \otimes \kappa_{v^2}$,
and $(C \cdot C') \in {\bf{Z}}$ is their intersection product on
 ${\bf{M}} \otimes \mathcal{O}_{F_{v^2}}$. We refer the reader to 
 \cite[$\S$1]{Ed} or \cite[1.6.5]{Nek} for more details. \end{remark}

\begin{remark}[Specialization to connected components.]

Fix a Shimura curve $M = M_H$ as above, associated to an
indefinite quaternion algebra $B$ over $F$. Fix a prime
$v \subset \mathcal{O}_F$ that divides the discriminant
of $B$.

\begin{proposition}\label{omeganat} There is a natural map
\begin{align}\label{omeganatural}
\omega_v: {\bf{Z}}[\mathcal{V}(\mathcal{G}_v)]^0 \longrightarrow
\Phi_v.\end{align}\end{proposition}

\begin{proof} See the argument of \cite[Corollary 5.12]{BD} (or that of 
\cite[$\S$ 4.4]{Lo}, where it is applied to each connected component of 
$\mathcal{G}_v$). Let us write the short exact sequence $(\ref{raynaud})$ as 
\begin{align*}\begin{CD} 0 @>>>
\mathcal{X}_v @>{\gamma}>> {\bf{Z}}[\mathcal{E}(\mathcal{G}_v)]
@>{d_*}>> {\bf{Z}}[\mathcal{V}(\mathcal{G}_v)] @>{\operatorname{deg}}>> {\bf{Z}} 
@>>>0, \end{CD}\end{align*} where
$\operatorname{deg}$ denotes the degree map. Taking distinguished
bases for ${\bf{Z}}[\mathcal{E}(\mathcal{G}_v)]$ and
${\bf{Z}}[\mathcal{V}(\mathcal{G}_v)]$, we can then consider the dual
exact sequence \begin{align*}\begin{CD} 0 @>>> {\bf{Z}} 
@>{\operatorname{diag}}>> {\bf{Z}}[\mathcal{V}(\mathcal{G}_v)] @>{d^*}>>
{\bf{Z}}[\mathcal{E}(\mathcal{G}_v)] @>{\widehat{\gamma}}>>
\widehat{\mathcal{X}}_v @>>>0, \end{CD}\end{align*} where
$\operatorname{diag}$ denotes the diagonal map. Let
$\lambda_0: {\bf{Z}}[\mathcal{E}(\mathcal{G}_v)] \longrightarrow
{\bf{Z}}[\mathcal{E}(\mathcal{G}_v)]$ denote the map induced by the
monodromy pairing of \cite[$\S$ 9]{Groth}. We then deduce that the map $\lambda$
in the short exact sequence \begin{align*}\begin{CD} 0 @>>> \mathcal{X}_v
@>{\lambda}>> {\widehat{\mathcal{X}}}_v @>{c_v}>> \Phi_{v}
@>>>0,\end{CD}\end{align*} of Theorem \ref{SGA7} must be given by the 
composition $\widehat{\gamma} \circ \lambda_0 \circ \gamma$.
The sought after map $\omega_v$ can then be defined as follows: 
\begin{align*} \omega_v: {\bf{Z}}[\mathcal{V}(\mathcal{G}_v)]^0
&\longrightarrow \Phi_v \\ x &\longmapsto \left( c_v \circ \widehat{\gamma}
\circ \lambda_0 \right)(y),\end{align*} where $y$ is chosen such
that $d_*(y) = x$.  \end{proof} \end{remark}

\begin{remark}[Specialization of divisors.] 

Let $\Div(M)$ denote the group of divisors on 
$M \otimes_F \overline{F}$ having coeffients in ${\bf{Z}}$. Let
$\Div^0(M)$ denote subgroup of divisors having degree $0$ on each
connected component of $M\otimes_F \overline{F}$. Hence, the class
of a divisor $D \in \Div^0(M)$ under linear equivalence corresponds
to an element $[D]$ of $J(\overline{F})$. Given a divisor $D \in \Div(M)$, 
let $\Supp(D)$ denote its support. Let \begin{align*}\red_v: M \otimes_F \overline{F}
\longrightarrow \mathcal{V}(\mathcal{G}_v) \cup \mathcal{E}(\mathcal{G}_v)\end{align*} 
denote the map that sends a point $P$ to either the connected component
containing its image in ${\bf{M}} \otimes \kappa_{v^2}$
if $P$ does not reduce to a singular point, or else to its image in 
${\bf{M}} \otimes \kappa_{v^2}$ (a singular point).
We consider divisors 
\begin{align*}D = \sum_P n_P P \in \Div^0(M)\end{align*} for which 
the following conditions hold:

\begin{itemize}
\item[(i)]Each $P\in \Supp(D)$ is defined over $F_{v^2}$.
\item[(ii)]The image of each $P \in \Supp(D)$ under $\red_v$
goes to a vertex in $\mathcal{V}(\mathcal{G}_v)$, i.e.
the image of each $P$ in ${\bf{M}} \otimes \kappa_{v^2}$ is a nonsingular point. 
\end{itemize} Let us for future reference call any such divisor 
{\it{$F_{v^2}$-nonsingular}}. The reduction mod $v$ of such a divisor $D$ 
then takes the form \begin{align*}\red_v(D) = 
\sum_P n_P \cdot \red_v(D) \in {\bf{Z}}[\mathcal{V}(\mathcal{G}_v)]^0.\end{align*}
Now, consider the specialization map \begin{align*}
\partial_v: J(F_{v^2}) &\longrightarrow \Phi_v.\end{align*}

\begin{proposition}\label{cmpointreduction} 
Let $D \in \Div^0(M)$ be an $F_{v^2}$-nonsingular divisor, with $[D]$ its class 
in $ J(F_{v^2})$. Then,
\begin{align}\label{reduction}\partial_v\left( [D]\right) = \omega_v \left( \red_v(D)
\right).\end{align} \end{proposition}

\begin{proof}
The image $\partial_v \left( [D] \right)$ can be described in terms 
of intersection numbers via Raynaud's description of $\Phi_v$, following the argument
of \cite[$\S$ 2]{Ed} (cf. \cite[Proposition 5.14]{BD}, \cite[1.6.6]{Nek}). The 
result is then simple to deduce.  \end{proof} \end{remark}

\section{Hecke module correspondences}

Let us return to the setup of Theorem \ref{cerednikvarshavsky}, keeping all of the notations and hypotheses as
above. We must first introduce some more precise notations. To this end, suppose we are given coprime ideals 
$\mathfrak{M}^+$ and $\mathfrak{M}^{-} $ of $\mathcal{O}_F$, where $\mathfrak{M}^{-}$ is the squarefree product of a number of 
primes congruent to $d-1 \mod 2$. We shall henceforth write $M(\mathfrak{M}^+, \mathfrak{M}^{-})$ to denote the 
Shimura curve of level $\mathfrak{M}^+$ associated to the indefinite quaternion algebra of discriminant $\mathfrak{M}^{-}$.
(Here ``indefinite" means that the underlying quaternion algebra is ramified at all but one of the real places of $F$, hence the condition on 
$\mathfrak{M}^{-}$). Fix a prime $v \subset \mathcal{O}_F$ that does not divide $\mathfrak{M}^+ \mathfrak{M}^{-}$. We assume that 
the level of $M(\mathfrak{M}^+, \mathfrak{M}^{-})$ is maximal at $v$, in which case there exists an integral model 
${\bf{M}}(\mathfrak{M}^{+}, \mathfrak{M}^{-})$ of $M(\mathfrak{M}^+, \mathfrak{M}^{-})$ over $\mathcal{O}_{F_v}$ 
(due to Carayol in the case that $v \nmid \mathfrak{M}^{-}$, or Cerednik-Varshavsky in 
the case that $v \mid \mathfrak{M}^{-}$). Let $J(\mathfrak{M}^{+}, \mathfrak{M}^{-})$ denote the jacobian of 
$M(\mathfrak{M}^{+}, \mathfrak{M}^{-})$, with ${\bf{J}}(\mathfrak{M}^{+}, \mathfrak{M}^{-})$ its N\'eron model 
over $\mathcal{O}_{F_v}$, and ${\bf{J}}_v^0(\mathfrak{M}^{+}, \mathfrak{M}^{-})$ the component of the identity
of its special fibre. Let $\mathcal{X}_v(\mathfrak{M}^{+}, \mathfrak{M}^{-})$ denote the character group
of the maximal torus of ${\bf{J}}_v^0(\mathfrak{M}^{+}, \mathfrak{M}^{-})$. Given an ideal $\mathfrak{m} \subset 
\mathcal{O}_F$  that does not divide $\mathfrak{M}^+ \mathfrak{M}^{-}$, let $M(\mathfrak{m}; \mathfrak{M}^+, 
\mathfrak{M}^{-})$ denote the Shimura curve $M(\mathfrak{M}^{+}, \mathfrak{M}^{-})$ with maximal level structure
at primes dividing $\mathfrak{m}$ inserted. Hence, $M(\mathfrak{m}\mathfrak{M}^+, \mathfrak{M}^{-})$ is the Shimura curve
 of $\mathfrak{m}\mathfrak{M}^{+}$-level structure associated to the indefinite Shimura curve of discriminant $\mathfrak{M}^{-}$,
 with the extra condition that the level be maximal at primes dividing $\mathfrak{m}$. 
 
 Suppose now that we are given two coprime ideals $\mathfrak{N}^+$ and $ \mathfrak{N}^{-} $ of $\mathcal{O}_F$ such that 
 $\mathfrak{N}^{-}$ is the squarefree product of a number of primes congruent to $d \mod 2$. Given a ring $\mathcal{O}$, recall 
 that we let $\mathbb{S}_2(\mathfrak{N}^{+}, \mathfrak{N}^{-}; \mathcal{O})$ denote the space of $\mathcal{O}$-valued automorphic 
 forms of weight $2$ and level $\mathfrak{N}^+$ on the totally definite quaternion algebra of discriminant $\mathfrak{N}^{-}$ over $F$. 
 Let $\mathbb{T}(\mathfrak{N}^+, \mathfrak{N}^{-})$ denote the associated algebra of Hecke operators. Given an ideal 
 $\mathfrak{n} \subset \mathcal{O}_F$ that does not divide the product $\mathfrak{N}^+ \mathfrak{N}^{-}$, 
 let $\mathbb{S}_2(\mathfrak{n}; \mathfrak{N}^{+}, \mathfrak{N}^{-}; {\bf{Z}})$ denote the space of forms of level $\mathfrak{n}\mathfrak{N}^+$, with 
 the level being maximal at primes dividing $\mathfrak{n}$. Fix a prime $v \subset \mathcal{O}_F$ that does not divide the product
 $\mathfrak{N}^+ \mathfrak{N}^{-}$. Let us now take $\mathfrak{M}^+ = \mathfrak{N}^+$ and $\mathfrak{M}^{-} = v \mathfrak{N}^{-}$
 in the setup above. In particular, we consider the Shimura curve $M(\mathfrak{N}^+, v\mathfrak{N}^{-} )$, with 
 $\mathcal{X}_v(\mathfrak{N}^+, v\mathfrak{N}^{-})$ the associated character group, and $\mathcal{G}_v 
 = (\mathcal{V}(\mathcal{G}_v), \mathcal{E}(\mathcal{G}_v))$ the associated dual graph. Putting things together, 
 we obtain the following diagram \`a la Ribet \cite{Ri2}, where the rows are exact, and the vertical arrows are isomorphisms:

\begin{align}\label{Rfd}\begin{CD}
 @.   \mathcal{X}_{\mathfrak{q}}(v\mathfrak{q}; \mathfrak{N}^+, \mathfrak{N}^{-}/\mathfrak{q})
 @>>>   \mathcal{X}_{\mathfrak{q}}(\mathfrak{q}; \mathfrak{N}^+, \mathfrak{N}^{-}/\mathfrak{q})^2 \\
@. @AAA @AAA \\
 @.  \Div^0 \left( {\bf{M}}(v\mathfrak{q}; \mathfrak{N}^+, \mathfrak{N}^{-}/\mathfrak{q})^{ss} \otimes \kappa_{\mathfrak{q}}\right)    
 @>>>    \Div^0 \left( {\bf{M}}(\mathfrak{q}; \mathfrak{N}^+, \mathfrak{N}^{-}/\mathfrak{q})^{ss} \otimes \kappa_{\mathfrak{q}} \right)^2  \\
@. @AAA @AAA \\
 \mathcal{X}_v(\mathfrak{N}^{+}, v\mathfrak{N}^{-}) @>>> {\bf{Z}}[\mathcal{E}(\mathcal{G}_v)] @>{d_*}>>  {\bf{Z}}[\mathcal{V}(\mathcal{G}_v)]^0 \\
@VVV @VVV @VVV \\
\mathcal{X}_v(\mathfrak{N}^{+}, v \mathfrak{N}^{-}) @>>> {\bf{Z}}[D^{\times}\backslash \widehat{D}^{\times}/U(v)] @>>>  {\bf{Z}}[D^{\times}\backslash \widehat{D}^{\times}/U]^0 \times {\bf{Z}}/2{\bf{Z}}\\
@VVV @VVV @VVV \\
\mathbb{S}_2(U(v), D; {\bf{Z}})^{\operatorname{v-new}} @>>> \mathbb{S}_2(U(v), D; {\bf{Z}}) @>{\alpha_*, \beta_*}>>  \mathbb{S}_2(U, D; {\bf{Z}})^{\oplus 2} \\
\end{CD}
\end{align}  Here, we start with the exact sequence of Theorem \ref{raynaud}. The identifications \begin{align*} {\bf{Z}}[\mathcal{E}(\mathcal{G}_v)]  &\cong 
{\bf{Z}}[D^{\times}\backslash \widehat{D}^{\times}/U(v)], ~~~{\bf{Z}}[\mathcal{V}(\mathcal{G}_v)]^0 \cong 
{\bf{Z}}[D^{\times}\backslash \widehat{D}^{\times}/U]^0 \times {\bf{Z}}/2{\bf{Z}}
\end{align*} come from Corollary \ref{MK}, making the bottom exact sequence a direct consequence of 
definitions. The identification \begin{align*} {\bf{Z}}[D^{\times}\backslash \widehat{D}^{\times}/U]^0 
&\cong   \Div^0 \left( {\bf{M}}(\mathfrak{q}; \mathfrak{N}^+, \mathfrak{N}^{-}/\mathfrak{q})^{ss} \otimes \kappa_{\mathfrak{q}} \right)
\end{align*} comes from Proposition \ref{carayolss}. The identification \begin{align*}
{\bf{Z}}[\mathcal{E}(\mathcal{G}_v)]  &\cong  \Div^0 \left( {\bf{M}}(v\mathfrak{q}; \mathfrak{N}^+, \mathfrak{N}^{-}/\mathfrak{q})^{ss} 
\otimes \kappa_{\mathfrak{q}}\right)  \end{align*} is deduced from Corollary \ref{BADid}. The identifications \begin{align*}
 \Div^0 \left( {\bf{M}}(v\mathfrak{q}; \mathfrak{N}^+, \mathfrak{N}^{-}/\mathfrak{q})^{ss} \otimes \kappa_{\mathfrak{q}}\right) 
 &\cong  \mathcal{X}_{\mathfrak{q}}(v\mathfrak{q}; \mathfrak{N}^+, \mathfrak{N}^{-}/\mathfrak{q}) \\
\Div^0 \left( {\bf{M}}(\mathfrak{q}; \mathfrak{N}^+, \mathfrak{N}^{-}/\mathfrak{q})^{ss} \otimes \kappa_{\mathfrak{q}} \right)
&\cong  \mathcal{X}_{\mathfrak{q}}(\mathfrak{q}; \mathfrak{N}^+, \mathfrak{N}^{-}/\mathfrak{q}) \end{align*} come from Corollary \ref{GOODid}.
In particular, we use $(\ref{Rfd})$ deduce the following result. Recall that we write $\eta_v = \left( \begin{array}{cccc} 0 & 1 \\ 0 & \varpi_v 
 \end{array}\right)$, where $\varpi_v$ is a fixed uniformizer of $v$. Let us write the associated monodromy exact sequences of Theorem
 \ref{SGA7} as \begin{align*}\begin{CD} 
 \mathcal{X}_v(\mathfrak{N}^+, v\mathfrak{N}^{-}) @>{\lambda}>> \widehat{\mathcal{X}}_v(\mathfrak{N}^+, v\mathfrak{N}^{-}) @>>> \Phi_v(\mathfrak{N}^+, v\mathfrak{N}^{-}) \\
 \mathcal{X}_{\mathfrak{q}}(v\mathfrak{q}; \mathfrak{N}^+, \mathfrak{N}^{-}/\mathfrak{q}) @>{\lambda(v \mathfrak{q})}>> \widehat{\mathcal{X}}_{\mathfrak{q}}(v\mathfrak{q};
 \mathfrak{N}^+, \mathfrak{N}^{-}/\mathfrak{q}) @>>> \Phi_{\mathfrak{q}}(v\mathfrak{q}; \mathfrak{N}^+, \mathfrak{N}^{-}/\mathfrak{q}) \\
 \mathcal{X}_{\mathfrak{q}}(\mathfrak{q}; \mathfrak{N}^+, \mathfrak{N}^{-}/\mathfrak{q}) @>{\lambda(\mathfrak{q})}>> \widehat{\mathcal{X}}_{\mathfrak{q}}(\mathfrak{q};
 \mathfrak{N}^+, \mathfrak{N}^{-}/\mathfrak{q}) @>>> \Phi_{\mathfrak{q}}(\mathfrak{q}; \mathfrak{N}^+, \mathfrak{N}^{-}/\mathfrak{q}). \end{CD}\end{align*}
  
\begin{theorem}\label{Ribet} We have the following diagram of $\mathbb{T}(v\mathfrak{q}; \mathfrak{N}^{+}, \mathfrak{N}^{-}/\mathfrak{q})$-modules,
where the rows are exact:
\begin{align*}\begin{CD}  
\widehat{\mathcal{X}}_v(\mathfrak{N}^{+}, v\mathfrak{N}^{-})@<<<  \widehat{\mathcal{X}}_{\mathfrak{q}}(v\mathfrak{q}; 
\mathfrak{N}^+, \mathfrak{N}^{-}/\mathfrak{q})@<{1_* \oplus {\eta_v}_*}<<  \widehat{\mathcal{X}}_{\mathfrak{q}}(\mathfrak{q}; \mathfrak{N}^+, 
\mathfrak{N}^{-}/\mathfrak{q})^2 \\ @AA{\lambda}A @AA{\lambda(v\mathfrak{q})}A @AA{\lambda(\mathfrak{q})}A \\
\mathcal{X}_v(\mathfrak{N}^{+}, v\mathfrak{N}^{-})@>>>  \mathcal{X}_{\mathfrak{q}}(v\mathfrak{q}; 
\mathfrak{N}^+, \mathfrak{N}^{-}/\mathfrak{q})@>>>  \mathcal{X}_{\mathfrak{q}}(\mathfrak{q}; \mathfrak{N}^+, 
\mathfrak{N}^{-}/\mathfrak{q})^2. \end{CD}\end{align*} \end{theorem}

\begin{proof} We extract the bottom exact sequence from that of the top of $(\ref{Rfd})$. The top exact sequence is 
then induced by duality.\end{proof}

\begin{corollary} [Jacquet-Langlands]\label{JLChc} ~~~
\begin{itemize}

\item[(i)] We have the following diagram of ${\bf{T}}(\mathfrak{N}^{+}, \mathfrak{N}^{-})$-modules,
where the rows are exact:
\begin{align*}\begin{CD}  
\widehat{\mathcal{X}}_v(\mathfrak{N}^{+}, v\mathfrak{N}^{-})@<<<  \widehat{\mathcal{X}}_{\mathfrak{q}}(v\mathfrak{q}; 
\mathfrak{N}^+, \mathfrak{N}^{-}/\mathfrak{q})@<{1_* \oplus {\eta_v}_*}<<  \widehat{\mathcal{X}}_{\mathfrak{q}}(\mathfrak{q}; \mathfrak{N}^+, 
\mathfrak{N}^{-}/\mathfrak{q})^2 \\ @AA{\lambda}A @AA{\lambda(v\mathfrak{q})}A @AA{\lambda(\mathfrak{q})}A \\
\mathcal{X}_v(\mathfrak{N}^{+}, v\mathfrak{N}^{-})@>>>  \mathcal{X}_{\mathfrak{q}}(v\mathfrak{q}; 
\mathfrak{N}^+, \mathfrak{N}^{-}/\mathfrak{q})@>>>  \mathcal{X}_{\mathfrak{q}}(\mathfrak{q}; \mathfrak{N}^+, 
\mathfrak{N}^{-}/\mathfrak{q})^2. \end{CD}\end{align*} 

\item[(ii)] The subring of of $\End(J(\mathfrak{N}^+, v\mathfrak{N}^{-}))$ generated by Hecke correspondences
on $M(\mathfrak{N}^+, v\mathfrak{N}^{-})$ is isomorphic to the Hecke algebra ${\bf{T}}(\mathfrak{N}^+, \mathfrak{N}^{-})$.
\end{itemize}

\end{corollary}

\begin{proof}
To show $(i)$, we simply take into account the identifications induced by the vertical arrows of $(\ref{Rfd})$ to
obtain a diagram of $\mathbb{T}(\mathfrak{N}^+, \mathfrak{N}^{-})$-modules. The result then follows
from the Jacquet-Langlands correspondence. Since $M(\mathfrak{N}^+, v\mathfrak{N}^{-}) \otimes F_v$
has a semistable model over $\mathcal{O}_{F_v}$, the general theory of N\'eron 
models (see \cite{BLR}[$\S$ 9], \cite{Nek}[1.6.2]) shows that $J(\mathfrak{N}^+, v\mathfrak{N}^{-})$ has purely toric 
reduction at $v$. Hence, the Hecke algebra $\mathbb{T}(\mathfrak{N}^+, v\mathfrak{N}^{-})$ acting faithfully on 
$\mathcal{X}_v(\mathfrak{N}^+, v\mathfrak{N}^{-})$ can be identified with the subalgebra of $\End(J(\mathfrak{N}^+, 
v\mathfrak{N}^{-}))$ generated by Hecke correspondences. The result then also follows from the Jacquet-Langlands 
correspondence. \end{proof} 

\section{Weak level raising}
 
 Recall that given a given a quaternion algebra $B$ and a level $H \subset \widehat{B}^{\times}$, we 
 let $\mathbb{T} = \mathbb{T}(H, B)$ denote the ${\bf{Z}}$-algebra generated by the standard Hecke 
 operators $T_w$ and $S_w$ for all primes $w \subset \mathcal{O}_F$ (where they are defined). Let
 us adopt the convention of writing $U_w$ for the operators $T_w$ if $w \subset \mathcal{O}_F$ is 
 a prime that divides the level.
  
\begin{definition} A maximal ideal $\mathfrak{m} \subset \mathbb{T}$ is said to be {\it{Eisenstein}}
if there exists an ideal $\mathfrak{f} \subset \mathcal{O}_F$ such that for all but finitely many primes
$w \subset \mathcal{O}_F$ that split completely in the ray class field of $\mathfrak{f} \mod F$, $T_w - 
2 \in \mathfrak{m}$ and $S_w -1 \in \mathfrak{m}$.\end{definition}

\begin{proposition}[Jarvis]\label{jarvis}
A maximal ideal $\mathfrak{m} \subset \mathbb{T}$ associated to an eigenform ${\bf{f}}$ is Eisenstein if and 
only if the associated Galois representation $\rho_{\bf{f}}: G_F \longrightarrow \GL(\mathcal{O})$ is reducible.
\end{proposition}

\begin{proof} See \cite[$\S$ 3]{Jar2}, which extends to totally real fields \cite{DT}[Proposition 2]. \end{proof} ~~\\

\begin{remark}[Level raising at one prime.]~~~

\begin{theorem}[Rajaei]\label{Rajaei} Let $\mathfrak{m} \subset \mathbb{T}(v\mathfrak{q}; \mathfrak{N}^+, \mathfrak{N}^{-}/\mathfrak{q})$
be any non-Eisenstein maximal ideal. 

\begin{itemize}
\item[(i)]We have the following diagram of $\mathbb{T}(v\mathfrak{q}; \mathfrak{N}^{+}, 
\mathfrak{N}^{-}/\mathfrak{q})$-modules, where the rows are exact:
\begin{align*}\begin{CD}  \widehat{\mathcal{X}}_v(\mathfrak{N}^{+}, v\mathfrak{N}^{-})_\mathfrak{m}@<<<  \widehat{\mathcal{X}}_{\mathfrak{q}}(v\mathfrak{q}; 
\mathfrak{N}^+, \mathfrak{N}^{-}/\mathfrak{q})_\mathfrak{m}@<{1_* \oplus {\eta_v}_*}<<  \widehat{\mathcal{X}}_{\mathfrak{q}}(\mathfrak{q}; \mathfrak{N}^+, 
\mathfrak{N}^{-}/\mathfrak{q})_\mathfrak{m}^2 \\ @AA{\lambda}A @AA{\lambda(v\mathfrak{q})}A @AA{\lambda(\mathfrak{q})}A \\
\mathcal{X}_v(\mathfrak{N}^{+}, v\mathfrak{N}^{-})_\mathfrak{m}@>>>  \mathcal{X}_{\mathfrak{q}}(v\mathfrak{q}; 
\mathfrak{N}^+, \mathfrak{N}^{-}/\mathfrak{q})_\mathfrak{m}@>>>  \mathcal{X}_{\mathfrak{q}}(\mathfrak{q}; \mathfrak{N}^+, 
\mathfrak{N}^{-}/\mathfrak{q})_\mathfrak{m}^2. \end{CD}\end{align*} 

\item[(ii)] We have an isomorphism of $\mathbb{T}(v\mathfrak{q}; \mathfrak{N}^{+}, 
\mathfrak{N}^{-}/\mathfrak{q})$-modules \begin{align*} \mathcal{X}_{\mathfrak{q}}(\mathfrak{q}; \mathfrak{N}^+,
\mathfrak{N}^{-}/\mathfrak{q})^2_{\mathfrak{m}} / \left(U_v^2 - S_v  \right)  &\cong \Phi_v(\mathfrak{N}^+, v\mathfrak{N}^{-})_{\mathfrak{m}}.
\end{align*}

\end{itemize}
\end{theorem}

\begin{proof}
For $(i)$, see \cite[Theorem 3]{Raj}, which is a generalization to totally real fields of the method of Ribet \cite{Ri2}.
For $(ii)$, see \cite[Corollary 4]{Raj}, which shows that there is an isomorphism of $\mathbb{T}(v\mathfrak{q}; \mathfrak{N}^{+}, 
\mathfrak{N}^{-}/\mathfrak{q})$-modules \begin{align*} \widehat{\mathcal{X}}_{\mathfrak{q}}(\mathfrak{q}; \mathfrak{N}^+,
\mathfrak{N}^{-}/\mathfrak{q})^2_{\mathfrak{m}} / \left(U_v^2 - S_v  \right)  &\cong \widehat{\mathcal{X}}_v(\mathfrak{N}^+, 
v\mathfrak{N}^{-})_{\mathfrak{m}}/ \mathcal{X}_v(\mathfrak{N}^+, v\mathfrak{N}^{-})_{\mathfrak{m}}.\end{align*} By 
\cite[Proposition 5]{Raj}, there is an isomorphism \begin{align*}
\widehat{\mathcal{X}}_{\mathfrak{q}}(\mathfrak{q}; \mathfrak{N}^+,
\mathfrak{N}^{-}/\mathfrak{q})_{\mathfrak{m}}  &\cong \mathcal{X}_{\mathfrak{q}}(\mathfrak{q}; \mathfrak{N}^+,
\mathfrak{N}^{-}/\mathfrak{q})_{\mathfrak{m}}. \end{align*} The result then
follows from the identification of Theorem \ref{SGA7}. \end{proof}

\begin{corollary}[Jacquet-Langlands] 

Let $\mathfrak{m} \subset {\bf{T}}(\mathfrak{N}^+, \mathfrak{N}^{-})$ be any non-Eisenstein maximal ideal. 

\begin{itemize}
\item[(i)]We have the following diagram of $ {\bf{T}}(\mathfrak{N}^+, \mathfrak{N}^{-})$-modules, where the rows are exact:
\begin{align*}\begin{CD}  \widehat{\mathcal{X}}_v(\mathfrak{N}^{+}, v\mathfrak{N}^{-})_\mathfrak{m}@<<<  \widehat{\mathcal{X}}_{\mathfrak{q}}(v\mathfrak{q}; 
\mathfrak{N}^+, \mathfrak{N}^{-}/\mathfrak{q})_\mathfrak{m}@<{1_* \oplus {\eta_v}_*}<<  \widehat{\mathcal{X}}_{\mathfrak{q}}(\mathfrak{q}; \mathfrak{N}^+, 
\mathfrak{N}^{-}/\mathfrak{q})_\mathfrak{m}^2 \\ @AA{\lambda}A @AA{\lambda(v\mathfrak{q})}A @AA{\lambda(\mathfrak{q})}A \\
\mathcal{X}_v(\mathfrak{N}^{+}, v\mathfrak{N}^{-})_\mathfrak{m}@>>>  \mathcal{X}_{\mathfrak{q}}(v\mathfrak{q}; 
\mathfrak{N}^+, \mathfrak{N}^{-}/\mathfrak{q})_\mathfrak{m}@>>>  \mathcal{X}_{\mathfrak{q}}(\mathfrak{q}; \mathfrak{N}^+, 
\mathfrak{N}^{-}/\mathfrak{q})_\mathfrak{m}^2. \end{CD}\end{align*} 

\item[(ii)] We have an isomorphism of $ {\bf{T}}(\mathfrak{N}^+, \mathfrak{N}^{-})$-modules 
\begin{align*} \mathcal{X}_{\mathfrak{q}}(\mathfrak{q}; \mathfrak{N}^+,
\mathfrak{N}^{-}/\mathfrak{q})^2_{\mathfrak{m}} / \left(U_v^2 - S_v  \right)  &\cong \Phi_v(\mathfrak{N}^+, v\mathfrak{N}^{-})_{\mathfrak{m}}.
\end{align*}
\end{itemize}
 \end{corollary}

\begin{proof}
The result follows directly from Theorem \ref{Rajaei} with Corollary \ref{JLChc}.
\end{proof}

\begin{theorem}\label{raiseonefree} Fix an ${\bf{f}} \in \mathcal{S}_2(\mathfrak{N}^+, \mathfrak{N}^{-})$ an eigenform,  $v \subset \mathcal{O}_F$
a prime, and $n \geq 1$ an integer. Assume that the associated Galois representation $\rho_{\bf{f}}$ is residually irreducible,
and that the prime $v \subset \mathcal{O}_F$ is $n$-admissible with respect to ${\bf{f}}$. Then, there exists a mod $\mathfrak{P}_n$ 
eigenform ${\bf{f}}_v$ associated to surjective homomorphism \begin{align*} {\bf{T}}_0(\mathfrak{N}^+, v\mathfrak{N}^{-}) 
&\longrightarrow \mathcal{O}_0/\mathfrak{P}_n \end{align*} such that the following properties hold:

\begin{itemize}
\item[(i)] $T_w \left({\bf{f}}_v\right) \equiv T_w \left( {\bf{f}} \right)
\mod \mathfrak{P}_n$ for all primes $w \nmid
v \mathfrak{N}$ of $\mathcal{O}_F$.

\item[(ii)] $U_w \left( {\bf{f}}_v \right) 
\equiv U_w \left( {\bf{f}} \right) \mod \mathfrak{P}_n$ 
for all primes $w \mid \mathfrak{N}$ of $\mathcal{O}_F$.

\item[(iii)] $U_v \left( {\bf{f}}_v \right) \equiv \varepsilon \cdot {\bf{f}}_v
\mod \mathfrak{P}_n$.
\end{itemize} Here, $\varepsilon \in \lbrace \pm 1
\rbrace$ is the integer for which $\mathfrak{P}_n$ divides
${\bf{N}}(v) +1 - \varepsilon \cdot a_v({\bf{f}}).$

\end{theorem} We describe two proofs of this result.

\begin{proof}[Proof 1]\emph{(Rajaei, Ribet, Taylor)} Ribet first 
proved the result for $d=1$ and $n=1$ in \cite{Ri2}, where the 
$n>1$ case follows by a simple inductive argument. The general 
case of $d \geq 1$ and $n=1$ is shown in Rajaei \cite[Main
Theorem 3 and Corollary 4]{Raj}, granted certain technical hypotheses on
$F$ that always hold in our setting. (That is, if ${\bf{Q}}(\zeta_p)^{+} \subset F$, 
then it is assumed that $\overline{\rho}_{\bf{f}}$ is not induced from a character. If
$d\equiv 0 \mod 2$, then it is assumed that the associated
automorphic representation $\pi_{\bf{f}}$ is either special or supercuspidal 
at some finite prime $w \nmid \mathfrak{p} v \subset \mathcal{O}_F$). The result is also 
proved by Taylor \cite[Theorem 1]{Tay} for
$d$ even with $\mathfrak{N}^{-} = \mathcal{O}_F$. The
general case with $n>1$ can be deduced from the
methods of Ribet \cite[$\S$ 7]{Ri2} developed by Rajaei
\cite[$\S$4]{Raj}. That is, in the setup above, one looks at the 
Hecke module structure(s) of 
\begin{align*} \mathcal{X}_{\mathfrak{q}}(\mathfrak{q}; \mathfrak{N}^+,
\mathfrak{N}^{-}/\mathfrak{q})^2_{\mathfrak{m}} / \left(U_v^2 - S_v  \right)  
&\cong \Phi_v(\mathfrak{N}^+, v\mathfrak{N}^{-})_{\mathfrak{m}}. \end{align*} to deduce the result. \end{proof}

\begin{proof}[Proof 2]\emph{(Kisin)}
If ${\bf{f}}$ is associated via Jacquet-Langlands to an eigenform on a 
totally definite quaternion algebra over $F$, then Kisin
\cite[$\S$ 3.1]{Kis} gives a different proof for arbitrary totally real
fields with minor hypotheses on the level structure
\cite[(3.1.1) and (3.1.2)]{Kis}. Note that by our hypotheses imposed on 
the integer factorization of $\mathfrak{N}$, the eigenform ${\bf{f}}$ is always
associated via Jacquet-Langlands to an eigenform on a totally definite quaternion algebra. 
So, the results of \cite[$\S$ 3.1]{Kis} apply. \end{proof} We now impose the following
crucial hypothesis on the Galois representation $\rho_{\ff}$ associated to ${\bf{f}}$, following the 
approach of Pollack-Weston \cite{PW}:

\begin{hypothesis}[Multiplicity one for character groups.]\label{freeness} 
Given $\mathfrak{m} \subset {\bf{T}}(\N^+, \N^{-})$ a non-Eisenstein maximal ideal, the completed character 
group $$\mathcal{X}_{\mathfrak{q}}(\mathfrak{q}; \mathfrak{N}^+, \mathfrak{N}^{-}/\mathfrak{q})_{\mathfrak{m}_{\bf{f}}} 
\otimes \mathcal{O}$$ is free of rank one over completed Hecke algebra ${\bf{T}}_0(\mathfrak{N}^+, \mathfrak{N}^{-})$. 
\end{hypothesis}

\begin{remark} The result is well known for $F={\bf{Q}}$, see for instance the explanation given in 
\cite[Theorem 6.2]{PW}. The general case is treated in Cheng \cite{Ch2}, granted suitable technical 
hypotheses. Rather than state these here explicitly, we shall just assume Hypothesis \ref{freeness} 
in what follows to reveal the mechanism of the proof. \end{remark} Let $\mathcal{I}_{\ff}$ denote the 
kernel of the natural homomorphism ${\bf{T}}_0(\mathfrak{N}^+, \mathfrak{N}^-) \longrightarrow 
\mathcal{O}_0/\mathfrak{P}_n$ associated to $\ff$, and $\mathcal{I}_{{\ff}_v}$ to denote that of  
${\bf{T}}_0(\mathfrak{N}^+, v\mathfrak{N}^-) \longrightarrow \mathcal{O}_0/\mathfrak{P}_n$ 
associated to ${\ff}_v$. We obtain the following crucial result.

\begin{corollary}\label{galoisid} If $\mathcal{X}_{\mathfrak{q}}(\mathfrak{q}; \mathfrak{N}^+,
\mathfrak{N}^{-}/\mathfrak{q})_{\mathfrak{m}_{\bf{f}}} \otimes \mathcal{O}$ is free of rank $1$ over 
${\bf{T}}_0(\mathfrak{N}^+, \mathfrak{N}^{-})$, then

\begin{itemize}

\item[(i)] There is an isomorphism of groups $\Phi_v(\mathfrak{N}^+, v\mathfrak{N}^{-})/ \mathcal{I}_{{\bf{f}}_v} 
\cong \mathcal{O}_0/ \mathfrak{P}_n$.

\item[(ii)] There is an isomorphism of $G_F$-modules $\Ta_p \left(  J(\mathfrak{N}^+, v\mathfrak{N}^{-}) \right)
/ \mathcal{I}_{{\bf{f}}_v} \cong T_{{\bf{f}}, n}$.

\end{itemize}

\end{corollary}

\begin{proof} For $(i)$, see \cite[Theorem 5.15 (2)]{BD}, or the generalization to totally real fields 
given in \cite[Theorem 3.3]{L30} (neither of which requires $\mathfrak{P}$-isolatedness). The 
idea in either case is to observe that the freeness condition implies that 
\begin{align*} \mathcal{X}_{\mathfrak{q}}(\mathfrak{q}; \mathfrak{N}^+, \mathfrak{N}^{-}/\mathfrak{q})/ \mathcal{I}_{\bf{f}} 
&\cong \mathcal{O}_0/ \mathfrak{P}_n. \end{align*} It can then be deduced via the Ribet exact sequence (Theorem \ref{Ribet}) 
that we have isomorphisms \begin{align*} \mathcal{X}_{\mathfrak{q}}(\mathfrak{q}; \mathfrak{N}^+, \mathfrak{N}^{-}/\mathfrak{q})^2/ 
\left( \mathcal{I}_{\bf{f}}, U_v - \varepsilon \right) \cong \mathcal{X}_{\mathfrak{q}}(\mathfrak{q}; \mathfrak{N}^+, \mathfrak{N}^{-}/
\mathfrak{q})^2/ \left( \mathcal{I}_{\bf{f}}, U_v^2 - 1 \right) \cong  \mathcal{O}_0/ \mathfrak{P}_n. \end{align*} The result then
follows from the isomorphism of Theorem \ref{Rajaei} (ii). For (ii), we use the argument of Pollack-Weston \cite[Proposition 4.4]{PW}.
That is, a straightforward generalization of the second part of the proof \cite[Lemma 5.16]{BD} shows that the following property is 
satisfied: for each element $z \in \Phi_v(\mathfrak{N}^+, v\mathfrak{N}^{-})/\mathcal{I}_{{\bf{f}}_v}$, there exists an element 
$t \in J(\mathfrak{N}^+, v\mathfrak{N}^{-})[p^{n'}](\overline{F}_{v^2})/\mathcal{I}_{{\bf{f}}_v}$ for some integer $n' \geq 1$ that maps to $z$
under the natural map $$J(\mathfrak{N}^+, v\mathfrak{N}^{-})[p^{n'}](\overline{F}_{v^2})/\mathcal{I}_{{\bf{f}}_v} \longrightarrow
\Phi(\mathfrak{N}^+, v\mathfrak{N}^{-})/\mathcal{I}_{{\bf{f}}_v}.$$ We may then use the same argument as given in 
\cite[Proposition 4.4]{PW} to show the result.
\end{proof}

Recall that given an integer $m \geq 1$, we write $K_{\mathfrak{p}^m}$ to denote the $m$-th layer of the 
dihedral ${\bf{Z}}_p^{\delta}$-extension of $K$. Let us define the corresponding $m$-th level component 
group to be the direct sum of component groups \begin{align*}\Phi_{v, m}(\mathfrak{N}^+, v\mathfrak{N}^{-}) 
&=\bigoplus_{\mathfrak{v} \mid v} \Phi_{\mathfrak{v}}(\mathfrak{N}^+, v\mathfrak{N}^{-}).\end{align*} Here, the sum
ranges over all primes $\mathfrak{v}$ above $v$ in $K_{\mathfrak{p}^m}$, and
$\Phi_{\mathfrak{v}}(\mathfrak{N}^+, v\mathfrak{N}^{-})$ denotes the component group 
associated to the jacobian ${\bf{J}}(\mathfrak{N}^+, v\mathfrak{N}^{-})$
at $\mathfrak{v}$. Let us then write
\begin{align*}\widehat{\Phi}_v(\mathfrak{N}^+, v\mathfrak{N}^{-}) &= 
\ilim m \Phi_{v, m}(\mathfrak{N}^+, v\mathfrak{N}^{-})\end{align*} to denote 
the inverse limit with respect to norm maps. Recall that $v$ must be inert in $K$ by
hypothesis (ii) of $n$-admissibility, and hence splits completely in
$K_{\mathfrak{p}^{\infty}}$ by class field theory. It follows that we have an isomorphism
\begin{align*}\widehat{\Phi}_{v}(\mathfrak{N}^+, v\mathfrak{N}^{-}) &\cong 
\Phi_{v}(\mathfrak{N}^+, v\mathfrak{N}^{-}) \otimes
\Lambda.\end{align*} Here, as before, we let $\Lambda$
denote the $\mathcal{O}$-Iwasawa algebra
$\mathcal{O}[[G_{\mathfrak{p}^{\infty}}]]$. The
isomorphism $\Phi_v(\mathfrak{N}^+, v\mathfrak{N}^{-})/I_{{\bf{f}}_v} \cong
\mathcal{O}_0/\mathfrak{P}_n$ of Corollary \ref{galoisid} (ii)
then allows us to make the identification \begin{align*}
\widehat{\Phi}_v(\mathfrak{N}^+, v\mathfrak{N}^{-}) /I_{{\bf{f}}_v}\cong
\Lambda /\mathfrak{P}^n.\end{align*} Now, write
\begin{align*}\widehat{J}(\mathfrak{N}^+, v\mathfrak{N}^{-})(K_{\mathfrak{p}^{\infty}})/I_{{\bf{f}}_v} = \ilim
m J(\mathfrak{N}^+, v\mathfrak{N}^{-})(K_{\mathfrak{p}^m})/I_{{\bf{f}}_v}\end{align*} to denote the
inverse limit with respect to norm maps. Taking the inverse limit of the
associated specialization maps to groups of connected components, we
then obtain a completed specialization map \begin{align}\label{ilspec}\widehat{\partial}_v:
\widehat{J}(\mathfrak{N}^+, v\mathfrak{N}^{-})(K_{\mathfrak{p}^{\infty}})/I_{{\bf{f}}_v}
\longrightarrow \widehat{\Phi}_v(\mathfrak{N}^+, v\mathfrak{N}^{-})/I_{{\bf{f}}_v} \cong
\Lambda/\mathfrak{P}^n.\end{align}

\begin{corollary}\label{ESid} If $\mathcal{X}_{\mathfrak{q}}(\mathfrak{q}; \mathfrak{N}^+,\mathfrak{N}^{-}/\mathfrak{q})_{\mathfrak{m}_{\bf{f}}} 
\otimes \mathcal{O}$ is free of rank $1$ over ${\bf{T}}_0(\mathfrak{N}^+, \mathfrak{N}^{-})$, then \\

\item[(i)] We have isomorphisms \begin{align} \Phi_v(\mathfrak{N}^+, v\mathfrak{N}^{-})/I_{{\bf{f}}_v}
&\cong H_{\sing}^1(K_v, T_{{\bf{f}},n})\label{minorisom}\\
\label{mainisom} \widehat{\Phi}_v(\mathfrak{N}^+, v\mathfrak{N}^{-})/I_{{\bf{f}}_v} &\cong
\widehat{H}_{\sing}^1(K_{\mathfrak{p}^{\infty}, v}, T_{{\bf{f}},n}),
\end{align} both of which are canonical up to choice of isomorphism
$$\Ta_p \left(  J(\mathfrak{N}^+, v\mathfrak{N}^{-}) \right)/ \mathcal{I}_{{\bf{f}}_v} 
\cong T_{{\bf{f}}, n}.$$

\item[(ii)] We have a commutative diagram \begin{align}
\label{eulersystemCD}\begin{CD}
\widehat{J}(\mathfrak{N}^+, v\mathfrak{N}^{-})(K_{\mathfrak{p}^{\infty}})/I_{{\bf{f}}_v}
@>{\mathfrak{K}}>>
\widehat{H}^1(K_{\mathfrak{p}^{\infty}}, T_{{\bf{f}},n})\\
@V{\widehat{\partial}_v} VV @V{\partial_v}VV \\
\widehat{\Phi}_v(\mathfrak{N}^+, v\mathfrak{N}^{-})/I_{{\bf{f}}_v} @>{(\ref{mainisom})}>>
\widehat{H}_{\sing}^1(K_{\mathfrak{p}^{\infty}, v}, T_{{\bf{f}},n}).
\end{CD}\end{align} Here, $\mathfrak{K}$ denotes the Kummer map, $\partial_v$
the residue map, and $\widehat{\partial}_v$ the induced
specialization map of $(\ref{ilspec})$. \end{corollary}

\begin{proof}
See \cite[Corollary 5.18]{BD}, the same argument applies granted 
the results of Corollary \ref{galoisid}.
\end{proof} \end{remark} ~~\\

\begin{remark}[Level raising at two primes.]

Keep all of the notations of the section above. Let us now fix two $n$-admissible primes 
$v_1, v_2 \subset\mathcal{O}_F$ with respect to ${\bf{f}}$ such that \begin{align*}{\bf{N}}(v_i) + 1 - \varepsilon_i \cdot
a_{v_1}({\bf{f}}) \equiv 0 \mod \mathfrak{P}_n\end{align*} for each of
$i=1,2$. As before, we keep all of the setup and hypotheses of Theorem
\ref{cerednikvarshavsky}, taking $v = v_1$ so that the indefinite 
quaternion algebra $B$ has discriminant $v_1 \mathfrak{N}^{-}$. 
Consider the composition of maps \begin{align*}\begin{CD} J(\mathfrak{N}^+, v_1\mathfrak{N}^{-})/I_{{\bf{f}}_{v_1}}
@>{\mathfrak{K}}>> H^1\left(K_{v_2}, \Ta_p\left(J(\mathfrak{N}^+, v_1\mathfrak{N}^{-})\right)
/I_{{\bf{f}}_{v_1}} \right)\\ @>{\phi}>> H^1(K_{v_2}, T_{{\bf{f}},n}).
\end{CD}\end{align*} Here, $\mathfrak{K}$ denotes the Kummer map, and 
$\phi$ is induced from a fixed isomorphism $\Ta_p \left(  J(\mathfrak{N}^+, v_1\mathfrak{N}^{-}) \right)
/ \mathcal{I}_{{\bf{f}}_{v_1}} \cong T_{{\bf{f}}, n}$. Now, since the representation $T_{{\bf{f}},n}$ is unramified at $v_2$, 
we have isomorphisms \begin{align*} H^1(K_{v_2}, T_{{\bf{f}},n}) &\cong
H_{\fin}^1(K_{v_2}, T_{{\bf{f}},n}) \cong
\mathcal{O}_0/\mathfrak{P}_n,\end{align*} 
where the latter isomorphism comes from Lemma \ref{2.6/2.7}.
Since $J(\mathfrak{N}^+, v_1\mathfrak{N}^{-})$ has good reduction at $v_2$, reduction mod 
$v_2$ gives the isomorphism \begin{align*}\red_{v_2}: J(\mathfrak{N}^+, v_1\mathfrak{N}^{-})(K_{v_2})
/I_{{\bf{f}}_{v_1}} \cong {\bf{J}}(\mathfrak{N}^+, v_1\mathfrak{N}^{-})\otimes \kappa_{v_2^2}/
I_{{\bf{f}}_{v_1}}.\end{align*}  Here (as before), ${\bf{J}}(\mathfrak{N}^+, v_1\mathfrak{N}^{-})$ denotes
the N\'eron model over $\mathcal{O}_{F_{v_2}}$ of the Jacobian $J(\mathfrak{N}^+, v_1\mathfrak{N}^{-})$.
Since $\mathcal{I}_{{\bf{f}}_v}$ is not Eisenstein, 
the natural inclusion $$\Div^0(M(\mathfrak{N}^+, v_1\mathfrak{N}^{-})) \subset 
\Div (M(\mathfrak{N}^+, v_1\mathfrak{N}^{-}))$$ induces an isomorphism \begin{align*}
\Div^0((\mathfrak{N}^+, v_1\mathfrak{N}^{-}))/\mathcal{I}_{{\bf{f}}_v} &\cong
\Div (M(\mathfrak{N}^+, v_1\mathfrak{N}^{-}))/\mathcal{I}_{{\bf{f}}_v}.\end{align*}
We thus obtain an injective map \begin{align*}\Div \left({\bf{M}}(\mathfrak{N}^+, v_1\mathfrak{N}^{-})^{ss}
\otimes \kappa_{v_2^2} \right) &\longrightarrow 
 {\bf{J}}(\mathfrak{N}^+, v_1\mathfrak{N}^{-}) \otimes \kappa_{v_2^2}/I_{{\bf{f}}_{v_1}}.
\end{align*} Hence, we obtain via the composition $\gamma = \phi \circ \mathfrak{K} \circ
\red_{v_2}^{-1}$ a map \begin{align}\label{gamma} \gamma: \Div \left( {\bf{M}}(\mathfrak{N}^+, 
v_1\mathfrak{N}^{-})^{ss} \otimes \kappa_{v_2^2} \right) \longrightarrow \mathcal{O}_0/\mathfrak{P}_n.
\end{align} 

Recall that $J(\mathfrak{N}^+, v_1\mathfrak{N}^{-})$ has the structure of 
a $\mathbb{T}(\mathfrak{N}^{+}, v_1\mathfrak{N}^{-})$-module,
as explained above. Let $T_w  \in \mathbb{T}(\mathfrak{N}^+, 
v_1\mathfrak{N}^{-})$ denote the Hecke operator at a prime 
$w \nmid v_1\mathfrak{N}^{+}\mathfrak{N}^{-}$, and $U_w 
\in \mathbb{T}(\mathfrak{N}^{+}, v_1\mathfrak{N}^{-})$ the 
operator at a prime $w \mid v_1 \mathfrak{N}^{+}\mathfrak{N}^{-}$.
For each such operator $T_w$, let us write 
$\overline{T}_w$ to denote the image of $T_w$
in $\mathbb{T}(\mathfrak{N}^{+}, v_1\mathfrak{N}^{-})/\mathcal{I}
_{{\bf{f}}_{v_1}}$. Similarly, for each such operator $U_w$, let us write 
$\overline{U}_w$ to denote the image of $U_w$
in $\mathbb{T}(\mathfrak{N}^{+}, v_1\mathfrak{N}^{-})/
\mathcal{I}_{{\bf{f}}_{v_1}}$. Observe that by definition of the 
homomorphism ${\bf{f}}_{v_1}$, we have the relations $\overline{T}_w 
\equiv a_w({\bf{f}}) \mod \mathfrak{P}_n$, $
\overline{U}_w \equiv a_w({\bf{f}}) \mod
\mathfrak{P}_n$, and $\overline{T}_{v_1} \equiv \varepsilon_1 \mod
\mathfrak{P}_n.$

\begin{lemma}\label{relations} The following relations hold for 
$ x \in \Div \left({\bf{M}}(\mathfrak{N}^+, v_1\mathfrak{N}^{-})^{ss} \otimes \kappa_{v_2^2} \right)$:
\begin{itemize}
\item[(i)]$\gamma(T_w \cdot x) = \overline{T}_{w}
\cdot \gamma(x)$ for all primes $w \nmid 
v_1v_2\mathfrak{N}^{+}\mathfrak{N}^{-}$ of $\mathcal{O}_F$.
\item[(ii)] $\gamma(U_w \cdot x) =
\overline{U}_w \cdot \gamma(x)$ for all primes 
$w \mid v_1\mathfrak{N}^{+}\mathfrak{N}^{-}$ of $\mathcal{O}_F$.
\item[(iii)]
$\gamma(T_{v_2} \cdot x) = \overline{T}_{v_2} \cdot \gamma(x)$
\item[(iv)] $\gamma(\Frob_{v_2} \cdot x) = \varepsilon_2 \cdot \gamma(x).$
\end{itemize}
\end{lemma}

\begin{proof} See \cite[Lemma 9.1]{BD}, the same proof carries over
to this setting. That is, by Lemma \ref{2.6/2.7}, we can identify 
$H_{\fin}^1(K_{v_2},T_{{\bf{f}},n})$ with the module \begin{align*}
T_{{\bf{f}},n}/\left(\Frob_{v_2}^2 - 1 \right)\end{align*} of 
$G_{K_{v_2}}$-coinvariants of $T_{{\bf{f}},n}$. We deduce 
from this that $\gamma$ sends a point $x$ to the image of 
$\left((\Frob_{v_2}^2 - 1)/\mathfrak{P}_n\right) x$ in 
$T_{{\bf{f}},n}/\left( \Frob_{v_2}^2 - 1 \right)$. This
implies the first two relations. The second two relations then
follow from the Eichler-Shimura relations (as given for instance in
\cite[$\S$ 10.3]{Ca}), by the same argument used in \cite[Lemma
9.1]{BD}. 
\end{proof} 

Recall that in the setup above, we let $D$ denote 
the totally definite quaternion algebra over $F$ obtained 
from $B$ by switching invariants at $\tau_1$ and $v_1$.
Hence, $\disc(D)=\mathfrak{N}^{-}$. Let $D'$ denote the totally
definite quaternion algebra over $F$ obtained from 
$D$ by switching invariants at $v_1$ and 
$v_2$. Hence, $\disc(D')= v_1v_2\mathfrak{N}^{-}$.
Let $U' \subset \widehat{D}^{' \times}$ be the 
compact open subgroup
defined by $U' = {H'}^{v_2} \times \mathcal{O}_{D'_{v_2}}^{\times}.$ 
Note that we have an isomorphism $U^{'v_2} \cong {H'}^{v_2}.$ Note
as well that by Proposition \ref{carayolss}, we have isomorphisms
\begin{align}\label{c2} \Div \left( {\bf{M}}(\mathfrak{N}^+, v_1 \mathfrak{N}^{-})^{ss}
\otimes \kappa_{v_2^2} \right) &\cong {\bf{Z}} [D'^{\times} \backslash \widehat{D}
^{' \times}/U'] \cong \mathbb{S}_2(\mathfrak{N}^+, v_1 v_2 \mathfrak{N}^{-}; {\bf{Z}}).\end{align}

\begin{proposition}\label{gammaeigenform} Keep the hypotheses of Theorem 
\ref{raiseonefree} and Corollary \ref{galoisid}. Assume that $F$ is linearly disjoint from the cyclotomic field
${\bf{Q}}(\zeta_p)$.The map \begin{align*}\gamma: \Div \left( {\bf{M}}
(\mathfrak{N}^+, v_1\mathfrak{N}^{-})^{ss}\otimes \kappa_{{v_2}^2} \right) \longrightarrow 
\mathcal{O}_0/\mathfrak{P}_n \end{align*}
constructed in $(\ref{gamma})$ is surjective, hence can be identified with 
a quaternionic eigenform in the space $\mathcal{S}_2(\mathfrak{N}^+, v_1 v_2 \mathfrak{N}^{-};
\mathcal{O}_0/\mathfrak{P}_n).$ In particular, associated to
$\gamma$ by the Jacquet-Langlands correspondence is a surjective
homomorphism \begin{align*}
 {\bf{g}}: {\bf{T}}_0(\mathfrak{N}^{+}, v_1v_2\mathfrak{N}^{-})
&\longrightarrow \mathcal{O}_0/\mathfrak{P}_n
\end{align*} such that 
\begin{itemize}
\item[(i)] $T_w \left( {\bf{g}} \right) \equiv
a_w({\bf{f}}) \mod \mathfrak{P}_n$ for all primes $w\nmid v_1
v_2\mathfrak{N}^{+}\mathfrak{N}^{-}$ of $\mathcal{O}_F$.
\item[(ii)] $U_w \left( {\bf{g}} \right) \equiv
a_w({\bf{f}}) \cdot {\bf{g}} \mod \mathfrak{P}_n$ for 
all primes $w\mid \mathfrak{N}^{+}\mathfrak{N}^{-}$ of $\mathcal{O}_F$.
\item[(iii)] $U_{v_i} \left( {\bf{g}} \right) \equiv \varepsilon_i \cdot
{\bf{g}} \mod \mathfrak{P}_n$ for each of $i =1,2$.
\end{itemize}\end{proposition} 

\begin{proof}
Granted that $\gamma$ is surjective, the result follows from Lemma \ref{relations}, along with the 
identifications of $(\ref{c2})$. The surjectivity of $\gamma$ is shown by Lemma \ref{p-torsion} and 
Proposition \ref{surjective} below. \end{proof}

\begin{lemma}\label{p-torsion}
If $F$ is linearly disjoint from the cyclotomic field
${\bf{Q}}(\zeta_p)$, then the subgroup of unipotent matrices mod
$\mathfrak{p}$ in $M_2(F)$ is torsionfree.
\end{lemma}

\begin{proof} The following proof was suggested to the author by Vladimir Dokchitser.
We first show that $\M(F)$ has no nontrivial matrices of order $p$.
That is, suppose otherwise that we had a matrix $A \neq {\bf{1}}$ of order $p$ 
in $\M(F)$. The eigenvalues of $A$ would then be $p$-th roots of $1$.
Hence, the trace of $A$ would lie in ${\bf{Q}}(\zeta_p)$, contradicting
our hypotheses on $F$. Now to prove the claim, we show that any matrix 
$B \in \M(F)$ that is unipotent mod $p$ must have $p$-power order. 
That is, suppose otherwise that such a matrix $B$ did not have 
$p$-power order. Then, its eigenvalues would be $m$-th roots 
of unity for some integer $m$ prime to $p$. But these eigenvalues 
cannot be congruent to $1$ mod $p$, as the polynomial 
$X^m-1$ is coprime to its own derivative mod $p$, hence has distinct 
roots over $\overline{\bf{F}}_p$. 
\end{proof} 

\begin{proposition}\label{surjective} Assume Condition 3. of Theorem \ref{RESULT}. If $F$ is linearly disjoint from 
the cyclotomic field ${\bf{Q}}(\zeta_p)$, then the map $\gamma$ constructed in $(\ref{gamma})$ above is surjective. \end{proposition}

\begin{proof} We generalize the argument of \cite[Theorem 9.2]{BD}, using the 
version of Ihara's lemma for Shimura curves shown in the main result of \cite{Ch2}. 
Hence, keep the setup of Theorem \ref{cerednikvarshavsky}, with $v = v_1$. Let us write 
\begin{align*} M(\mathfrak{N}^+, v_1 \mathfrak{N}^{-})({\bf{C}})&= B'^{\times}\backslash 
\widehat{B}'^{\times} \times X /H' = \coprod_i \Gamma_i \backslash \mathfrak{H},
\end{align*} where the subgroups $\Gamma_i \subset B'^{\times}$ are the associated arithmetic
subgroups (see for instance the definition given in \cite[$\S$ 3]{CV}). By embedding $B' 
\longrightarrow B_{\mathfrak{p}}^{\times}$, we view these arithmetic subgroups $\Gamma_i$ as
(discrete) subgroups of $\GL(F_{\mathfrak{p}})$. Let $J(\N^+, v_1 \N^{-})^{ss} \otimes \kappa_{v_2^2}$
denote the subgroup of $J(\N^+, v_1 \N^{-}) \otimes \kappa_{v_2^2}$ generated by divisors supported 
on supersingular points. Since the composition of maps defining the homomorphism \begin{align*}
J(\N^+, v_1 \N^{-}) \otimes \kappa_{v_2^2}/ \mathcal{I}_{{\bf{f}}_{v_1}} &\longrightarrow 
\mathcal{O}_0/ \mathfrak{P}_n \end{align*} is surjective, it suffices to show that the 
image of $J(\N^+, v_1 \N^{-})^{ss} \otimes \kappa_{v_2^2}$ in $J(\N^+, v_1 \N^{-}) \otimes \kappa_{v_2^2}
/ \mathcal{I}_{{\bf{f}}_{v_1}}$ fills the whole group. To this end, let us define subgroups \begin{align*}
\Gamma_i(v_2) &= \left( \Gamma_i \left[ \frac{1}{v_2} \right]^{\times} / \mathcal{O}_F^{\times}\left[
\frac{1}{v_2}\right]  \right)^1, ~~~\Gamma(v_2) = \prod_i \Gamma_i(v_2)
\end{align*} where the superscript $1$ denotes elements of reduced norm $1$. Let $\widetilde{M}(\N^+,
v_1 \N^{-})$ denote the Shimura curve obtained from $M(\N^+. v_1 \N^{-})$ obtained by imposing extra
$H_{\mathfrak{p}}'^{1}$-level structure, with $\widetilde{J}(\N^+,v_1 \N^{-})$ its jacobian. Let us then 
write the corresponding arithmetic subgroups  that are congruent modulo $\mathfrak{p}$ to unipotent matrices as :
\begin{align*} \widetilde{\Gamma}_i(v_2)  &\subset \Gamma_i(v_2), ~~~\widetilde{\Gamma}(v_2) = \prod_i 
\widetilde{\Gamma}_i(v_2). \end{align*} Since the subgroups $\widetilde{\Gamma}_i(v_2)$ are torsionfree
by Lemma \ref{p-torsion}, a general theorem of Ihara (\cite[Theorem G]{Ih}) implies that there is a canonical 
isomorphism \begin{align}\label{ihara} \widetilde{J}(\N^+, v_1 \N^{-}) \otimes \kappa_{v_2^2}/  \widetilde{J}(\N^+, v_1 \N^{-})^{ss} 
\otimes \kappa_{v_2^2} &\cong \widetilde{\Gamma}(v_2)^{\ab}. \end{align} Here, $ \widetilde{\Gamma}(v_2)^{\ab}$
denotes the abelianization of $ \widetilde{\Gamma}(v_2)$. Since $v_2$ splits the quaternion algebra
$B'$ associated to $M(\N^+, v_1 \N^{-})$, we can fix an embedding \begin{align*}
\iota: B' &\longrightarrow B_{v_2}' \cong  M_2(F_{v_2}), \end{align*} to obtain an induced action of 
$B'$ on the Bruhat-Tits tree $\mathcal{T}_{v_2} = (\mathcal{V}(\mathcal{T}_{v_2}), \mathcal{E}(\mathcal{T}_{v_2}))$ 
of $B_{v_2}'^{\times}/F_{v_2}^{\times} \cong \PGL(F_{v_2})$. Let us then fix a vector of vertices $\mathfrak{v}_0 = \lbrace \mathfrak{v}_0^i
\rbrace_{i=1}^h$ in $\mathcal{V}(\mathcal{T}_{v_2})$ such that the stabilizer \begin{align*} \widetilde{\Gamma}_{\mathfrak{v}_0^i}(v_2)
&:= \Stab_{\mathfrak{v}_0^i} \left(\widetilde{\Gamma}_i(v_2)\right) \end{align*} for each index $i$ can be identified with the 
image of $\Gamma_i(v_2)$ in $\widetilde{\Gamma}_i(v_2)$ via $\iota$, so that we have an identification of Riemann
surfaces \begin{align*} \widetilde{M}(\N^+, v_1 \N^{-})({\bf{C}}) &= \coprod_i^h  \widetilde{\Gamma}_{\mathfrak{v}_0^i}(v_2)
\backslash \mathfrak{H}.  \end{align*} Fix a vector of oriented edges $\mathfrak{e}_0 = \lbrace \mathfrak{e}_0^i
\rbrace_{i=1}^h$ in $\mathcal{E}^*(\mathcal{T}_{v_2})$ such that the stabilizer \begin{align*} \widetilde{\Gamma}_{\mathfrak{e}_0^i}(v_2)
&:= \Stab_{\mathfrak{e}_0^i} \left(\widetilde{\Gamma}_i(v_2)\right) \end{align*} for each index $i$ can be identified with the 
subgroup of upper triangular matrices mod $v_2$ via $\iota$, so that we have an identification of Riemann
surfaces \begin{align*} \widetilde{M}(v_2 ; \N^+, v_1 \N^{-})({\bf{C}}) &= \coprod_i^h  \widetilde{\Gamma}_{\mathfrak{e}_0^i}(v_2)
\backslash \mathfrak{H}.  \end{align*} Let $\mathfrak{v}_1$ denote the vector of vertices $\lbrace \mathfrak{v}_1^i \rbrace_{i = 1}^h$
in $\mathcal{V}(\mathcal{T}_{v_2})$ such that $\mathfrak{v}_1^i = t(\mathfrak{e}_0^i)$ for each index $i$. Let us for ease of notation
write the products as \begin{align*} \widetilde{\Gamma}_{\mathfrak{v}_0} (v_2) &= \prod_{i=1}^h \widetilde{\Gamma}_{\mathfrak{v}_0^i}(v_2),
~ \widetilde{\Gamma}_{\mathfrak{e}_0} (v_2) = \prod_{i=1}^h \widetilde{\Gamma}_{\mathfrak{e}_0^i}(v_2), ~
\widetilde{\Gamma}_{\mathfrak{v}_1} (v_2) = \prod_{i=1}^h \widetilde{\Gamma}_{\mathfrak{v}_1^i}(v_2).\end{align*} Hence we obtain 
from Serre \cite[Proposition 1.3 $\S$ II.2.8]{Se} (with $i =1$, $M = \kappa_{\mathfrak{p}}$, and $G = \widetilde{\Gamma}(v_2)$) the exact sequence
\begin{align}\label{serre}\begin{CD} \Hom(\widetilde{\Gamma}(v_2), \kappa_{\mathfrak{p}}) @>>> \Hom(\widetilde{\Gamma}_{\mathfrak{v}_0}(v_2), 
\kappa_{\mathfrak{p}}) \oplus \Hom(\widetilde{\Gamma}_{\mathfrak{v}_1}(v_2), \kappa_{\mathfrak{p}}) \\ @>{d}>> \Hom(\widetilde{\Gamma}_{\mathfrak{e}_0}(v_2), 
\kappa_{\mathfrak{p}}). \end{CD} \end{align} Now, via duality we see that the map $d$ in $(\ref{serre})$ is the degeneracy map of Ihara's lemma for Shimura curves, as 
described for instance in \cite{Ch2} (cf. \cite[Theorem 2, p. 451]{DT} with \cite[Proposition 9.2]{BD}). Roughly, Ihara's lemma is the assertion that for 
any non-Eisenstein maximal ideal $\mathfrak{m}  \subset \mathbb{T}(v_2; \N^+, v_1 \N^{-})$, the natural degeneracy map \begin{align*}
\begin{CD}H^1(M(v_2 ; \N^+, v_1 \N^{-}), \kappa_{\mathfrak{p}})^{\oplus 2} @>{1_* \oplus \eta_{\mathfrak{p}}}>> 
H^1(M(\mathfrak{p} v_2; \N^+, v_1 \N^{-}), \kappa_{\mathfrak{p}})\end{CD} \end{align*} is injective after localization at $\mathfrak{m}$.
This conjecture is proved in certain cases in the unpublished manuscript \cite{Ch2}. We shall invoke this result in the following way. Recall
that we let $\mathfrak{m}_{{\bf{f}}_{v_1}} \supset \mathcal{I}_{{\bf{f}}_{v_1}}$ denote the maximal ideal of the Hecke algebra $\mathbb{T}(\N^+, v_1 \N^{-})$ 
corresponding to the mod $\mathfrak{P}_n$ eigenform ${\bf{f}}_{v_1}$ of Theorem \ref{raiseonefree}. Let us write $\widetilde{\mathfrak{m}}_{{\bf{f}}_{v_1}} 
\supset \widetilde{\mathcal{I}}_{{\bf{f}}_{v_1}}$ to denote the corresponding maximal ideal in the Hecke algebra $\widetilde{\mathbb{T}}(\N^+, v_1 \N^{-})$ associated to 
$\widetilde{M}(\N^+, v_1 \N^{-})$. Since $\widetilde{\mathfrak{m}}_{{\bf{f}}_{v_1}}$ is associated to an irreducible Galois representation, 
we know by Proposition \ref{jarvis} that $\widetilde{\mathfrak{m}}_{{\bf{f}}_{v_1}}$ is not Eisenstein. Hence, by Ihara's lemma for Shimura curves, the degeneracy map $d$ is injective after 
localization at $\widetilde{\mathfrak{m}}_{{\bf{f}}_{v_1}}$. We can then argue following \cite[Theorem 9.2, p. 59]{BD} that $ \Hom(\widetilde{\Gamma}(v_2), \kappa_{\mathfrak{p}}) 
[\widetilde{\mathfrak{m}}_{{\bf{f}}_{v_1}}]=0$. Hence, $\widetilde{\Gamma}(v_2)^{\ab}/\widetilde{\mathfrak{m}}_{{\bf{f}}_{v_1}} =0$, in which case it follows from Nakayama's lemma that \begin{align*}
\widetilde{\Gamma}(v_2)^{\ab}/\widetilde{\mathcal{I}}_{{\bf{f}}_{v_1}} =0. \end{align*} Hence, by $(\ref{ihara})$, the image of $\widetilde{J}(\N^+, v_1 \N^{-})^{ss} \otimes \kappa_{v_2^2}$ in 
$\widetilde{J}(\N^+, v_1 \N^{-}) \otimes \kappa_{v_2^2}/ \widetilde{\mathcal{I}}_{{\bf{f}}_{v_1}}$ fills the whole group. To complete the argument, consider the natural map
\begin{align}\label{natural} \widetilde{J}(\N^+, v_1 \N^{-}) \otimes \kappa_{v_2^2} &\longrightarrow J(\N^+, v_1 \N^{-}) \otimes \kappa_{v_2^2}.\end{align} A standard argument with Shimura subgroups shows 
that the cokernel of this map $(\ref{natural})$ has order prime to $p$ (see \cite{Li} with \cite[Lemma 7.20]{L30}). Roughly, the idea is the following. Let $\Pi$ denote the kernel of this natural map.
The criterion of \cite{Li} applied to each connected component of $M(\N^+, v_1 \N^{-})$ shows that there is an injective map \begin{align*}
 \Pi &\longrightarrow \Hom(\Gamma(v_2)/\Gamma_{\mathfrak{v}_0}(v_2), {\bf{S}} ), \end{align*} where ${\bf{S}}$ denotes the complex numbers of modulus $1$. It is 
 then easy to see that the order of $\Pi$ must be prime to $p$. Hence by duality, the cokernel must have over prime to $p$. Hence, the composition of $(\ref{natural})$ 
 with the projection \begin{align*} J(\N^+, v_1 \N^{-}) \otimes \kappa_{v_2^2} \longrightarrow  J(\N^+, v_1 \N^{-}) \otimes \kappa_{v_2^2}/ \mathfrak{P}_n \end{align*} is surjective. Now, since
we have already  shown that the natural map \begin{align*} \widetilde{J}(\N^+, v_1 \N^{-})^{ss} \otimes \kappa_{v_2^2} &\longrightarrow 
\widetilde{J}(\N^+, v_1 \N^{-}) \otimes \kappa_{v_2^2}/ \widetilde{\mathcal{I}}_{{\bf{f}}_{v_1}}\end{align*} is surjective, we see that the natural map \begin{align*}
J(\N^+, v_1 \N^{-})^{ss} \otimes \kappa_{v_2^2} &\longrightarrow  J(\N^+, v_1 \N^{-})  \otimes \kappa_{v_2^2} / \mathcal{I}_{{\bf{f}}_{v_1}}\end{align*} is surjective, 
as required. \end{proof} Let us for the record state the version of Ihara's lemma for Shimura curves over totally real fields used in the
proof of Propostion \ref{surjective} above.

\begin{hypothesis}[Ihara's lemma for Shimura curves.]\label{ihara}
For any non-Eisenstein maximal ideal $\mathfrak{m}  \subset \mathbb{T}(v_2; \N^+, v_1 \N^{-})$, the natural degeneracy map \begin{align*}
\begin{CD}H^1(M(v_2 ; \N^+, v_1 \N^{-}), \kappa_{\mathfrak{p}})^{\oplus 2} @>{1_* \oplus \eta_{\mathfrak{p}}}>> 
H^1(M(\mathfrak{p} v_2; \N^+, v_1 \N^{-}), \kappa_{\mathfrak{p}})\end{CD} \end{align*} is injective after localization at $\mathfrak{m}$.
\end{hypothesis} 

\end{remark}

\section{Construction of the Euler system}

We first review the theory of CM points on Shimura curves 
over totally real fields, giving a characterization of the images
of these points under reduction modulo a ramified prime for 
the underlying quaternion algebra. We then review the notion
of the Hodge embedding, then use this to define suitable classes
from the images of CM points in the $v$-adic uniformization.

\begin{remark}[Setup.]

Fix a finite prime $v \subset \mathcal{O}_F$.
Recall that we fixed a totally imaginary quadratic extension
$K$ of $F$. Let us suppose that $v$ remains inert in $K$. Writing $K_v$ 
to denote the completion of $K$ at the prime above $v$ in $K$, and 
$F_{v^2}$ to denote the quadratic unramified extension of $F_v$, we 
have isomorphism of fields $K_v \cong F_{v^2}$. Let us then fix such 
an isomorphism $K_v \cong F_{v^2}$. \end{remark}

\begin{remark}[Complex points.] 

Suppose in general that we have any Shimura curve $M_H$, as defined 
above. Given elements $b \in \widehat{B}^{\times}$ and $z \in X$, we write 
$\left[b, z\right]_H$ to denote the point of $M_H({\bf{C}})$ 
represented by the pair $(b, z)$. \end{remark}

\begin{remark}[Embeddings.] 

Let us return to the setup of Theorem \ref{cerednikvarshavsky}. Hence,
no prime dividing $\disc(D)$ is split in $K$. 
Under this assumption, there exists an injective $F$-algebra homomorphism
$\iota: K \longrightarrow D$ (see \cite{Vi}). Let us fix 
such a homomorphism $\iota$, writing $\iota_v: K \otimes_F F_v 
\longrightarrow D_v$ to denote its component at a prime $v$ of $F$,
and $\widehat{\iota}: \widehat{K} \longrightarrow \widehat{D}$ to 
denote its adelization. Let us assume that for our fixed prime $v$, we have 
the identification $\iota_v^{-1}(U_v) = \mathcal{O}_{K_v}^{\times}$. 
Under this assumption, there exists an $F$-algebra 
injection $\iota':K \longrightarrow B'$ such that ${\iota_v'}^{-1}
(\mathcal{O}_{B'_v}) = \mathcal{O}_{K_v}$. Let us fix such an embedding $\iota'$. 
Since the Skolem-Noether theorem implies that any two local embeddings 
$\iota_v$, $\iota'_v$ are conjugate by an element of $B_v^{' \times}$, we can 
an will assume that the homomorphisms $\varphi$, $\iota$ and $\iota'$ are 
compatible outside of $v$ in the sense that  ${\iota'}^v$ is given by 
the composition \begin{align*}\begin{CD}\widehat{K}^v @>{\iota^v}>> 
\widehat{D}^{v} @>{\varphi}>> \widehat{B}'^v.\end{CD}\end{align*}
\end{remark}

\begin{remark}[CM points.] 

Fix an embedding $K \rightarrow {\bf{C}}$ that extends $\tau_1: F 
\rightarrow {\bf{R}} \subset {\bf{C}}$. The action of $\iota_{\tau_1}'(K^{\times}) 
\subset \iota_{\tau_1}'(K_{\tau_1}^{\times}) \subset {B'}_{\tau_1}^{\times} \cong \GL({\bf{R}})$
on $X = \mathfrak{H}^{\pm}={\bf{C}}-{\bf{R}}$ fixes exactly two points, 
of which one lies in the complex upper half plane $\mathfrak{H}^{+}$.
Let us write $z' = z'_{\iota'}$ to denote this point. We then 
define the set of points with complex multiplication (CM) by $K$ 
on $M_{H'}({\bf{C}})$ to be: \begin{align*}
\CM(M_{H'},K) &= \lbrace \left[b', z'\right]_{H'}: 
b' \in \widehat{B}'^{\times} \rbrace  \subset M_{H'}({\bf{C}}).\end{align*}
By Shimura's reciprocity law, the set $\CM(M_{H'}, K)$ is contained in 
$ M_{H'}(K^{\ab})$, with $G_K^{\ab}$ acting on $\CM(M_{H'},K)$ 
by the rule \begin{align}\label{galoisaction} \forall a \in \widehat{K}^{\times},
~~~  \rec_K(a) \left[ b', z' \right]_{H'}
&= \left[ \widehat{\iota}'(a)b', z'\right]_{H'}.
\end{align} Here, \begin{align*}
\rec_K: {\bf{A}}_K^{\times}/K^{\times} &\cong G_K^{\ab}\end{align*}
denotes the reciprocity map of class field theory, normalized to 
send uniformizers to their corresponding geometric Frobenius elements. 
\end{remark}

\begin{remark}[CM points of a given conductor.] 

Recall that given an ideal $\mathfrak{c} \subset \mathcal{O}_F$, we let 
$\mathcal{O}_{\mathfrak{c}} = \mathcal{O}_F + \mathfrak{c}\mathcal{O}_K$ 
denote the $\mathcal{O}_F$-order of conductor $\mathfrak{c}$ in $K$. 
Given a maximal order $R' \subset B'$, we call an embedding 
$\iota: K \longrightarrow B'$ an {\it{optimal embedding of 
conductor $\mathfrak{c}$ into $R'$}} if \begin{align*} 
\iota(\mathcal{O}_{\mathfrak{c}}) = \iota(K) \cap R'.
\end{align*} Given a maximal order $R' \subset B'$ and an element 
$b' \in \widehat{B}'^{\times}$, let $R'_{b'}\subset B'$ denote
the maximal order defined by ${b'}^{-1}R'b'\cap B'$. We say that a
point $\left[b', z' \right]_{H'} = \left[b', z'_{\iota'} \right]_{H'}$ 
in $ \CM(M_{H'},K)$ has {\it{conductor $\mathfrak{c}$}} if the 
associated embedding $\iota':K \longrightarrow B' $ is an optimal
embedding of conductor $\mathfrak{c}$ (see \cite[Ch. III]{Vi}). It 
is then simple to see from class field theory that a CM point of 
conductor $\mathfrak{c}$ in $\CM(M_{H'},K)$ is defined over the 
ring class field $K[\mathfrak{c}]$ of conductor $\mathfrak{c}$ 
over $K$. \end{remark}

\begin{remark}[CM points in the $v$-adic uniformization.] 

We have the following description of $\CM$ points by $K$ on 
$M_{H'}({\bf{C}})$ in the $v$-adic uniformization. Let \begin{align*}
\CM(M_{H'},K)_{\operatorname{v-unr}} &= \lbrace \left[b', z'\right]_{H'}: 
b' \in \widehat{B}'^{\times}, b'_v=1 \rbrace \subset \CM(M_{H'},K)\end{align*}
to be the subset of $\CM$ points by $K$ on $M_{H'}({\bf{C}})$ that 
are unramified outside of $v$. The the action of $G_K^{\ab}$ on this 
set of points is given by the rule \begin{align*} \forall a \in 
\widehat{K}^{\times}, ~~~  \rec_K(a) \left[ b', z' \right]_{H'} 
&= \left[ \widehat{\iota}'(a^{v})b', z' \right]_{H'}.
\end{align*} Here, $a^{v}$ denotes the projection of $a$ 
to $\widehat{K}^{v  \times} = \lbrace x \in \widehat{K}^{\times}: x_v =1 \rbrace $.
From this action, we deduce (cf. \cite[1.8.2]{Nek}) that a point 
$x = \left[ b', z'\right]_{H'} \in \CM(M_{H'},K)_{\operatorname{}v-unr}$ 
is defined over the finite abelian extension $K(x)$ of $K$
characterized by the isomorphism \begin{align*} \rec_K:  \widehat{K}^{\times}/
K^{\times} \widehat{\iota}'^{-1}(b'H'b') &\cong \Gal(K(x)/K). \end{align*}
Moreover, as observed in \cite[1.8.2]{Nek}, the prime $v$ splits
completely in $K(x)$ because $\iota_v'(\mathcal{O}_{K_v}^{\times}) \subset 
H_v' = b_v' H_v b_v^{-1}$. Fix an embedding $K_v \rightarrow F_v^{\unr}$,
equivalently an isomorphism $K_v \cong F_{v^2}$ over $F_v$. This 
choice determines one of two fixed points for the action of
$\iota_v(K^{\times}) \subset \iota_v(K_v^{\times})\subset \GL(F_v)$ on
${\bf{P}}^1(K_v)-{\bf{P}}^1(F_v)$, call it $z$ ($=z_{\iota}$). The 
image of $\CM(M_{H'},K)_{\operatorname{v-unr}} $ in $M_{H'}(K_v)$ 
according to either Theorem \ref{cv} is then given by 
\begin{align*} \CM(M_{H'}, K)_{\operatorname{v-unr}} =
\lbrace \left[ d, z \right]_{U^v}:
d \in \widehat{D}^{\times}, d_v = 1 \rbrace \subset M_{H'}(K_v).
\end{align*} 

Let us now write $\mathcal{G}_{v} = (\mathcal{V}(\mathcal{G}_{v}),
\mathcal{E}(\mathcal{G}_{v}))$ to denote the dual graph of 
the special fibre ${\bf{M}}_{H'} \otimes \kappa_{v^2}$, which
is just the special fibre of the basechange to $\mathcal{O}_{F_{v^2}}$ 
of the integral model ${\bf{M}}_{H'}$ over $\mathcal{O}_{F_v}$. 
Let \begin{align*}
\red_v: M_{H'}\otimes_F F_{v^2} \longrightarrow 
\mathcal{V}(\mathcal{G}_{v}) \cup \mathcal{E}(\mathcal{G}_{v})
\end{align*} denote the map that sends a point $x$ to 
either the connected component containing its image in 
${\bf{M}}_{H'} \otimes \kappa_{v^2}$, or else to its image
in ${\bf{M}}_{H'}\otimes \kappa_{v^2}$ (a singular point). 

\begin{proposition}\label{modv} We have that
$\red_v \left( \CM(M_{H'}, K)_{\operatorname{v-unr}}\right) 
\subset \mathcal{V}(\mathcal{G}_{v})$. \end{proposition}

\begin{proof} The result seems to be
well known (see \cite[$\S$ 5]{BD} or \cite[5.2]{Lo}). 
That is, fix a point $x \in \CM(M_{H'}, K)_{\operatorname{v-unr}}$.
We saw above that $x$ is rational over the abelian extension
$K(x)$, where $v$ splits completely. Hence, writing $K(x)_v$
to denote the localization of $K(x)$ at any fixed prime above 
$v$ in $K(x)$, we have the identification $K(x)_v \cong K_v$.
In particular, we may view $x$ as a point in $M_{H'}(K_v) 
\cong M_{H'}(F_{v^2})$. It therefore makes sense to compute 
the image of $x$ in ${\bf{M}}_{H'} \otimes \kappa_{v^2}$.  
 
Now, recall that the image of $x$ in the $v$-adic uniformization
${\bf{M}}_{H'} \otimes \kappa_{v^2}$ is parametrized by the class 
of a pair $(d, \iota)$, where $d \in \widehat{D}^{\times v}$, and 
$\iota: K \rightarrow D$ is a suitably chosen $F$-algebra injection. 
The action of $\iota(K^{\times}) \subset \iota_v(K_v^{\times}) \subset 
D_v^{\times} \cong \GL(F_v)$ on $\Omega({\bf{C}}_v)$ fixes two distinct 
points, $z_1 = z_{1, \iota}$ and $z_2 = z_{2, \iota}$ say.  These
points are contained in ${\bf{P}}^1({K_v}) - {\bf{P}}^1(F_v) \cong
{\bf{P}}^1(F_{v^2}) - {\bf{P}}^1(F_v)$. Let $z_3$ denote the 
point at infinity in ${\bf{P}}^1(K_v) \cong {\bf{P}}^1(F_{v^2})$.
As explained in Mumford \cite[$\S$ 1]{Mum}, any triple of 
distinct points $z_1, z_2, z_3$ in ${\bf{P}}^1(F_{v^2})$ corresponds
canonically to a unique vertex $\mathfrak{v}_{z_1, z_2, z_3}$ in
the Bruhat-Tits tree of $\SL(F_{v^2})$. The inclusion 
$\red_v(\CM(M_{H'}, K)_{\operatorname{v-unr}}) \subset \mathcal{V}
(\mathcal{G}_{v})$ can then be deduced from the Mumford-Kurihara
uniformization of ${\bf{M}}_{H'} \otimes \kappa_{v^2}$ (Theorem
\ref{mumfordkurihara}) with $(\ref{cv})$.   \end{proof}
Note that since any $n$-admissible prime $v \subset \mathcal{O}_F$
with respect to $\ff$ splits completely in $K_{\mathfrak{p}^{\infty}}$,
the argument of Propostion \ref{modv} shows that the CM points
in $M(K_{\mathfrak{p}^{\infty}})$ also satisfy this property under reduction
mod $v$.\end{remark}

\begin{remark}[Hodge classes.]

Let us now explain how divisors on $M_{H'}(\overline{F})$
give rise to classes in the associated jacobian $J_{H'}(\overline{F})$,
following \cite[$\S$5]{Lo} and \cite[$\S$3]{L30}. Recall that
$J_{H'}$ has the structure of a ${\bf{T}}_0(\mathfrak{N}^{+},
v\mathfrak{N}^{-})$-module, as explained above. Following 
Zhang \cite[4.1]{Z0}, we make the following

\begin{definition} The {\it{Hodge class of $M_{H'}$}} is the unique
class $\xi \in \Pic(M_{H'})$ such that:

\begin{itemize}
\item[(i)] The degree on $\xi$ on each connected component of
$M_{H'}$ is one.
\item[(ii)] The action of the operator
$T_{\mathfrak{q}} \in {\bf{T}}_0(\mathfrak{N}^{+}, v\mathfrak{N}^{-})$
for each prime $\mathfrak{q} \nmid \mathfrak{N}$ is 
given by mutliplication by ${\bf{N}}(\mathfrak{q}) +1$. \end{itemize}
\end{definition} The existence and uniqueness of such a class are
shown by Zhang in \cite[4.1]{Z0}. Let $\Pic^{\operatorname{Eis}}(M_{H'})$ 
denote the subgroup of $\Pic(M_{H'})$ generated by divisors whose restriction to
each connected component of $M_{H'}$ is given by a multiple
of the Hodge class $\xi$. Zhang \cite{Z}[6.1] shows that
there is a decomposition \begin{align*} \Pic(M_{H'}) &=
\Pic^0(M_{H'}) \oplus \Pic^{\operatorname{Eis}}(M_{H'}).\end{align*}
Using an argument of Ribet \cite[Theorem 5.2 (c)]{Ri2}, or its
subsequent generalization by Jarvis in \cite[$\S$ 3]{Jar2}, it 
can then be shown that the ${\bf{T}}_0(\mathfrak{N}^{+},
v\mathfrak{N}^{-})$-module  $\Pic^{\operatorname{Eis}}(M_{H'})$ 
is Eisenstein. Since $\mathcal{I}_{{\bf{f}}_v} \subset 
{\bf{T}}_0(\mathfrak{N}^{+}, v\mathfrak{N}^{-})$ is 
not Eisenstein, it can then be deduced by a standard argument that the natural inclusion 
$\Pic^0(M_{H'}) \subset \Pic(M_{H'})$ induces 
an isomorphism \begin{align*} \Pic^0(M_{H'})/
\mathcal{I}_{{\bf{f}}_v} \cong \Pic(M_{H'})/I_{{\bf{f}}_v}.
\end{align*} See for instance \cite[2.5]{L30} (where there is 
a typo on line 21 of p. 15). \end{remark}

\begin{remark}[Construction of classes.]

Let us now assume that $M_{H'} = M(\mathfrak{N}^+, v\mathfrak{N}^{-})$
as above, where $v \subset \mathcal{O}_F$ is an $n$-admissible prime 
with respect to $\ff$. Fix a sequence of points $\lbrace P_m \rbrace_{m \geq 1}$,
where each $P_m$ is a CM point of conductor $\mathfrak{p}^m$ in 
$\CM(M_{H'}, K)$. Let us assume that this sequence is compatible
in the following sense. Each point $P_m$ is given by the class of some 
pair $(g_m', z') = (g_m', z'_{\iota'})$, where $\iota'$ is an optimal
embedding of $\mathcal{O}_{\mathfrak{p}^m} \subset \mathcal{O}_K$ into 
$R_{g_m'} = {g_m'}^{-1}Rg_m' \subset B'$. By assumptions made made above and
throughout, we can fix an isomorphism $\iota_{\mathfrak{p}}: B'_{\mathfrak{p}} \cong
\M(F_{\mathfrak{p}})$. Following the explanation of \cite[3.3]{Lo2},
we can then associate to the the local (Eichler) order 
$(R_{g_m'})_{\mathfrak{p}} \subset B'_{\mathfrak{p}}$ a directed edge $e_{g_m'}
= (s(e_{g_m'}), t(e_{g_m'}))$ in the Bruhat-Tits tree of $B_{\mathfrak{p}}^{' \times}
/F_{\mathfrak{p}}^{\times} \cong \PGL(F_{\mathfrak{p}})$. We then say that 
the sequence $\lbrace P_m \rbrace_{m \geq 1} $  is {\it{compatible}} if $t(e_{g_m'}) = 
s(e_{g_{m+1}'})$ for all $m \geq 1$. Let us now fix such an oriented
sequence $\lbrace P_m \rbrace_{m \geq 1}$. For each point $P_m$ in the 
sequence, let us write $P_m^*$ to 
denote the image of $\alpha_{\mathfrak{p}}^{-m}P_m$ in 
$J_{H'}(K[\mathfrak{p}^m])/\mathcal{I}_{{\bf{f}}_v}$. The points 
$P_m^*$ are norm compatible, and their images under the Kummer maps
\begin{align*}\begin{CD} J_{H'}(K[\mathfrak{p}^m])/\mathcal{I}_{{\bf{f}}_v}
@>{\mathfrak{K}}>> H^1(K[\mathfrak{p}^m], \Ta_p(J_{H'})/\mathcal{I}_{{\bf{f}}_v})
\end{CD}\end{align*} give rise to a sequence of classes $\zeta_m[v]$
that are compatible under corestriction maps. Under the isomorphism
of Corollary \ref{galoisid} (ii), these classes give rise to an element 
$\zeta[v]$ in the cohomology group $\widehat{H}^1(K[\mathfrak{p}^{\infty}], 
T_{{\bf{f}},n}).$ Let $\zeta(v)$ denote image of this class under corestriction from 
$K[\mathfrak{p}^{\infty}]$ to $K_{\mathfrak{p}^{\infty}}$. Hence, we have constructed
from our compatible system of CM points $\lbrace P_m \rbrace_{m \geq 1}$
in $\CM(M_{H'},K)$ a class \begin{align}\label{zetaclass} \zeta(v) 
\in \widehat{H}^1(K_{\mathfrak{p}^{\infty}}, T_{{\bf{f}},n}).\end{align}
\end{remark}

\section{Explicit reciprocity laws}

Putting everything together, we may now at last deduce the 
first and second explicit reciprocity laws introduced above.

\begin{remark}[The first explicit reciprocity law.]

Keep all of the notations and hypotheses above. 
Hence, ${\bf{f}} \in \mathcal{S}_2(\mathfrak{N}^{+},
\mathfrak{N}^{-})$ is a $\mathfrak{p}$-ordinary eigenform, 
and $\zeta(v)$ is the class of $\widehat{H}^1(K_{\mathfrak{p}^{\infty}}, T_{{\bf{f}},n})$
constructed above.

\begin{theorem}\label{ERL1}\emph{(The first explicit reciprocity law)}
Keep all of the hypotheses of Theorem \ref{raiseonefree} and Corollary \ref{galoisid}. 
Then, $\vartheta_v\left(\zeta(v)\right) =0$. Moreover, the equality \begin{align*}\partial_v\left(
\zeta(v) \right) &= \mathcal{L}_{{\bf{f}}}
\end{align*} holds in $\widehat{H}_{\sing}^1(K_{\mathfrak{p}^{\infty}, v}, T_{{\bf{f}},n})
\cong \Lambda/\mathfrak{P}^n,$ up to multiplication
by elements of $\mathcal{O}^{\times}$ or
$G_{\mathfrak{p}^{\infty}}$.
\end{theorem}

\begin{proof} See \cite[$\S$8]{BD}. By the commutative 
diagram of Corollary \ref{ESid} (ii),
it suffices to show that \begin{align*} \widehat{\partial}_v 
\left( \lbrace P_m^* \rbrace \right) \equiv 
\mathcal{L}_{{\bf{f}}} \mod \mathfrak{P}^n. \end{align*}
Let us write 
$Q_m$ to denote the image of the class $P_m^*$ under the
norm map from $K[\mathfrak{p}^m]$ to $K_{\mathfrak{p}^m}$.
By Proposition \ref{modv}, the image of each class
$Q_m$ in the special fibre ${\bf{M}}^{(v)} \otimes \kappa_{v^2}$
is a nonsingular point, hence given by a vertex
$v_{Q_m}$ in the dual graph $\mathcal{G}_v$ under the map $\red_v: J^{(v)}(K_{\mathfrak{p}^m})/
\mathcal{I}_{{\bf{f}},v} \longrightarrow \mathcal{V}(\mathcal{G}_v)
\cup \mathcal{E}(\mathcal{G}_v)$. On the other hand, recall the
natural map $\omega_v: {\bf{Z}}[\mathcal{V}(\mathcal{G}_v)]^0
\longrightarrow \Phi_v $ constructed in Proposition \ref{omeganat}.
We know by Proposition \ref{cmpointreduction} that \begin{align*}
 \omega_v \circ \red_v(Q_m) = \partial_v (Q_m) \in \Phi_v/
\mathcal{I}_{{\bf{f}}_v} \cong \mathcal{O}_0/\mathfrak{P}_n,\end{align*}
where the isomorphism comes from Corollary \ref{galoisid} (i). The 
result can now be deduced from the adelic description of the 
vertex set $\mathcal{V}(\mathcal{G}_v)$ above, which (via 
Jacquet-Langlands) allows us to view the specialization map 
as a map \begin{align*} D^{\times}\backslash \widehat{D}^{\times}/U &\longrightarrow 
\mathcal{O}_0/\mathfrak{P}_n \end{align*} having the 
same eigenvalues as ${\bf{f}}$. That is, it can be deduced
from the description above of the induced action of 
$G_{\mathfrak{p}^{\infty}}$ on this vertex set, along with
the canonical bijection $\eta_{\mathfrak{p}}$ coming 
from strong approximation at $\mathfrak{p}$, that 
\begin{align*} \partial_{\mathfrak{v}_m}\left( \sigma Q_m  \right) \equiv
\alpha_{\mathfrak{p}}^{-m} [\sigma, \mathfrak{e}_j]_{\Phi}
\mod \mathfrak{P}^n.\end{align*} Here, we have fixed a prime
$\mathfrak{v}_{\infty}$ above $v$ in $K_{\mathfrak{p}^{\infty}}$ and 
let $\mathfrak{v}_m = \mathfrak{v}_{\infty} \cap K_{\mathfrak{p}^m}$.
The result is now clear via the construction of 
$\mathcal{L}_{\bf{f}}$ from these elements. \end{proof}\end{remark}

\begin{remark}[The second explicit reciprocity law.]

Fix two $n$-admissible primes $v_1, v_2 \subset\mathcal{O}_F$ 
with respect to ${\bf{f}}$ such that \begin{align*}{\bf{N}}(v_i) + 1 - \varepsilon_i \cdot
a_{v_1}({\bf{f}}) \equiv 0 \mod \mathfrak{P}_n\end{align*} for each of
$i=1,2$. As usual, we keep all of the setup and hypotheses of Theorem
\ref{cerednikvarshavsky}, taking $v = v_1$ so that the indefinite 
quaternion algebra $B'$ has discriminant $v_1 \mathfrak{N}^{-}$. 

\begin{theorem}\label{ERL2}\emph{(The second explicit reciprocity law)}
Keep the hypotheses of Theorem \ref{raiseonefree} and Corollary \ref{galoisid}. 
Assume additionally that $F$ is linearly disjoint from the cyclotomic field ${\bf{Q}}(\zeta_p)$. Then,
the relation \begin{align*} \vartheta_{v_1}\left( \zeta(v_2) \right)=
\mathcal{L}_{{\bf{g}}}\end{align*} holds in
$\widehat{H}^1_{\fin}(K_{\mathfrak{p}^{\infty}, v_n},
T_{{\bf{f}},n}) \cong \Lambda/\mathfrak{P}^n$, up to
multiplication by elements of $\mathcal{O}^{\times}$ or
$G_{\mathfrak{p}^{\infty}}$.\end{theorem}
\begin{proof}
The proof is the same as that for Theorem \ref{ERL1}, replacing
${\bf{f}}$ with the mod $\mathfrak{P}_n$ eigenform ${\bf{g}}$
of Proposition \ref{gammaeigenform}. \end{proof}
\end{remark}

\begin{remark}[Acknowledgement.] It is a pleasure to thank the many people with
whom I discussed various aspects this work, in particular Kevin Buzzard, Allen
Cheng, Henri Darmon, Vladimir Dokchitser, Ben Howard, Chan-Ho Kim, Matteo 
Longo, Chung Pang Mok, Jan Nekovar, James Newton, Rob Pollack, Jay Pottharst 
and Tony Scholl. It is also a pleasure to thank the anonymous referees for their many 
helpful comments. \end{remark}

\end{document}